\documentclass[a4paper,11pt,reqno,noindent]{amsart}
\usepackage[scale=0.8]{geometry}

\usepackage{amsmath}
\usepackage{amssymb}
\usepackage{amsthm}
\usepackage{dsfont}
\usepackage{amsfonts}
\usepackage{cases}
\usepackage{url}
\theoremstyle{plain}
\newtheorem{theor}{Theorem}[section]
\newtheorem{lem}[theor]{Lemma}
\newtheorem{prop}[theor]{Proposition}
\newtheorem{cor}[theor]{Corollary}

\newtheorem{hyp}[theor]{Hypothesis}

\theoremstyle{definition}
\newtheorem{defin}[theor]{Definition}

\newtheorem{ex}[theor]{Example}

\newtheorem{rem}[theor]{Remark}

\usepackage[applemac]{inputenc}
\usepackage[polutonikogreek,english,frenchb]{babel}
\usepackage[T1]{fontenc}
\usepackage[toc,page]{appendix}
\usepackage{textcomp}
\usepackage{amscd}
\usepackage{graphicx}
\usepackage{enumerate}
\usepackage{cancel}
\usepackage{extpfeil}
\usepackage[all]{xy}
\usepackage{empheq}
\usepackage{esint}
\usepackage{mathtools}
\usepackage{float}
\usepackage{wrapfig}

\newcommand{\N}{\mathbb N}
\newcommand{\e}{\varepsilon}
\newcommand{\Q}{\mathbb Q}
\newcommand{\R}{\mathbb R}
\newcommand{\Z}{\mathbb Z}

\newcommand{\T}{\mathbb T}
\newcommand{\B}{\mathcal B}

\newcommand{\D}{\mathcal D}
\newcommand{\F}{\mathcal F}

\newcommand{\I}{\mathcal I}

\newcommand{\calM}{\mathcal M}

\newcommand{\Nc}{\mathcal N}

\newcommand{\Id}{\operatorname{Id}}
\newcommand{\p}{\mathbb{P}}
\newcommand{\E}{\mathbb{E}}

\newcommand{\dist}{\operatorname{dist}}

\newcommand{\supess}{\operatorname{sup\,ess}}
\newcommand{\supessd}[1]{\mathrel{\mathop{\supess}\limits_{#1}}}

\newcommand{\Mes}{\operatorname{Mes}}

\newcommand{\pot}{{\operatorname{pot}}}
\newcommand{\sol}{{\operatorname{sol}}}

\newcommand{\conv}{{\operatorname{conv}}}

\newcommand{\Zc}{\mathcal Z}
\newcommand{\loc}{{\operatorname{loc}}}

\newcommand{\sign}{{\operatorname{sign}}}

\newcommand{\per}{{\operatorname{per}}}
\newcommand{\adh}{{\operatorname{adh\,}}}
\newcommand{\dom}{{\operatorname{dom}}}

\newcommand{\Ld}{\operatorname{L}}

\newcommand{\supp}{\operatorname{supp}}

\newcommand{\inter}{\operatorname{int}}

\usepackage[toc,page]{appendix}

\newcommand{\cvf}[1]{\mathrel{\mathop{\xrightharpoonup{#1}}}}

\newcommand{\step}[1]{\noindent \textit{Step} #1.}
\numberwithin{equation}{section}

\usepackage{color}

\usepackage[colorlinks,citecolor=black,urlcolor=black]{hyperref}

\title[Stochastic homogenization of nonconvex unbounded integral functionals]{Stochastic homogenization of nonconvex unbounded integral functionals
with convex growth}
\author[M. Duerinckx]{Mitia Duerinckx}
\author[A. Gloria]{Antoine Gloria}
\date{\today}
\address[Mitia Duerinckx]{D\'epartement de math\'ematique, Universit\'e Libre de Bruxelles, Belgium \\  and MEPHYSTO team, Inria Lille - Nord Europe, Villeneuve d'Ascq, France}
\email{mduerinc@ulb.ac.be}
\address[Antoine Gloria]{D\'epartement de math\'ematique, Universit\'e Libre de Bruxelles, Belgium \\  and MEPHYSTO team, Inria Lille - Nord Europe, Villeneuve d'Ascq, France}
\email{agloria@ulb.ac.be}

\begin{document}
\selectlanguage{english}
\maketitle

We consider the well-travelled problem of homogenization of random integral functionals. When the integrand has standard growth conditions, the qualitative theory is well-understood. When it comes to unbounded functionals, that is, when the domain of the integrand is not the whole space and may depend on the space-variable, there is no satisfactory theory. In this contribution we develop a complete qualitative stochastic homogenization theory for nonconvex unbounded functionals with convex growth. We first prove that if the integrand is convex and has $p$-growth from below (with $p>d$, the dimension), then it admits homogenization regardless of growth conditions from above.  This result, that crucially relies on the existence and sublinearity at infinity of correctors, is also new in the periodic case. In the case of nonconvex integrands, we prove that a similar homogenization result holds provided the nonconvex integrand admits a two-sided estimate by a convex integrand  (the domain of which may depend on the space variable) that itself admits homogenization. This result is of interest to the rigorous derivation of rubber elasticity from polymer physics, which involves the stochastic homogenization of such unbounded functionals.

\tableofcontents

\section{Introduction}

Let $O$ be a bounded Lipschitz domain of $\R^d$, $d,m\geq 1$. We consider the well-travelled problem of homogenization of random integral functionals $I_\e:W^{1,p}(O;\R^m)\to [0,\infty]$ given by
$$
I_\e(u)\,=\,\int_O W\left(\frac{x}{\e},\nabla u(x)\right)dx,
$$
where $W$ is a random Borel function stationary in its first variable that satisfies  for almost every $y\in \R^d$ and all $\Lambda \in \R^{m\times d}$ the two-sided estimate
\begin{equation}\label{eq:intro-gc}
\frac{1}{C}|\Lambda|^p-C\leq  V(y,\Lambda) \,\leq \, W(y,\Lambda) \,\leq \,
C(1+V(y,\Lambda)),
\end{equation}
for some $C>0$, $p>1$, and a random convex Borel function $V:\R^d \times \R^{m\times d} \to [0,\infty]$.
The originality of the growth condition we consider here is that $V(y,\cdot)$ may take infinite values and that its domain may depend on $y$, so that the domain of the homogenized integrand $\overline W$ (if it exists) is unknown a priori.
The motivation for considering such a problem comes from the derivation of nonlinear elasticity from the statistical physics of polymer-chain networks, cf. \cite{Alicandro-Cicalese-Gloria-11,Gloria-LeTallec-Vidrascu-08b,DeBuhan-Gloria-LeTallec-Vidrascu-10}. Indeed, the free energy of the polymer-chain network is given by two contributions: a steric effect (for which proving homogenization is one of the most important open problems of the field), and the sum of free energies of the deformed chains. The free energy of a single chain is a convex increasing function of the square of the length of the deformed polymer-chain, which blows up at finite deformation. 
The corresponding problem in a continuum setting would be the homogenization of the nonconvex integrand
\begin{equation}\label{eq:motivation}
W(y,\Lambda)\,=\,V(y,\Lambda)+g(\det \Lambda) \,\leq\, C(1+V(y,\Lambda)),
\end{equation}
where $V$ is an \emph{infinite-valued convex stationary ergodic integrand} whose domain depends on the space variable, and $g$ is a convex function (in this paper we assume that $g$ is controlled by $V$, which unfortunately rules out the finite compressibility of matter).

\medskip

Homogenization of multiple integrals has a long history, and we start with the state of the art when $V$ and $W$ are periodic in the first variable:
\begin{enumerate}[(i)]
\item The first contribution (beyond the linear case) is due to Marcellini \cite{Marcellini-78}, who addressed the homogenization of \emph{convex periodic} integrands satisfying a \emph{polynomial standard growth condition}, that is \eqref{eq:intro-gc} for $V(y,\Lambda)=|\Lambda|^p$.
\item Marcellini's result was then generalized to  \emph{nonconvex periodic} integrands satisfying 
a \emph{polynomial standard growth condition}, by Braides \cite{Braides-85} (which covers in addition almost-periodic coefficients) and M\"uller \cite[Theorem~1.3]{Muller-87}, independently.
\item In \cite[Theorem~1.5]{Muller-87}, M\"uller also addressed the case of a \emph{convex periodic} integrand satisfying a \emph{convex standard growth condition} \eqref{eq:intro-gc} for $V(y,\Lambda)=\tilde V(\Lambda)$ with $\tilde V:\R^{m\times d}\to \R^+$ a convex finite-valued map, and $p>d$.
\item In \cite[Chapter~21]{Braides98}, Braides and Defranceschi  treated the case of \emph{nonconvex periodic} integrands (see also~\cite{CorboEsposito-DeArcangelis-92} in the convex case)
satisfying~\eqref{eq:intro-gc} where $V$ is \emph{convex periodic} and satisfies the \emph{polynomial non-standard growth condition}

\begin{equation*}
\frac{1}{C}|\Lambda|^p-C\,\leq\, V(y,\Lambda)\,\leq \, C(1+|\Lambda|^q)
\end{equation*}
for some $q<p^*$ (with $p^*$ the Sobolev-conjugate of $p>1$), and the \emph{doubling property}
\begin{equation*}
V(y,2\Lambda) \,\leq \, C(1+V(y,\Lambda)).
\end{equation*}
\item In~\cite{Braides-Garroni-95}, given a collection of well-separated periodic inclusions, Braides and Garroni treated the case of {\it nonconvex periodic} integrands satisfying a {\it polynomial standard growth condition} as well as the (strong) {\it doubling property}
\begin{equation*}
W(y,2\Lambda) \,\leq \, CW(y,\Lambda).
\end{equation*}
outside the inclusions, but only satisfying inside the inclusions a \emph{convex standard growth condition}~\eqref{eq:intro-gc} for some possibly unbounded map $V(y,\Lambda)$ that is convex in the $\Lambda$-variable.
\item More recently Anza Hafsa and Mandallena studied in \cite{Anza-Mandallena-11} the homogenization of \emph{quasiconvex periodic} integrands satisfying a \emph{standard (unbounded) convex growth condition}, that is, \eqref{eq:intro-gc} for $V(y,\Lambda)=\tilde V(\Lambda)$ with $\tilde V:\R^{m\times d}\to [0,\infty]$ a convex infinite-valued map such that $\tilde V(\Lambda)\geq |\Lambda|^p$, and with $p>d$. Note that in this case the domain is fixed.
\end{enumerate}
When $W$ is random, the results are more sparse:
\begin{enumerate}[(vii)]
\item The first contribution beyond the linear case is due to Dal Maso and Modica, who addressed the homogenization of \emph{convex random stationary} integrands satisfying a \emph{polynomial standard growth condition} \cite{DalMaso-Modica-86}, generalizing Marcellini's result~(ii) to the random setting.
\item[(viii)] Messaoudi and Michaille later treated the homogenization of \emph{quasiconvex stationary ergodic} integrands satisfying a \emph{polynomial standard growth condition} \cite{Messaoudi-Michaille-94}, following Dal Maso and Modica's approach.
\item[(ix)] In their monograph on homogenization, Jikov, Kozlov and Oleinik treated the case of \emph{convex stationary ergodic} integrands satisfying a \emph{polynomial non-standard growth condition} 
\begin{equation*}
\frac{1}{C}|\Lambda|^p-C\,\leq\, W(y,\Lambda)\,\leq \, C(1+|\Lambda|^q)
\end{equation*}
for some $q<p^*$ (with $p^*$ the Sobolev-conjugate of $p>1$), see \cite[Chapter~15]{JKO94}. 
\end{enumerate}
For scalar functionals, that is, when $m=1$, results are much more precise, and we refer the reader to the monograph \cite{Carbone-DeArcangelis-02} by Carbone and De Arcangelis (which is however only concerned with the periodic setting)
and \cite{CCDAG-02,CCDAG-04}. 
When the domain of $V(y,\cdot)$ is not the whole of $\R^{d}$, the authors call $I_\e$ an unbounded functional.
The main technical tool for {\it scalar} unbounded functionals is truncation of test-functions (see also \cite[Section~15.2]{JKO94}), which cannot be used for systems in general (see however the end of this introduction). In particular, such truncation arguments replace the Sobolev embedding we shall use for systems and allow one to relax the assumption $p>d$ for scalar problems.

\medskip

In this contribution we give a far-reaching generalization of (i)--(iv) and (vi)--(ix) for systems in the random setting, by relaxing the assumption that the domain of $V(y,\cdot)$ in \eqref{eq:intro-gc} is independent of $y$. Our contribution also generalizes (v) by relaxing all geometric assumptions. We argue in two steps. 
For convex integrands, our result shows that homogenization holds without any growth condition from above (cf. Theorem~\ref{th:conv} for Neumann boundary conditions, and Corollary~\ref{cor:dir} for the more subtle case of Dirichlet boundary conditions), so that we may homogenize the bound $V$ itself. 
We proceed by truncation of the energy density (following the approach by M\"uller in~\cite{Muller-87}), and first prove in Proposition~\ref{prop:commut} that homogenization and truncation commute at the level of the definition of the homogenized energy density.
The proof relies on the existence of correctors with stationary gradients for convex problems and exploits quantitatively their sublinearity at infinity, see Lemma~\ref{lem:sublincor} (which is a substantial difference between the periodic and random cases, and makes the latter more subtle).
The second main technical achievement is the construction of recovery sequences in Proposition~\ref{prop:gammasupN} (a gluing argument based on affine boundary data trivially fails since the domain of $V(y,\cdot)$ depends on~$y$ --- this difficulty is already present in the periodic setting, and prevents us from using the standard homogenization formula with Dirichlet boundary conditions).
In the case of nonconvex integrands with a two-sided convex estimate \eqref{eq:intro-gc}, we show
in Theorem~\ref{th:nonconv} that homogenization reduces to the homogenization of the convex bound $V$. The first obstacle in this program is the definition of the homogenized energy density itself.
Indeed, in the absence of correctors (which is a consequence of nonconvexity), one usually defines the homogenized energy density through an asymptotic limit with linear boundary data on increasing cubes. As pointed out above, such an approach fails in general when the domain of the integrand is not fixed. 
Instead, in  Lemma~\ref{lem:Whomnc}, we use the (well-defined) corrector of the associated \emph{convex} problem as boundary data for the \emph{nonconvex} problem on these increasing cubes. Next we argue by blow-up in Proposition~\ref{prop:gammaliminfnc} for the $\Gamma$-$\liminf$ inequality (following the approach by Fonseca and M\"uller in \cite{Fonseca-Muller-92}), and make a systematic use of the corrector of the convex problem to control the nonconvex energy from above. Then, for the $\Gamma$-$\limsup$ we argue  in Proposition~\ref{prop:gammalimsupnc} by relaxation (following the approach introduced by Fonseca in \cite{Fonseca-88} for relaxation and first used in homogenization by Anza Hafsa and Mandallena in \cite{Anza-Mandallena-11}), making a similar use of the corrector of the convex problem in the estimates.

\medskip

To conclude this introduction, let us go back to our motivation, that is, the homogenization of \eqref{eq:motivation}.
On the one hand, we have reduced the homogenization for such integrands to the homogenization for the convex integrand $V$. On the other hand, we have proved homogenization 
for convex integrands without growth condition from above, and therefore proved homogenization for \eqref{eq:motivation}.
In the specific setting of \eqref{eq:motivation},
we can sharpen the general results described above, by simplifying the definition of the homogenized energy density $\overline W$, cf. Corollary~\ref{cor:model-nc}.
We believe a similar approach can be successfully implemented in the discrete setting considered in \cite{Alicandro-Cicalese-Gloria-11}
for the derivation of nonlinear elasticity from polymer physics.

In the scalar case $m=1$, combining our approach with truncation arguments, we may further refine our general results, in particular relaxing the condition $p>d$. Our approach then improves (and extends to the stochastic setting) some scalar results of \cite{CCDAG-02,Carbone-DeArcangelis-02}.

\medskip

The main results are given in Section~\ref{sec:results}. The proofs of the results for convex integrands are displayed in Section~\ref{sec:convex}, whereas Section~\ref{sec:nonconv} is dedicated to the proofs for nonconvex integrands. In Section~\ref{chap:improved} we turn to various possible improvements of our general results under additional assumptions.
In the appendix we prove several standard and less standard results on approximation of functions, measurability of integral functionals, and Weyl decompositions in probability, that are abundantly used in this article.

\section{Main results}\label{sec:results}

Let $(\Omega,\F,\p)$ be a complete probability space and let  $\tau:=(\tau_y)_{y\in\R^d}$  be a measurable {\it ergodic} action of $(\R^d,+)$ on $(\Omega,\F,\p)$, that are fixed once and for all throughout the paper. We denote by $\E$ the expectation on $\Omega$ with respect to $\p$.

Consider a map $W:\R^d\times\R^{m\times d}\times\Omega\to[0,\infty]$ that is $\tau$-stationary in the sense that, for all $\Lambda\in\R^{m\times d}$, all $\omega\in\Omega$, and all $y,z\in\R^d$,
\begin{align}\label{eq:statnormalrandomint}
W(y,\Lambda,\tau_{-z}\omega)=W(y+z,\Lambda,\omega),
\end{align}
and assume that $W(y,\cdot,\omega)$ is lower semicontinuous on $\R^{m\times d}$ for almost all $y,\omega$.
We also assume in the rest of this paper that, for almost all $\omega$, the map $y\mapsto W(y,\Lambda+u(y),\omega)$ is measurable for all $u\in\Mes(\R^d;\R^{m\times d})$, and that, for almost all $y\in\R^d$, the map $\omega\mapsto W(y,\Lambda+v(\omega),\omega)$ is measurable for all $v\in\Mes(\Omega;\R^{m\times d})$.
Continuity in the second variable and joint measurability (in which case $W$ is called a {\it Carath\'eodory} integrand) would ensure these properties; weaker sufficient conditions for these properties are given in Appendix~\ref{app:intnormal}.
Such integrands $W$ will be called {\it $\tau$-stationary normal random integrands}.

We further make the following additional measurability assumption on $W$:
\begin{hyp}\label{hypo:asmeasinf}
For any jointly measurable function $f:\R^d\times \Omega\to\R$ and any bounded domain $O\subset\R^d$,
\begin{align*}
\omega\mapsto\inf_{u\in W^{1,1}_0(O;\R^m)}\int_OW(y/\e,f(y,\omega)+\nabla u(y),\omega)dy
\end{align*}
is $\F$-measurable on $\Omega$.\qed
\end{hyp}

As discussed in Appendix~\ref{app:infmes}, this last hypothesis is always satisfied if $W$ is convex in the second variable, and more generally if it is sup-quasiconvex (in the sense of Definition~\ref{def:approxquasiconv}), which includes e.g. the case of a sum of a convex integrand and of a ``nice'' nonconvex part.

For any bounded domain $O\subset\R^d$, we define the following family of random integral functionals, parametrized by $\e>0$,
\begin{align}\label{eq:defIeps}
I_\e(\cdot,\cdot;O):W^{1,1}(O;\R^m)\times\Omega\to [0,\infty]: \quad(u,\omega)\mapsto I_\e(u,\omega;O):=\int_OW(y/\e,\nabla u(y),\omega)dy.
\end{align}
The aim of this paper is to prove homogenization for $I_\e$ as $\e \downarrow 0$ under mild growth conditions on $W$, which we formulate in terms of $\Gamma$-convergence for the weak convergence of $W^{1,p}(O;\R^m)$ (for some $p>1$).
When $\Lambda \mapsto W(y,\Lambda,\omega)$ is convex for almost all $y,\omega$, we say that $W$ is a {\it $\tau$-stationary convex normal random integrand}, and shall use the notation $V$ and $J_\e$ instead of $W$ and $I_\e$, that is, for every bounded domain $O\subset\R^d$ and $\e>0$,
\begin{align}\label{eq:defJeps}
J_\e(\cdot,\cdot;O):W^{1,1}(O,\R^m)\times\Omega\to [0,\infty]:\quad  (u,\omega)\mapsto J_\e(u,\omega;O):=\int_OV(y/\e,\nabla u(y),\omega)dy.
\end{align}
The notation $W$ and $I_\e$ will be used for nonconvex integrands.
We start our analysis with the case of convex integrands, then turn to nonconvex integrands, present an application to nonlinear elasticity, and conclude with a discussion of several possible improvements of these general results under additional assumptions.

\subsection{Convex integrands}

In this subsection we state homogenization results for $J_\e$ with (essentially) no growth condition from above.
We start with Neumann boundary conditions, and then address the more subtle case of Dirichlet boundary conditions.

\subsubsection{Homogenization with Neumann boundary conditions}

Our first result is as follows.
\begin{theor}[Convex integrands with Neumann boundary data]\label{th:conv}
Let $V:\R^d\times\R^{m\times d}\times\Omega\to [0,\infty]$ be a $\tau$-stationary convex normal random integrand that satisfies the following
uniform coercivity condition: there exist $C>0$ and $p>d$ such that for almost all $\omega$ and $y$, we have for all $\Lambda$,
\begin{align}\label{eq:aslowerbound0}
\frac{1}{C}|\Lambda|^{p}-C\,\le\, V(y,\Lambda,\omega).
\end{align}
Assume that the convex function $M:=\supess_{y,\omega}V(y,\cdot,\omega)$ has $0$ in the interior of its domain.
Then, for almost all $\omega\in\Omega$ and all  bounded Lipschitz domains $O\subset\R^d$, the integral functionals $J_\e(\cdot,\omega;O)$ $\Gamma$-converge
to the integral functional $J(\cdot;O):W^{1,p}(O;\R^m)\to[0,\infty]$ defined by
\begin{align*}
J(u;O)=\int_O\overline V(\nabla u(y))dy,
\end{align*}
for some lower semicontinuous convex function $\overline V:\R^{m\times d}\to [0,\infty]$ characterized by the following three equivalent formulas:
\begin{enumerate}[(i)]
\item Formula in probability: for all $\Lambda\in \R^{m\times d}$,
\begin{align}\label{eq:homogformproba}
\overline V(\Lambda)=\inf_{g\in F^p_\pot(\Omega)^m}\E[V(0,\Lambda+g,\cdot)],
\end{align}
where the space of mean-zero potential random variables $F^p_\pot(\Omega)$ is recalled in Section~\ref{chap:preliminweyl}.
\item Dirichlet formula with truncation: for any increasing sequence $V^k\uparrow V$ of $\tau$-stationary convex random integrands 
that satisfy the standard $p$-growth condition
\begin{equation}\label{eq:coercivity-conv}
\frac{1}{C}|\Lambda|^{p}-C\le V^k(y,\Lambda,\omega)\le C_k(1+|\Lambda|^p)
\end{equation}
for all $y,\omega,\Lambda$, and some sequence $C_k<\infty$, 
we have for almost all $\omega$, all $\Lambda$, and all bounded Lipschitz domains $O\subset\R^d$,
\begin{align}\label{eq:homogformmuller}
\overline{V}(\Lambda)=\lim_{k\uparrow\infty} \lim_{\e\downarrow 0} \inf_{\phi\in W^{1,p}_0(O/\e;\R^m)}\fint_{O/\e} V^k(y,\Lambda+\nabla \phi(y),\omega)dy.
\end{align}
\item Convexification formula: for all $\Lambda$ and all bounded Lipschitz domains $O\subset\R^d$, we have, for almost all $\omega$,
\begin{align}\label{eq:formconvexif}
\overline{V}(\Lambda)=\lim_{t\uparrow1}\lim_{\e\downarrow 0} \inf_{\phi\in W^{1,p}(O/\e;\R^m)\atop \fint_{O/\e}\nabla\phi=0}\fint_{O/\e} V(y,t\Lambda+\nabla \phi(y),\omega)dy.
\end{align}
As a consequence of convexity, the limit $t\uparrow1$ can be omitted when $\Lambda\notin \partial\dom\overline V$.\qed
\end{enumerate}
\end{theor}

Comments are in order:
\begin{itemize}
\item The limit $t\uparrow1$ cannot be omitted in \eqref{eq:formconvexif} in general for $\Lambda\in\partial\dom\overline V$.
Indeed, let $V$ coincide with a convex map $\tilde V:\R^{m\times d}\to [0,\infty]$ with a closed domain, and which is not lower semicontinuous at the boundary of its domain.
In the interior of $\dom \tilde V$, $\overline V$ coincides with $\tilde V$. However, since $\overline V$ is necessarily lower semicontinuous on its domain, it cannot 
coincide with $\tilde V$ on $\partial \dom \tilde V$.
\item In the proof we take~\eqref{eq:homogformmuller} as the defining formula for $\overline V$, following the approach by M\"uller in~\cite{Muller-87}. Formula~\eqref{eq:homogformproba} is interesting in two respects: first, it is intrinsinc (no approximation is required), and second it is an exact formula (there is no asymptotic limit involved). The equivalence of both formulas, which can be interpreted as the commutation of trunction and homogenization, is the key to the proof of the $\Gamma$-convergence result. 
\end{itemize}

\medskip

We may extend Theorem~\ref{th:conv} in two directions:
\begin{itemize}
\item The extension of Theorem~\ref{th:conv} to the case of domains with holes (or more generally to soft inclusions, for which the coercivity assumption \eqref{eq:aslowerbound0} does not hold everywhere) is straighforward provided we have a suitable extension result.
When holes are well-separated, such extension results are standard (see e.g.~\cite[Sections~3.1]{JKO94}). 
For general situations however, this can become a subtle issue (see in particular~\cite[Sections~3.1 and~3.5]{JKO94}). 
In the particular case of the periodic setting, there is a very general extension result \cite{ACPDMP-92}, which is used e.g. in~\cite{Braides-Garroni-95}. 
\item The assumption $p>d$ is crucial in the generality of Theorem~\ref{th:conv}, which is used in the form of the Sobolev embedding of $W^{1,p}(O;\R^m)$ in $\Ld^\infty(O;\R^m)$.
In the case $1<p\le d$, 
the conclusions of the theorem still hold true provided that $V(y,\Lambda,\omega)\le M(\Lambda)$ for some convex function $M:\R^{m\times d}\to\R$ that satisfies
the growth condition $\lim_{|\Lambda|\to\infty}M(\Lambda)/|\Lambda|^q<\infty$ for $q=dp/(d-p)$ if $p<d$ or for some $q<\infty$ if $p=d$. 
In the scalar case $m=1$, the use of the Sobolev embedding can be avoided by a truncation argument, as explained in Corollary~\ref{cor:subcritp} below, see also \cite{Carbone-DeArcangelis-02,CCDAG-02}.
\end{itemize}

\subsubsection{Dirichlet boundary conditions}\label{chap:generalBC}
We now discuss the homogenization result in the case of Dirichlet boundary conditions (the case of mixed boundary data can then be dealt with in a straightforward way, and we leave the details to the reader). A first remark is that Dirichlet data have to be well-prepared, as the following elementary example shows.

\begin{ex}{}\label{example}
Consider random unit spherical inclusions centered at the points of a Poisson point process, choose the integrand $V$ to be equal to $|\Lambda|^p$ outside the inclusions and to have a bounded domain $\mathcal D\subset\R^{m\times d}$ inside the inclusions. Given a (nonempty) bounded open set $O\subset\R^d$, for almost all $\omega$, the realization of the inclusions corresponding to $\omega$ intersects $\partial (O/\e)$ for infinitely many $\e>0$, and hence $\limsup_\e J_\e(u+\Lambda\cdot x,\omega;O)=\infty$ for all $u\in W^{1,p}_0(O;\R^m)$, due to the Dirichlet boundary condition. In contrast, if the intensity of the underlying Poisson process is not too big, it is easily seen that the homogenized integral functional $J$ defined in Theorem~\ref{th:conv} is finite-valued. This proves that, for all $\Lambda\notin\mathcal D$ and almost all $\omega$, $J_\e(\cdot+\Lambda\cdot x,\omega;O)$ cannot $\Gamma$-converge to $J(\cdot;O)$ on $W^{1,p}_0(O;\R^m)$, due to the intersection of some rigid inclusions with the boundary of the domain where the Dirichlet condition is imposed.
\qed\end{ex}

We propose two ways to prepare Dirichlet data:
\begin{itemize}
\item by relaxing the boundary data so that the energy remains finite for all $\e>0$ while ensuring that the boundary data are recovered at the limit $\e \downarrow 0$ --- we call ``lifting'' this procedure;
\item by replacing the integrand $V$ by a softer integrand on a neighborhood of the boundary where the boundary condition is imposed --- we call this a ``soft buffer zone''. 
\end{itemize}

\begin{cor}[Convex integrands with Dirichlet boundary data]\label{cor:dir}
Let $V$, $M$, $J_\e$, and $J$ be as in Theorem~\ref{th:conv} for some $p>d$.  Then, for almost all $\omega\in\Omega$ and all  bounded Lipschitz domains $O\subset\R^d$, we have
\begin{enumerate}[(i)]
\item For all boundary data $u\in W^{1,p}(O;\R^m)$ such that $J(\alpha u;O)<\infty$ for some $\alpha >1$, there exists a lifted sequence $(u_\e)_\e$ with $u_\e\cvf{} u$ in $W^{1,p}(O;\R^m)$, such that 
we have on $W^{1,p}_0(O;\R^m)$:
\begin{eqnarray*}
J(\cdot+u;O)&=&\Gamma\text{-}\lim_{t\uparrow1}~\Gamma\text{-}\liminf_{\e\downarrow0}~J_\e(\cdot+tu_\e,\omega;O)\\
&=&\Gamma\text{-}\lim_{t\uparrow1}~\Gamma\text{-}\limsup_{\e\downarrow0}~J_\e(\cdot+tu_\e,\omega;O).
\end{eqnarray*}
In particular, 
\begin{eqnarray*}
\inf_{v\in W^{1,p}_0(O)} J(v+u;O)&=&\lim_{t\uparrow 1}\liminf_{\e \downarrow 0} \inf_{v\in W^{1,p}_0(O)}J_\e(v+tu_\e;O) \\
&=&\lim_{t\uparrow 1}\limsup_{\e \downarrow 0} \inf_{v\in W^{1,p}_0(O)}J_\e(v+tu_\e;O).
\end{eqnarray*}
If $u$ satisfies the additional condition $\int_OM(\nabla u(y))dy<\infty$, then we may take $u_\e \equiv u$, and if this condition is strengthened to 
$\int_OM(\alpha \nabla u(y))dy<\infty$ for some $\alpha>1$ then the limit $t\uparrow1$ can be omitted.
\item For all boundary data $u\in W^{1,p}(O;\R^m)$ such that $J(u;O)<\infty$, we have on $W^{1,p}_0(O;\R^m)$:
\begin{eqnarray*}
J(\cdot+u;O)&=&\Gamma\text{-}\lim_{t\uparrow1, \eta\downarrow0}~\Gamma\text{-}\liminf_{\e\downarrow0}~J_\e^\eta(\cdot+tu,\omega;O)\\
&=&\Gamma\text{-}\lim_{t\uparrow1, \eta\downarrow0}~\Gamma\text{-}\limsup_{\e\downarrow0}~J_\e^\eta(\cdot+tu,\omega;O),
\end{eqnarray*}
where $J_\e^\eta$ is the following modification of $J_\e$ on an $\eta$-neighborhood of $\partial O$: 
\begin{align}\label{eq:defVepsOeta}
J_\e^\eta(v,\omega;O):=~&\int_OV_\e^{O,\eta}(y,\nabla v(y),\omega)dy,\nonumber\\
V_\e^{O,\eta}(y,\Lambda,\omega):=~&\begin{cases}V(y/\e,\Lambda,\omega),&\text{if $\dist(y,\partial O)>\eta$};\\ |\Lambda|^p,&\text{if $\dist(y,\partial O)<\eta$.}\end{cases}
\end{align}
In particular, 
\begin{eqnarray*}
\inf_{v\in W^{1,p}_0(O)} J(v+u;O)&=&\lim_{t\uparrow 1,\eta \downarrow 0}\liminf_{\e \downarrow 0} \inf_{v\in W^{1,p}_0(O)}J_\e^\eta(v+tu;O) \\
&=&\lim_{t\uparrow 1,\eta \downarrow 0}\limsup_{\e \downarrow 0} \inf_{v\in W^{1,p}_0(O)}J_\e^\eta(v+tu;O).
\end{eqnarray*}
If $u$ satisfies the additional condition $J(\alpha u;O)<\infty$ for some $\alpha>1$, then the limit $t\uparrow1$ can be omitted.\qed
\end{enumerate}
\end{cor}

Comments are in order:
\begin{itemize}
\item The results of Corollary~\ref{cor:dir} are not completely satisfactory.
Indeed, if one makes a diagonal extraction of $t$ and $\eta$ wrt $\e$ to obtain a $\Gamma$-convergence in $\e$ only, then the extraction for the $\Gamma$-limsup depends on the target function $v+u$ and not only on the boundary data $u$ as we would hope for.  
This dependence is however restricted to a dilation parameter only in the case of the lifting.
In the case of the buffer zone, the result can be (optimally) improved to $\eta_\e=\theta\e$ for any fixed $\theta>0$, under some additional structural assumption in the form of the existence of stationary quasi-correctors (see Proposition~\ref{prop:statcorbuffer}). For specific examples for which this assumption holds, see Corollary~\ref{cor:dir-bis} below.
\item In the specific situation when $\dom \overline V=\dom M$ (this is trivially the case when the domain is fixed, i.e. $\dom V(y,\cdot,\omega)=\dom M$ for almost all $y,\omega$),
then the strong assumptions $\int_OM(\nabla u(y))dy<\infty$ and $\int_OM(\alpha \nabla u(y))dy<\infty$ (for some $\alpha>1$) in part~(i) of the statement reduce to the simpler assumptions $J(u;O)<\infty$ and $J(\alpha u;O)<\infty$ (for some $\alpha>1$), respectively. In particular, in that situation, the lifting can always be chosen to be trivial, $u_\e:=u$ for all $\e$.
\item In \cite{Braides-Garroni-95} (see also \cite[Chapter 20]{Braides98}), Braides and Garroni prepare the boundary data in a different way in the specific case of stiff inclusions. In particular they introduce an operator $R^\e$ which acts on functions $u$ as follows: on each stiff inclusion $R^\e(u)$ has value the average of $u$ on the considered stiff inclusion, away from all inclusions $R^\e(u)$ coincides with $u$, and in between $R^\e(u)$ is an interpolation between $u$ and the average of $u$ on the inclusion. Such a construction can be used here as well, but seem to admit no natural generalization in other settings than stiff inclusions.
\end{itemize}

\subsection{Nonconvex integrands with convex growth}
In the case when $W$ is nonconvex and admits a two-sided estimate by a convex function (which may depend on the space variable), we show that a  $\Gamma$-convergence result similar to the convex case holds. Before we precisely state this result, let us recall the notion of radial uniform upper semicontinuity (which is trivially satisfied by convex maps).

\begin{defin}\label{def:ruusc}
A map $Z:\R^{m\times d}\to[0,\infty]$ is said to be {\it ru-usc} (i.e. {\it radially uniformly upper semicontinuous}) if there is some $\alpha\ge0$ such that the function 
\begin{align*}
\Delta^\alpha_Z(t)=\sup_{\Lambda\in\dom Z}\frac{Z(t\Lambda)-Z(\Lambda)}{\alpha+Z(\Lambda)}
\end{align*}
satisfies $\limsup_{t\uparrow1}\Delta^\alpha_Z(t)\le0$.
A $\tau$-stationary normal random integrand $W$ is said to be {\it ru-usc}  if there exists a $\tau$-stationary integrable random field $a:\R^d\times\Omega\to[0,\infty]$ such that the  function
\begin{align*}
\Delta^a_W(t):=\supessd{y\in\R^d}\quad\supessd{\omega\in\Omega}\sup_{\Lambda\in\dom W(y,\cdot,\omega)}\frac{W(y,t\Lambda,\omega)-W(y,\Lambda,\omega)}{a(y,\omega)+W(y,\Lambda,\omega)}
\end{align*}
satisfies $\limsup_{t\uparrow1}\Delta^a_W(t)\le0$.\qed
\end{defin}

The following result is a far-reaching generalization of \cite[Theorem~1.5]{Muller-87} to a wide class of \emph{random and nonconvex} integrands; it is also a substantial extension of \cite[Corollary~2.2]{Anza-Mandallena-11}.

\begin{theor}[Nonconvex integrands with convex growth]\label{th:nonconv}
Let $W:\R^d\times\R^{m\times d}\times\Omega\to [0,\infty]$ be a (nonconvex) ru-usc $\tau$-stationary normal random integrand satisfying Hypothesis~\ref{hypo:asmeasinf}. Assume that, for almost all $\omega$, $y$, and for all $\Lambda$,
\begin{align}\label{eq:convhomb}
V(y,\Lambda,\omega)\,\le\, W(y,\Lambda,\omega)\,\le\, C(1+V(y,\Lambda,\omega)),
\end{align}
for some $C>0$ and some $\tau$-stationary convex random integrand $V:\R^d\times\R^{m\times d}\times\Omega\to [0,\infty]$
that satisfies the assumptions of Theorem~\ref{th:conv} for some $p>d$.
Then, for almost all $\omega\in\Omega$ and all  bounded Lipschitz domains $O\subset\R^d$, the integral functionals $I_\e(\cdot,\omega;O)$ $\Gamma$-converge
to the integral functional $I(\cdot;O):W^{1,p}(O;\R^m)\to[0,\infty]$ defined by
\begin{align*}
I(u;O)=\int_O\overline W(\nabla u(y))dy,
\end{align*}
for some ru-usc lower semicontinuous quasiconvex function $\overline W:\R^{m\times d}\to [0,\infty]$ that satisfies $\overline V\le\overline W\le C(1+\overline V)$,
where $\overline V$ is the homogenized integrand associated with $V$ by Theorem~\ref{th:conv}. In addition, the results stated in Corollary~\ref{cor:dir} for $J_\e$ also hold for $I_\e$.

For all $\Lambda\in\R^{m\times d}$, let $g_\Lambda$ be the potential field in probability minimizing $\E[V(0,\Lambda+\cdot)]$ (cf. \eqref{eq:homogformproba}),
and note that $x\mapsto g_\Lambda(\tau_x \omega)$ is a gradient field on $\R^d$ for almost all $\omega \in \Omega$, which we denote by $\nabla \varphi_\Lambda(x,\omega)$. The homogenized integrand $\overline W$ is characterized for all $\Lambda \in \R^{m\times d}$ by
\begin{align}\label{eq:homogformnc}
\overline W(\Lambda)=\liminf_{t\uparrow1}\liminf_{\Lambda'\to\Lambda}\lim_{\e \downarrow 0} \inf_{v\in W^{1,p}_0(O/\e;\R^m)}\fint_{O/\e} W(y,t\Lambda'+t\nabla\varphi_{\Lambda'}(y,\omega)+\nabla v(y),\omega)dy.
\end{align}
for any bounded Lipschitz domain $O\subset\R^d$ and almost every $\omega \in \Omega$.\qed
\end{theor}

In general, we do not expect that the limits $t\uparrow1$ and $\Lambda'\to\Lambda$ can be dropped in~\eqref{eq:homogformnc}, see however Corollary~\ref{cor:model-nc} below under Hypothesis~\ref{hypo:model-nc}.

\subsection{Application to nonlinear elasticity}\label{subsec:nonlinear}

In the example from the statistical physics of polymer-chain networks, the integrand has the specific decomposition \eqref{eq:motivation}. Moreover, the nonconvex part of the integrand satisfies the following assumption, which in particular implies that $W$ satisfies Hypothesis~\ref{hypo:asmeasinf} (see indeed Lemma~\ref{lem:psupquasiconv}), as well as the ru-usc property.
\begin{hyp}\label{hypo:model-nc}
There exists a $\tau$-stationary convex map $V:\R^d\times\R^{m\times d}\times\Omega\to [0,\infty]$ and some $p>1$ 
such that 
$$
W(y,\Lambda,\omega)=V(y,\Lambda,\omega)+W^{nc}(y,\Lambda,\omega),
$$
where $W^{nc}:\R^d\times\R^{m\times d}\times\Omega\to[0,\infty]$ is a (nonconvex) $\tau$-stationary normal random integrand satisfying the $p$-th order upper bound and the local Lipschitz conditions, for some $C>0$,
\begin{enumerate}[(i)]
\item for almost all $y,\omega$, for all $\Lambda$,
\[W^{nc}(y,\Lambda,\omega)\,\le\,C(|\Lambda|^{p}+1);\]
\item for almost all $y,\omega$, for all $\Lambda,\Lambda'$,
\[|W^{nc}(y,\Lambda,\omega)-W^{nc}(y,\Lambda',\omega)|\le C(1+|\Lambda|^{p-1}+|\Lambda'|^{p-1})|\Lambda-\Lambda'|.\]\qed
\end{enumerate}
\end{hyp}
Under this assumption, the following variant of Theorem~\ref{th:nonconv} holds and yields in particular a simpler formula for the homogenized energy density.
\begin{cor}\label{cor:model-nc}
Let $W:\R^d\times\R^{m\times d}\times\Omega\to [0,\infty]$ be a (nonconvex) $\tau$-stationary normal random integrand satisfying Hypothesis~\ref{hypo:model-nc}. Assume that, for almost all $\omega$, $y$, and for all $\Lambda$,
\begin{align}\label{eq:convhomb2}
V(y,\Lambda,\omega)\,\le\, W(y,\Lambda,\omega)\,\le\, C(1+V(y,\Lambda,\omega)),
\end{align}
for some $C>0$, where $V:\R^d\times\R^{m\times d}\times\Omega\to [0,\infty]$ satisfies the assumptions of Theorem~\ref{th:conv} for some $p>d$.
Then the conclusions of Theorem~\ref{th:nonconv} hold true.
In addition we have for all $\Lambda\in \R^{m\times d}$ for almost all $\omega$
\begin{align}\label{eq:homogformnc-bis}
\overline W(\Lambda)=\lim_{R\uparrow \infty} \inf_{v\in W^{1,p}_0(Q_R;\R^m)}\fint_{Q_R} W(y,\Lambda+\nabla\varphi_{\Lambda}(y,\omega)+\nabla v(y),\omega)dy,
\end{align}
where $\varphi_\Lambda$ is as in Theorem~\ref{th:nonconv}.
\qed
\end{cor}

The behavior of $\overline W$ close to the boundary of its domain is of particular interest.
From a mechanical point of view, in the case of \eqref{eq:motivation} with $V$ satisfying $V(y,\Lambda_n)\to \infty$ as $\dist(\Lambda_n, \partial \text{dom}V(y,\cdot))\to  0$, we expect that $\overline W(\Lambda_n)\to \infty$ as $\dist(\Lambda_n, \partial \text{dom}\overline W)\to 0$.
Except in a few explicit examples, we do not know under which condition and how to prove such a property.
In the case of the geometry that saturates the Hashin-Shtrikman bound (obtained by a stationary and statistically isotropic Vitali covering of $\R^d$ by balls, cf. \cite[Section~6.2]{JKO94}),
where each ball is the homothetic image of a unit ball with a given spherical inclusion, the corrector gradient field is explicit when $\Lambda$ is a dilation,
and $t\mapsto \overline{W}(t\Id)$ indeed blows up when $t$ approaches the boundary of the domain.

\subsection{Some improved results}
The general results above naturally call for various questions concerning possible improvements:
\begin{itemize}
\item What about the subcritical case $1<p\le d$?
\item What is the minimal size $\eta_\e$ of the soft buffer zone needed in the presence of Dirichlet boundary conditions (see Corollary~\ref{cor:dir}(ii))? Under which conditions can we take $\eta_\e=\theta\e$ for some constant $\theta>0$?
\item What about periodic boundary data? In particular, can one approximate $\overline V$ by periodization (in law) in the spirit of~\cite{EGMN-15}?
\end{itemize}
These three questions are partially addressed below under various additional assumptions.

\subsubsection{Subcritical case $1<p\le d$}
The first improvement concerns the growth condition from below in Theorem~\ref{th:conv}. It is relaxed here to any $p>1$ in the convex scalar case $m=1$ under the additional assumption that $V$ has fixed domain. The idea is to avoid the use of Sobolev embedding by using suitable truncations (in the spirit of e.g. \cite[proof of Lemma~13.1.5]{Carbone-DeArcangelis-02}), which are indeed only available in the scalar case with fixed domain. We recover in particular in this way the results of~\cite[Section~13.4]{Carbone-DeArcangelis-02} in Sobolev spaces.

\begin{cor}[Subcritical case]\label{cor:subcritp}
Let $V$ and $M$ satisfy the assumptions of Theorem~\ref{th:conv} in the scalar case $m=1$, for some $p>1$. Also assume that $\dom V(y,\cdot,\omega)=\dom M$ for almost all $y,\omega$.
Then the conclusions of Theorem~\ref{th:conv} and Corollary~\ref{cor:dir} hold true for this $p>1$.
\qed
\end{cor}

\subsubsection{Minimal soft buffer zone}

The second improvement concerns the size of the buffer zone for Dirichlet boundary data, at least for affine target functions. The minimal size $\eta_\e=\theta\e$, for any constant $\theta>0$, is achieved under the technical structural assumption that stationary quasi-correctors exist, cf. Proposition~\ref{prop:statcorbuffer}. Understanding the validity of this assumption in general seems to be a difficult question of functional analytic nature. It is trivially satisfied in the periodic case. It is also valid provided that truncations are available, which holds in the scalar case with fixed domain.

\begin{cor}[Minimal soft buffer zone]\label{cor:dir-bis}
Let $V,J_\e,J,M$ and $W,I_\e,I$ be as in Theorems~\ref{th:conv} and~\ref{th:nonconv} for some $p>1$. Also assume that one of the following holds:
\begin{enumerate}[(1)]
\item $p>d$, and, for all $\Lambda,\omega$, $V(\cdot,\Lambda,\omega)$ and $W(\cdot,\Lambda,\omega)$ are $Q$-periodic;
\item $m=1$, and $\dom V(y,\cdot,\omega)=\dom M$ is open for almost all $y,\omega$.
\end{enumerate}
Then, for all $\Lambda$, for almost all $\omega\in\Omega$,
\begin{eqnarray*}
\overline V(\Lambda)=\lim_{\e \downarrow 0} \inf_{v\in W^{1,p}_0(O)}\fint_OV_\e^{O,\theta\e}(y,\Lambda+\nabla v(y),\omega)dy,
\end{eqnarray*}
where $V_\e^{O,\theta \e}$ is defined as in~\eqref{eq:defVepsOeta} with $\eta=\theta\e$.
The same result also holds for $V,J_\e,J$ replaced by $W,I_\e,I$ (for $p>d$).
\qed
\end{cor}

\subsubsection{Approximation by periodization}
The last improvement concerns the approximation of $\overline V$ by periodization.
As for Dirichlet boundary conditions, periodic boundary conditions require ``well-prepared data''. 
In this case, the well-preparedness is formalized in terms of periodization of the law of the energy density (as opposed to the more naive periodization in space, cf. Figure~\ref{fig:Poisson-again}):

\begin{defin}\label{def:admperlaw}
A collection $(V^R)_{R>0}$ of random maps $V^R:\R^d\times\R^{m\times d}\times\Omega\to[0,\infty]$ is said to be an {\it admissible periodization in law} for $V$ if
\begin{enumerate}[(i)]
\item Periodicity in law: for all $R>0$, $V^R$ is $Q_R=[-\frac{R}{2},\frac{R}{2})^d$-periodic and is stationary with respect to translations on the torus $\T_R=(\R/R\Z)^d$ (more precisely there exists a measurable action $\tau^R$ of $\T_R$ on $\Omega$ such that, for all $y=y_1+y_2$ with $y_1\in (R\Z)^d$ and $y_2\in Q_R$, we have $V^R(y,\cdot,\omega)=V^R(y_2,\cdot,\omega)=V^R(0,\cdot,\tau^R_{-y_2}\omega)$);
\item Stabilization property: for all $\theta \in (0,1)$ and for almost all $\omega$, there exists $R_\theta(\omega)>0$ such that  for all $R>R_\theta(\omega)$ we have $V^R(\cdot,\cdot,\omega)|_{ Q_{\theta R}\times \R^{m\times d}}= V(\cdot,\cdot,\omega)|_{Q_{\theta R}\times \R^{m\times d}}$.\qed
\end{enumerate}
\end{defin}

\begin{rem}
It may seem quite unnatural to make in this definition the action $\tau^R$ depend on $R$. Another definition would consist in keeping the same integrand $V$ and action $\tau$ but letting the probability measure vary according to $\p_R:=(T_R)_*\p$, which amounts to defining $V^R(y,\Lambda,\omega)=V(y,\Lambda,T_R(\omega))$, for some projection $T_R:\Omega\to\Omega$ with suitable properties such as the $\tau$-invariance of the set $T_R(\Omega)$. As shown in the examples below, $T_R$ will typically be the (natural) $Q_R$-periodization of an underlying Poisson process, and $T_R(\Omega)$ will be the set of $Q_R$-periodic realizations (that is of $\p$-measure $0$). Note that this is a particular case of the formalism of Definition~\ref{def:admperlaw} above, for the conjugated action $\tau^R_y:=\tau_{y}\circ T_R$.
\end{rem}

The main goal here is to prove an approximation formula for $\overline V$ and $\overline W$ by periodization in law, that is an asymptotic formula of the following form: for all $\Lambda \in \R^{m\times d}$, for almost all $\omega$,
\begin{equation}\label{eq:approx-per0}
\overline V(\Lambda)\sim\lim_{R \uparrow \infty} \inf_{u\in W^{1,p}_\per(Q_R)} \fint_{Q_R} V^R(y,\Lambda+\nabla u(y),\omega)dy,
\end{equation}
and similarly for $\overline W$. Such a formula would be of particular interest for numerical approximations of $\overline V$ and $\overline W$. The notion of periodization in law was introduced in \cite{EGMN-15}. It crucially differs from ``periodization in space'', which would consist in replacing the argument of the limit in \eqref{eq:approx-per0}
by
\[\inf_{u\in W^{1,p}_\per(Q_R)} \fint_{Q_R} V(y,\Lambda+\nabla u(y),\omega)dy.\]
The difference between both types of periodization is illustrated on Figure~\ref{fig:Poisson-again} for Poisson random inclusions.
\begin{figure}
\centering
\begin{minipage}[b]{0.49\linewidth}
\includegraphics[scale=.7]{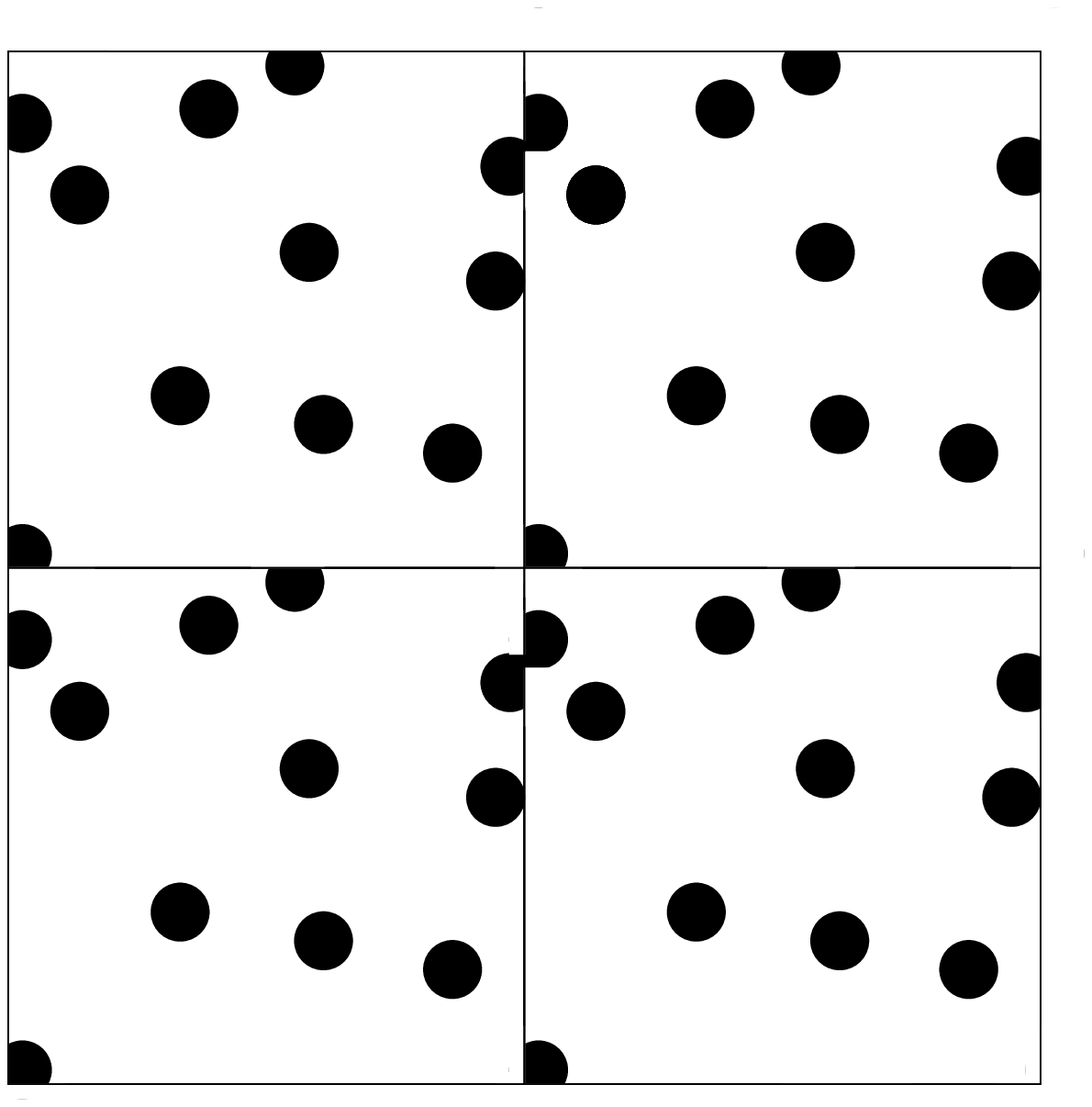} 
\end{minipage}
\begin{minipage}[b]{0.49\linewidth}
\includegraphics[scale=.7]{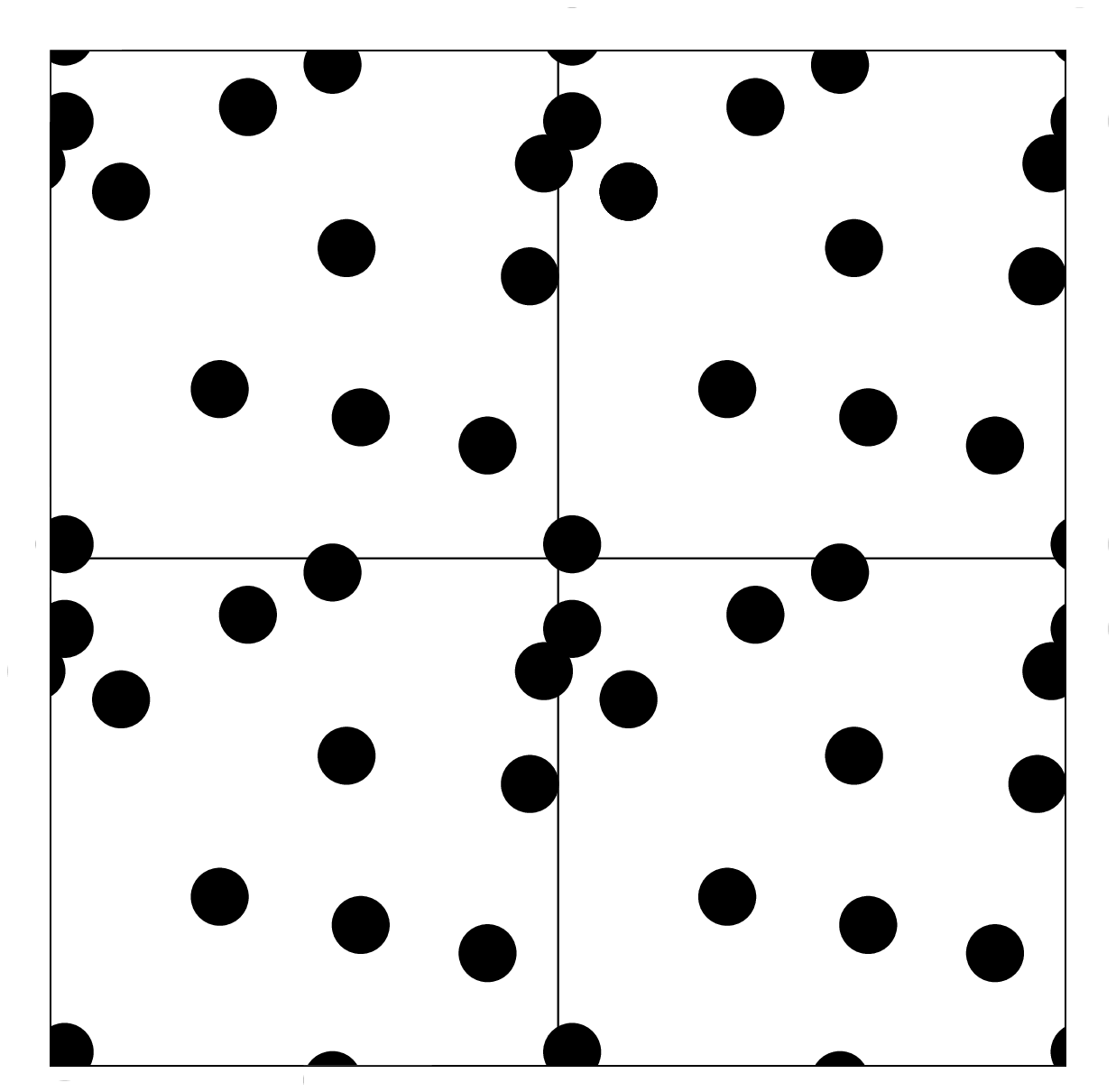}
\end{minipage}
\caption{For Poisson random inclusions: periodization in space (left) versus periodization in law (right), four periods.}
\label{fig:Poisson-again}
\end{figure}
Empirical results in~\cite{EGMN-15} tend to show that periodization in law is more precise than periodization in space for the approximation of homogenized coefficients in the case of discrete linear elliptic equations, whereas \cite{Glotto-Neukamm-15,Gloria-Nolen-14} provide with a complete numerical analysis (convergence rate, quantitative central limit theorem) of this method when the coefficients are independent and identically distributed.
In the case of the unbounded integral functionals considered here, it is not clear whether the approximation by periodization \emph{in space} converges in general. Note that for problems with discrete stationarity (like the random chessboard or periodic structures) such periodizations are not needed but the sequences $\e,R$ have to be discrete: $\e_k=1/k,R=k$ with $k\in \N$.

\medskip

Before we state our results, let us shortly discuss periodization in law on the example of random inclusions.
Let
\begin{align}\label{eq:inclusions}
V(y,\Lambda,\omega):=a(\Lambda)\mathds1_{\R^d\setminus E(\omega)}(y)+b(\Lambda)\mathds1_{E(\omega)}(y),
\end{align}
where $a,b:\R^{m\times d}\to[0,\infty]$ are convex functions and $E(\omega)$ is a stationary random set (to be thought of as a collection of inclusions). 
A typical choice for $E$ is $E=\bigcup_{n=1}^\infty B(q_n)$, where $\rho=(q_n)_n$ is a stationary point process in $\R^d$. 
For $R>0$, set $\rho^R=(q_{n,R})_n$, $E^R:=\bigcup_{n=1}^\infty B(q_{n,R})$, and define $V^R(y,\Lambda,\omega):=a(\Lambda)\mathds1_{\R^d\setminus E^R(\omega)}(y)+b(\Lambda)\mathds1_{E^R(\omega)}(y)$. For $V^R$ to be an admissible periodization in law of $V$, it is sufficient that $(\rho^R)_R$ is 
an admissible periodization of $\rho$ in the following sense:
\begin{enumerate}[(i')]
\item Periodicity in law: for all $R>0$, $\rho^R$ is $Q_R$-periodic and $\tau^R$-stationary with respect to translations on the torus $\T_R=(\R/R\Z)^d$;
\item Stabilization: for all $\theta \in(0,1)$ and for almost all $\omega$ there is $R_\theta(\omega)>0$ so that for all $R\ge R_\theta(\omega)$,
$\rho^R(\omega)\cap Q_{\theta R}= \rho(\omega)\cap Q_{\theta R}$.
\end{enumerate}
Let us now describe three typical examples, for which the periodization in law can be explicitly constructed:
\begin{enumerate}[(a)]
\item {\it Poisson point process.} If $\rho$ is a Poisson point process on $\R^d$, a suitable periodization in law is given by
$\rho^R=\bigcup_{z\in \Z^d}(Rz+\rho\cap Q_R)$. Property~(i') is satisfied by complete independence of the Poisson process, while property~(ii') holds for $\theta=1$.
\item {\it Random parking point process.} We now let $\rho$ be the random parking measure on $\R^d$ defined in~\cite{Penrose-01}.
Penrose's graphical construction is as follows: $\rho$ is obtained as a transformation $\rho:=T(\rho_0)$ of a Poisson point process $\rho_0$ on $\R^d\times\R^+$. 
Following~\cite[Remark~5]{Gloria-Penrose-13}, we may periodize the underlying Poisson process on $Q_R$ as above
by setting  $\rho_0^R:=\bigcup_{z\in\Z^d}((Rz,0)+\rho_0\cap (Q_R\times \R^+))$, and define a periodization of $\rho$ as $\rho^R:=T(\rho_0^R)$. Property~(i') holds by construction.
Property~(ii') is more subtle and relies on the stabilization properties of the random parking measure.
By a union bound argument we may rephrase the exponential stabilization of~\cite[Lemma~3.5]{Schreiber-Penrose-Yukich-07}
as follows: There exist $K,\kappa>0$ and for all $R>0$ there exists a random variable $t_R$
such that
\begin{itemize}
\item {\it $\rho\cap Q_R$ does not depend on $\rho_0\cap ((\R^d\setminus  Q_{t_R})\times\R^+)$}, and in particular, for any locally finite subset $F\subset (\R^d\setminus Q_{t_R})\times\R^+$,  we have
$$
Q_R\cap T(F\cup(\rho_0\cap(Q_{t_R}\times\R^+)))=Q_R\cap T(\rho_0\cap( Q_{t_R}\times\R^+));
$$
\item for all $L\ge 0$, $\p[t_R\ge L]\le R^d K e^{-\kappa  (L-R)}$.
\end{itemize}
Hence, for all $\theta \in(0,1)$,
\begin{align*}
&\sum_{R=1}^\infty\p[\text{$\rho\cap Q_{\theta R}$ does depend on $\rho_0\cap((\R^d\setminus Q_R)\times \R^+)$}]\\
\le~&\sum_{R=1}^\infty\p[t_{\theta R}\ge R]\le\sum_{R=1}^\infty (\theta R)^d K e^{-\kappa (1-\theta)R}<\infty.
\end{align*}
By the Borel-Cantelli lemma, $\rho\cap Q_{\theta R}$ does eventually not depend on $\rho_0\cap((\R^d\setminus Q_R)\times \R^+)$  as $R\uparrow\infty$
almost surely. In particular, for all $\theta\in(0,1)$ and for almost all $\omega$, there exists some $R_\omega>0$ with the following property: for all $R\ge R_\omega$, $\rho(\omega)\cap Q_{\theta R}$ does not depend on $\rho_0(\omega)\cap((\R^d\setminus Q_R)\times\R^+)$, so that $\rho^R(\omega)\cap Q_{\theta R}=\rho(\omega)\cap Q_{\theta R}$, that is,~(ii').
\item {\it Hardcore Poisson point process.} As noticed in~\cite[Step~1 of Section~4.2]{DG-15} (in a more general context), a hardcore Poisson point process $\rho$ on $\R^d$ can be obtained by applying the transformation $T$ of~(b) to a Poisson point process $\rho_0$ on $\R^d\times(0,1)$, that is, $\rho=T(\rho_0)$. 
The same argument as in~(b) (with however easier bounds) then yields the desired periodization in law of $\rho$.
\end{enumerate}

\medskip

In terms of applications, example (b) is of particular interest since it is used in \cite{Gloria-Penrose-13,Gloria-LeTallec-Vidrascu-08b,DeBuhan-Gloria-LeTallec-Vidrascu-10} as a suitable random point set for the derivation of nonlinear elasticity from polymer physics started in~\cite{Alicandro-Cicalese-Gloria-11}.

\medskip

We now turn to our $\Gamma$-convergence results with periodic boundary data under periodization in law (yielding in particular an approximation of the homogenized energy density of the form~\eqref{eq:approx-per0}).
The main difficulty is the following: in order to carry out the analysis as before, we would somehow need the uniform sublinearity of the correctors associated to the family of periodized integrands $(V^R)_{R>0}$. It is not clear to us whether this uniformity holds in general. In what follows, we give an alternative argument in the scalar case with fix domain, as well as in the very particular example of well-separated random stiff inclusions in a soft matrix. Note that because of the separation assumption this does not include the example (a) of Poisson inclusions; however we believe that up to some technicalities the proof could be adapted to that case. For simplicity, we consider spherical inclusions, but random shapes could of course be considered too.

\begin{cor}[Approximation by periodization]\label{cor:per}
Let $V,J_\e,J,M$ and $W,I_\e,I$ be as in Theorems~\ref{th:conv} and~\ref{th:nonconv} for some $p>1$.
Also assume that one of the following holds:
\begin{enumerate}[(1)]
\item $m=1$, and $\dom V(y,\cdot,\omega)=\dom M$ for almost all $y,\omega$;
\item $p>d$, and $V(y,\Lambda,\omega)\le C(1+|\Lambda|^p)$ for all $\omega$ and all $y\notin E^\omega$, for some random stationary set $E^\omega=\bigcup_{n=1}^\infty B_{R_n^\omega}(q_n^\omega)\subset\R^d$ satisfying almost surely, for all $n$, and some constant $C>0$,
\[\frac1{R_n^\omega} ~{\inf_{m,m\ne n}\dist(B_{R_m^\omega}(q_m^\omega),B_{R_n^\omega}(q_n^\omega))}\ge \frac1C,\qquad R_n^\omega\le C.\]
\end{enumerate}
Let $(V^R)_{R>0}$ be an admissible periodization in law for $V$ in the sense of Definition~\ref{def:admperlaw}, and for all $\e>0$ and all $\Lambda \in \R^{m\times d}$ denote by $J^{\per}_\e(\cdot,\Lambda,\cdot)$ the following random integral functional on the unit cube $Q=[-\frac{1}{2},\frac{1}{2})^d$:
$$
J^{\per}_\e(u,\Lambda,\omega):=\int_Q V^{1/\e}({y}/{\e},\Lambda+\nabla u(y),\omega)dy, \qquad u\in W^{1,p}_{\per}(Q;\R^m),
$$
where $W^{1,p}_\per(Q;\R^m)$ denotes the closure in $W^{1,p}(Q;\R^m)$ of the set of smooth periodic functions on $Q$.
Then, for almost all $\omega$ and for all $\Lambda$,  the integral functionals $J_\e^{\per}(\cdot,\Lambda,\omega)$ $\Gamma$-converge to $J(\cdot+\Lambda\cdot x,\omega)$ on the space $W^{1,p}_{\per}(Q;\R^m)$.
In particular, we have the following approximation of $\overline V$ by periodization:
for all $\Lambda \in \R^{m\times d}$, for almost all $\omega$,
\begin{align*}
\overline V(\Lambda)&=\lim_{t\uparrow1}\lim_{R \uparrow \infty} \inf_{u\in W^{1,p}_\per(Q_R)} \fint_{Q_R} V^R(y,t\Lambda+\nabla u(y),\omega)dy\\
&=\lim_{t\uparrow1}\lim_{R \uparrow \infty} \E\left[\inf_{u\in W^{1,p}_\per(Q_R)} \fint_{Q_R} V^R(y,t\Lambda+\nabla u(y),\cdot)dy\right].
\end{align*}
By convexity, if $\Lambda\in\inter \dom\overline V$, then the limits $t\uparrow1$ can be dropped. Note that $\dom\overline V=\R^{m\times d}$ holds in case~(2).
The same results also hold for $V,J_\e,J$ replaced by $W,I_\e,I$, provided that $p>d$.
\qed
\end{cor}

\section{Proof of the results for convex integrands}\label{sec:convex}

This section is dedicated to the proofs of Theorem~\ref{th:conv}, Corollary~\ref{cor:dir},
and Corollaries~\ref{cor:dir-bis} and~\ref{cor:per}. Let $V$ be a convex $\tau$-stationary normal random integrand. Up to the addition of a constant, we may restrict to the following stronger version of~\eqref{eq:aslowerbound0}: for almost all $\omega$, $y$, we have, for all $\Lambda$,
\begin{align}\label{eq:cond12s}
\frac1C|\Lambda|^{p}\,\le\, V(y,\Lambda,\omega),
\end{align}
for some $C>0$ and $1<p<\infty$. 
We assume that $0$ belongs to the interior of the domain of the convex function $M:=\supess_{y,\omega}V(y,\cdot,\omega)$.

Following the strategy of~\cite[Theorem~1.5]{Muller-87}, we proceed by truncation of $V$. We let $(V^k)_k$ be an increasing sequence of $\tau$-stationary convex normal random integrands
$V^k:\R^d\times\R^{m\times d}\times\Omega\to\R^+$ such that, for almost all $\omega$, $y$, and for all $\Lambda$, we have
\begin{eqnarray}\label{eq:SGC-trun}
\lim_{k\uparrow \infty} V^k(y,\Lambda,\omega)\,=\,V(y,\Lambda,\omega),\qquad\text{and}\qquad\frac1C |\Lambda|^p\, \le \, V^k(y,\Lambda,\omega)\,\le\, C^k(|\Lambda|^p+1),
\end{eqnarray}
for some $C>0$ and some sequence $C^k \uparrow \infty$ (see \cite[Lemma~3.4]{Muller-87} for such a construction). Let $\Omega_{0}\subset\Omega$ be a subset of maximal probability on which all these assumptions (about $V$ and the $V^k$'s) are simultaneously pointwise satisfied.

We shall prove the existence of a subset $\Omega'\subset\Omega_0$ of maximal probability such that, for all $\omega\in\Omega'$ and all bounded Lipschitz domains $O\subset\R^d$, the functionals $J_\e(\cdot,\omega;O)$ $\Gamma$-converge to the functional $J(\cdot;O)$ on $W^{1,p}(O;\R^m)$, where we recall
the definitions
$$
J_\e(u,\omega;O):=\int_O V(y/\e,\nabla u(y),\omega)dy,\qquad J(u;O):=\int_O\overline V(\nabla u(y))dy.
$$
As usual, the proof of $\Gamma$-convergence splits into two parts: the proof of a lower bound ($\Gamma$-$\liminf$ inequality) and the explicit construction of a recovery sequence which achieves the lower bound ($\Gamma$-$\limsup$ inequality). 


\subsection{Preliminaries}

We first need to briefly recall the standard stationary differential calculus in probability (first introduced by~\cite[Section~2]{PapaVara}), as well as some results on ergodic Weyl decompositions.

\subsubsection{Stationary differential calculus in probability}\label{chap:preliminD}
Let $1\le p<\infty$. For all $1\le i\le d$, consider the partial action $(T^i_h)_{h\in\R}$ of $(\R,+)$ on $\Ld^p(\Omega)$, defined by $(T^i_h f)(\omega)= f(\tau_{-he_i}\omega)$, for $h\in\R$. The actions $(T^i_h)_{h\in\R}$ (for $1\le i\le d$) commute with each other and are unitary and strongly continuous by Lemma~\ref{lem:unit}. For all $i$, we may then consider the infinitesimal generator $D_i$ of $(T^i_h)_{h\in\R}$, defined by
\[D_i f=\lim_{h\to0}\frac{T^i_h f- f}{h},\quad f\in\Ld^p(\Omega),\]
whenever the limit exists in the strong sense of $\Ld^p(\Omega)$. By classical semigroup theory, the generators $D_i$ are closed linear operators with dense domains $\D_i\subset\Ld^p(\Omega)$, and the intersection $W^{1,p}(\Omega):=\bigcap_{i=1}^d\D_i$ is also dense in $\Ld^p(\Omega)$. Moreover, $W^{1,p}(\Omega)$ is endowed with a natural Banach space structure.

For $f\in W^{1,p}(\Omega)$, its {\it stationary gradient} is then defined by $Df:=(D_1f,\ldots,D_df)\in\Ld^p(\Omega;\R^d)$.
Through the usual correspondence between random variables and $\tau$-stationary random fields as recalled in Appendix~\ref{app:cadreproba} (for all $g\in \Ld^p(\Omega)$, write $g(x,\omega):=g(\tau_{-x}\omega)$, defining $g\in\Ld^p_{\loc}(\R^d;\Ld^p(\Omega))$), we may define $Df(x,\omega):=D f(\tau_{-x}\omega)$ for all $x$. By unitarity of the action $T$, the operator $D$ is skew-symmetric, so that the following ``integration by parts formula'' holds, for all $ f\in W^{1,p}(\Omega)$ and $ g\in W^{1,p'}(\Omega)$, $p'=p/(p-1)$,
\[\E [D f]=0 \qquad\text{and} \qquad\E [f D g]=-\E[gD f].\]

As explained in Section~\ref{app:statsobolev} (see in particular Lemma~\ref{lem:statsobolev}), for almost all $\omega$, the function $Df(\cdot,\omega)$ is nothing but the distributional derivative of $f(\cdot,\omega)\in \Ld^p_{\loc}(\R^d)$, and the following identity holds:
\begin{align}\label{eq:linkDnabla}
W^{1,p}(\Omega)=\{f\in W^{1,p}_{\loc}(\R^d;\Ld^p(\Omega)): f(x+y,\omega)=f(x,\tau_{-y}\omega),\forall x,y,\omega\}.
\end{align}
This justifies that in the sequel we simply use the notation $Df=\nabla f$.

\subsubsection{Ergodic Weyl decomposition}\label{chap:preliminweyl}
Ergodicity of the measurable action $\tau$ of $(\R^d,+)$ on the probability space $(\Omega,\F,\p)$ is crucial in the sequel. Let $1<p<\infty$. In analogy with the classical Weyl subspaces of $\Ld^p_{\loc}(\R^d;\R^d)$, we define the subspaces of potential and solenoidal random fields with respect to the differential calculus associated with the group action in the following way: for $p'=p/(p-1)$,
\begin{align}\label{eq:defweylerg-0}
\Ld^p_\pot(\Omega)&=\{f\in\Ld^p(\Omega;\R^d)\,:\,\E [f\cdot(\nabla\times g)]=0,~\forall g\in W^{1,p'}(\Omega;\R^d)\},\\
\Ld^p_\sol(\Omega)&=\{f\in\Ld^p(\Omega;\R^d)\,:\,\E[ f\cdot \nabla g]=0,~\forall g\in W^{1,p'}(\Omega)\}.\nonumber
\end{align}
Reinterpreting these definitions in physical space, we easily obtain the following reformulations in terms of stationary extensions:
\begin{align}\label{eq:defweylerg-1}
\Ld^p_\pot(\Omega)&=\{f\in\Ld^p(\Omega;\R^d)\,:\,\text{for almost all $\omega$, $x\mapsto f(\tau_{x}\omega)\in\Ld^p_\loc(\R^d;\R^d)$ is potential}\},\\
\Ld^p_\sol(\Omega)&=\{f\in\Ld^p(\Omega;\R^d)\,:\,\text{for almost all $\omega$, $x\mapsto f(\tau_{x}\omega)\in\Ld^p_\loc(\R^d;\R^d)$ is solenoidal}\},\nonumber
\end{align}
where a function $h\in \Ld^p_\loc(\R^d)$ is said to be potential (resp. solenoidal) if $\nabla\times h=0$ (resp. $\nabla\cdot f=0$) in the distributional sense. Constant functions belong to both subspaces, and we further define
\begin{align}\label{eq:defweylerg-F}
F^p_\pot(\Omega)=\{ f\in\Ld^p_\pot(\Omega)\,:\,\E[ f]=0\}\quad\text{and}\quad F^p_\sol(\Omega)=\{ f\in\Ld^p_\sol(\Omega)\,:\,\E[ f]=0\}.
\end{align}
The spaces $\Ld^p_\pot(\Omega)$, $\Ld^p_\sol(\Omega)$, $F^p_\pot(\Omega)$ and $F^p_\sol(\Omega)$ are all closed in $\Ld^p(\Omega;\R^d)$, and the following Banach direct sum decomposition holds
\begin{align}\label{eq:decweyl}
\Ld^p(\Omega;\R^d)=F^p_\pot(\Omega)\oplus F^p_\sol(\Omega)\oplus\R^d,
\end{align}
as well as the following density results:
\begin{align}\label{eq:densweyl}
F^p_\pot(\Omega)&=\adh_{\Ld^p(\Omega;\R^d)}\{\nabla g\,:\,g\in W^{1,p}(\Omega)\},\\
F^p_\sol(\Omega)&=\adh_{\Ld^p(\Omega;\R^d)}\{\nabla\times g\,:\,g\in W^{1,p}(\Omega)\}.\nonumber
\end{align}
Since we did not find a suitable reference for these results (besides the Hilbert setting), we provide a proof of~\eqref{eq:decweyl} and~\eqref{eq:densweyl}  in Appendix~\ref{app:weylerg}.

\subsection{$\Gamma$-convergence of truncated energies}\label{chap:gammaconvtrunc}
Since the approximations $V^k$ of $V$ all satisfy standard polynomial growth conditions, we can appeal to the classical stochastic homogenization result of \cite{DalMaso-Modica-86} (which could be reproved via a direct adaptation of the (periodic) arguments of~\cite[Theorem 1.3]{Muller-87}). More precisely, there exists a subset $\Omega_1\subset\Omega_{0}$ of maximal probability such that, for all $\omega\in\Omega_1$, all $k$, and all $\Lambda\in\R^{m\times d}$, the following limit exists (as a consequence of the Ackoglu-Krengel subadditive ergodic theorem) and defines the homogenized integrand $\overline V^k$:
\begin{align}\label{eq:firstdefwbark}
\overline V^k(\Lambda)=\lim_{R\uparrow\infty}\inf_{\phi\in W^{1,p}_0(Q_R;\R^m)}\fint_{Q_R}V^k(y,\Lambda+\nabla\phi(y),\omega)dy,
\end{align}
where $Q_R:=[-\frac{R}{2},\frac{R}{2})^d$. By dominated convergence, this convergence also holds when taking the expectation of the infimum. 
In addition, for any bounded Lipschitz domain $O\subset\R^d$, and for all $\omega\in\Omega_1$ and all $k$, the functionals $J^k_\e(\cdot,\omega;O)$ $\Gamma$-converge, as $\e\downarrow0$, to the functional $J^k(\cdot;O)$,  defined by
$$
J^k_\e(u,\omega;O):=\int_OV^k(y/\e,\nabla u(y),\omega)dy,\qquad\text{and}\qquad J^k(u;O):=\int_O\overline V^k(\nabla u(y))dy.
$$
Since $k\mapsto V^k$ is increasing, $k\mapsto \overline V^k$ is increasing as well, and for all $\Lambda\in \R^{m\times d}$
we may define $\overline V(\Lambda):=\lim_{k\to \infty}\overline V^k(\Lambda)$. In particular,
\begin{align}\label{eq:defWbar}
\overline V(\Lambda)=\sup_k\overline V^k(\Lambda)=\sup_k\lim_{R\uparrow\infty}\inf_{\phi\in W^{1,p}_0(Q_R;\R^m)}\fint_{Q_R}V^k(y,\Lambda+\nabla\phi(y),\omega)dy.
\end{align}
It remains to pass to the limit $k\uparrow \infty$ in the $\Gamma$-convergence result. The key is to prove the commutation of homogenization and truncation, which we do in Subsection~\ref{chap:commutform} below.

Alternative formulas for $\overline V^k$ are obtained in Lemma~\ref{lem:equivk}.
Since $\overline V^k$ is convex and everywhere finite, it is continuous on $\R^{m\times d}$,  and we may directly deduce from the definition $\overline V(\Lambda):=\sup_k\overline V^k(\Lambda)$:
\begin{lem}
The map $\overline V:\R^{m\times d}\to[0,\infty]$ is convex and lower semicontinuous.
\qed
\end{lem}

\subsection{$\Gamma$-$\liminf$ inequality}

In view of the definition of $\overline V$, the $\Gamma$-$\liminf$ inequality is an elementary consequence of the monotone convergence theorem:
\begin{prop}[$\Gamma$-$\liminf$ inequality]\label{prop:gammaliminfneu}
For all $\omega\in\Omega_1$, all bounded domains $O\subset\R^d$, and all sequences $(u_\e)_\e\subset W^{1,p}(O;\R^m)$ with $u_\e\cvf{} u$ in $W^{1,p}(O;\R^m)$, we have
\[\liminf_{\e\downarrow0}J_\e(u_\e,\omega;O)\ge J(u;O).\]
\qed
\end{prop}
\begin{proof}
Let $O$, $(u_\e)_\e$, and $u$ be as in the statement. Then, for all $\omega\in\Omega_1$ and all $k\in\N$, using the $\Gamma$-$\liminf$ result for $J^{k}_\e(\cdot,\omega;O)$ towards $J^k(\cdot;O)$ (see Section~\ref{chap:gammaconvtrunc}), and that $V\ge V^k$,
$$
\liminf_{\e\downarrow0}J_\e(u_\e,\omega;O)\ge\liminf_{\e\downarrow0}J_\e^{k}(u_\e,\omega;O)\ge J^k(u;O),
$$
so that, by monotone convergence,
\[\liminf_{\e\downarrow0}J_\e(u_\e,\omega;O)\ge\lim_{k\uparrow \infty}J^{k}(u;O)=J(u;O).\qedhere\]
\end{proof}

From this $\Gamma$-$\liminf$ result, we deduce the locality of recovery sequences, if they exist. 
\begin{cor}[Locality of recovery sequences]\label{cor:localizlim}
If for some $\omega\in\Omega_1$, some bounded domain $O\subset\R^d$, and some function $u\in W^{1,p}(O;\R^m)$, there exists a sequence $(u_\e)_\e\subset W^{1,p}(O;\R^m)$ with $u_\e\cvf{}u$ in $W^{1,p}(O;\R^m)$ and $J_\e(u_\e,\omega;O)\to J(u;O)$, then we also have $J_\e(u_\e,\omega;O')\to J(u;O')$ for any subdomain $O'\subset O$. Hence, by an extension result, the $\Gamma$-$\limsup$ inequality on a bounded Lipschitz domain $O$ implies the $\Gamma$-$\limsup$ inequality on any Lipschitz subdomain $O'\subset O$.\qed
\end{cor}
\begin{proof}
Choose a subdomain $O'\subset O$, and define $O'':=O\setminus O'$. We then have by assumption
\begin{align*}
J(u;O)=\lim_{\e\downarrow0}J(u_\e;O)&=\lim_{\e\downarrow0}\left(J_\e(u_\e,\omega;O')+J_\e(u_\e,\omega;O'')\right)\\
&\ge \liminf_{\e\downarrow0}J_\e(u_\e,\omega;O')+\liminf_{\e\downarrow0}J_\e(u_\e,\omega;O'').
\end{align*}
Now by Proposition~\ref{prop:gammaliminfneu} we have $\liminf_{\e\downarrow0}J_\e(u_\e,\omega;O')\ge J(u;O')$ and $\liminf_{\e\downarrow0}J_\e(u_\e,\omega;O'')\ge J(u;O'')$. The conclusion then follows from the identity $J(u;O')+J(u;O'')=J(u;O)$.
\end{proof}

\subsection{Commutation of truncation and homogenization}\label{chap:commutform}

The crucial ingredient to prove the commutation of truncation and homogenization is the reformulation of the asymptotic homogenization formula in the probability space. For that purpose, we first introduce the following proxy for $\overline V$:
\begin{align}\label{eq:defpb1}
\overline P(\Lambda):=\inf_{f\in F^{p}_\pot(\Omega)^m}\E[V(0,\Lambda+f,\cdot)].
\end{align}
Likewise, for all $k\in \N$, we set
\begin{align}\label{eq:defpb1-k}
\overline P^k(\Lambda):=\inf_{f\in F^{p}_\pot(\Omega)^m}\E[V^k(0,\Lambda+f,\cdot)].
\end{align}
In this case, due to the growth condition~\eqref{eq:SGC-trun}, we may prove (see Lemma~\ref{lem:equivk} below)
that
$$
\lim_{\e\downarrow0}\E\left[\inf_{\phi\in W^{1,p}_0(O/\e;\R^m)}\fint_{O/\e}V^k(y,\Lambda+\nabla\phi(y),\cdot)dy\right]=\overline P^k(\Lambda).
$$
However, for $V$ itself, this equality has no chance to be true if $\Lambda\in \dom \overline P\setminus\dom M$ since the left-hand side could be infinite (because of the Dirichlet boundary condition) while the right-hand side is not --- see Example~\ref{example}. 
We thus rather use a ``relaxed version'' of the Dirichlet boundary conditions and set for all $\e>0$ 
\begin{align}\label{eq:defpe1}
P^\e(\Lambda,\omega;O):=\inf_{\phi\in W^{1,p}(O/\e;\R^m)\atop \fint_{O/\e}\nabla\phi=0}\fint_{O/\e}V(y,\Lambda+\nabla\phi(y),\omega)dy.
\end{align}
As opposed to the case of Dirichlet boundary conditions, there is no natural subadditive property in this definition (two test-functions on disjoint domains cannot be glued together). This difficulty will be overcome by using a more sophisticated gluing argument that relies quantitatively on the following sublinearity property of the correctors.
\begin{lem}[Sublinearity of correctors]\label{lem:sublincor}
For all $\Lambda\in\R^{m\times d}$, there exists a corrector field $\varphi_\Lambda\in\Mes(\Omega;W^{1,p}_\loc(\R^d;\R^m))$ such that $\nabla\varphi_\Lambda(0,\cdot)\in F^p_\pot(\Omega)^m$, and
$$
\overline P(\Lambda)=\E\left[V(0,\Lambda+\nabla\varphi_\Lambda(0,\cdot),\cdot)\right].
$$
In addition, $\varphi_\Lambda$ is sublinear at infinity in the sense that, for almost all $\omega\in\Omega$, 
\begin{align}\label{eq:sublincor}
\e\varphi_\Lambda(\cdot/\e,\omega)\cvf{} 0
\end{align}
weakly in $W^{1,p}(O;\R^m)$ for all bounded domains $O\subset\R^d$.\qed
\end{lem}
\begin{rem}
Although the space $\{\nabla g:g\in W^{1,p}(\Omega)\}$ is dense in $F^p_\pot(\Omega)$ (see~\eqref{eq:densweyl}), the infimum~\eqref{eq:defpb1} defining $\overline P(\Lambda)$ cannot be replaced in general by an infimum over this smaller dense subspace because of a possible lack of strong continuity of the functional,
see however Proposition~\ref{prop:statcorbuffer}. \qed 
\end{rem}
\begin{proof}
Let $\Lambda\in \R^{m\times d}$ be fixed.
By convexity and by the lower bound \eqref{eq:cond12s} on $V$, $\chi \mapsto \E[V(0,\Lambda+\chi,\cdot)]$ is lower semicontinuous and coercive on $F^p_\pot(\Omega)^m$, and therefore attains its infimum. Let $g\in F^p_\pot(\Omega)^m$ be a minimizer.
The $\tau$-stationary extension $(x,\omega) \mapsto g(\tau_x\omega)$ of  $g$ is a potential field on $\R^d$ for almost every $\omega$.
Hence, there exists a map $\varphi_\Lambda\in\Mes(\Omega; W^{1,p}_\loc(\R^d;\R^m))$ such that $g(\tau_x\omega)=\nabla\varphi_\Lambda(x,\omega)$ for almost all $x,\omega$ (see indeed Proposition~\ref{prop:measpot}).
The claim now follows from the combination of the following two applications of the Birkhoff-Khinchin ergodic theorem: for almost all $\omega$,
\begin{eqnarray}
\nabla \varphi_\Lambda(\cdot/\e,\omega)&\cvf{}&0,\quad\text{(weakly) in }\Ld^p(O;\R^m), \label{eq:pr-sublin-1}\\
\e \int_O \varphi_\Lambda(y/\e,\omega)dy& \to& 0.  \label{eq:pr-sublin-2}
\end{eqnarray}
Indeed, by Poincar\'e's inequality and \eqref{eq:pr-sublin-1}, the sequence $y\mapsto \e \varphi_\Lambda(y/\e,\omega)- \e \int_O \varphi_\Lambda(z/\e,\omega)dz$ 
converges weakly to zero in $W^{1,p}(O;\R^m)$ for almost every $\omega$.
Combined with \eqref{eq:pr-sublin-2}, this implies \eqref{eq:sublincor}.

To conclude, we turn to the proofs of \eqref{eq:pr-sublin-1} and \eqref{eq:pr-sublin-2}.
The weak convergence \eqref{eq:pr-sublin-1} to zero is a direct consequence of the Birkhoff-Khinchin ergodic theorem
in the form $\nabla \varphi_\Lambda(\cdot/\e,\omega)\cvf{}\E[\nabla \varphi_\Lambda(0,\cdot)]=0$ in $\Ld^p(O;\R^m)$.
It remains to prove \eqref{eq:pr-sublin-2}.
We may assume wlog $\varphi_\Lambda(0,\cdot)=0$ almost surely, so that 
\begin{align}
\left|\e\fint_{O}\varphi_\Lambda(y/\e,\omega)dy\right|=\left|\e\fint_{O/\e}\varphi_\Lambda(\cdot,\omega)\right| &=\left|\e\fint_{O/\e} \int_0^1 x\cdot\nabla \varphi_\Lambda(tx,\omega)dtdx\right|\nonumber\\
&\le \fint_0^{1/\e}\left|\fint_{O}x\cdot\nabla \varphi_\Lambda(tx,\omega)dx\right|dt.\label{eq:boundintphil}
\end{align}
For almost all $\omega$, the function $\psi_\omega(t):=\fint_{O}x\cdot\nabla \varphi_\Lambda(tx,\omega)dx$ is continuous on $(0,\infty)$. By~\eqref{eq:pr-sublin-1}, $\psi_\omega(t)\to0$ as $t\uparrow\infty$ for almost all $\omega$. By joint measurability and (local) integrability of $\nabla\varphi_\Lambda$, and by stationarity, $0$ is a Lebesgue point of $\nabla\varphi_\Lambda(\cdot,\omega)$ for almost all $\omega$, and hence $\limsup_{t\downarrow0}|\psi_\omega(t)|<\infty$ for almost all $\omega$. The result~\eqref{eq:pr-sublin-2} then follows from~\eqref{eq:boundintphil}.
\end{proof}

For all $\Lambda\in\dom\overline P$, let $\varphi_\Lambda$ be defined as in Lemma~\ref{lem:sublincor}, and let $\Omega_\Lambda\subset\Omega_1$ be a subset of maximal probability such that~\eqref{eq:sublincor} holds on $\Omega_\Lambda$ for all bounded Lipschitz domains. Restricting $\Omega_\Lambda$ further, the Birkhoff-Khinchin ergodic ensures that, for all $\omega\in\Omega_\Lambda$, we have for all bounded subsets $O\subset\R^d$ and all $t\in\Q$,
\begin{align}\label{eq:ergthoml0}
\fint_{O/\e}\nabla\varphi_{t\Lambda}(\cdot,\omega)\xrightarrow{\e\downarrow0}0
\end{align}
and
\begin{align}\label{eq:ergthoml}
\fint_{O/\e}V(y,t\Lambda+\nabla\varphi_{t\Lambda}(y,\omega),\omega)dy\xrightarrow{\e\downarrow0}\E[V(0,t\Lambda+\nabla\varphi_{t\Lambda}(0,\cdot))]=\overline P(t\Lambda).
\end{align}

We now turn to the proof that $\lim_\e P^\e(\Lambda,\omega;O)=\overline P(\Lambda)$ for all $\Lambda$ for almost all $\omega \in \Omega$.
The following inequality is the most subtle part.
\begin{lem}\label{lem:ineqhomog}
For all $\Lambda\in\inter\dom\overline P$ and all bounded domains $O\subset\R^d$, there exists a sequence $\psi_{\Lambda,O,\e}\in \Mes(\Omega;W^{1,p}(O/\e;\R^m))$ such that, for all $\omega\in\Omega_\Lambda$, $\fint_{O/\e}\nabla \psi_{\Lambda,O,\e}(\cdot,\omega)=0$, 
$$\e\psi_{\Lambda,O,\e}(\cdot/\e,\omega)\cvf{}0$$
weakly in $W^{1,p}(O;\R^m)$ as $\e\downarrow0$, 
and
\begin{equation}\label{eq:commut-dur}
\overline P(\Lambda)\ge\limsup_{\e\downarrow0}\fint_{O/\e}V(y,\Lambda+\nabla\psi_{\Lambda,O,\e}(y,\omega),\omega)dy\ge \limsup_{\e\downarrow0}P^\e(\Lambda,\omega;O).
\end{equation}
\qed
\end{lem}

\begin{proof}
Let $\Lambda\in\inter\dom\overline P$ be fixed, and let $\omega\in\Omega_\Lambda$.
For all $t\in[0,1)\cap\Q$ and $\e>0$, set
$$
\Lambda_{O,\e,t}^\omega:=-t\fint_{O/\e}\nabla\varphi_{\Lambda/t}(\cdot,\omega),\quad\text{and}\quad\psi_{\Lambda,O,\e,t}(x,\omega):=t\varphi_{\Lambda/t}(x,\omega)+\Lambda_{O,\e,t}^\omega\cdot x.
$$
By definition, $\fint_{O/\e}\nabla\psi_{\Lambda,O,\e,t}=0$, and by  Lemma~\ref{lem:sublincor}, $\e\psi_{\Lambda,O,\e,t}(\cdot/\e,\omega)\cvf{}0$ in $W^{1,p}(O;\R^m)$ as $\e\downarrow0$.
Hence,
\begin{align*}
P^\e(\Lambda,\omega;O)&\le\fint_{O/\e}V(y,\Lambda+\nabla\psi_{\Lambda,O,\e,t}(y,\omega),\omega)=:\widehat P^\e_t(\Lambda,\omega;O).
\end{align*}
By convexity
\begin{align*}
\widehat P^\e_t(\Lambda,\omega;O)&=\fint_{O/\e}V(y,\Lambda+t\nabla\varphi_{\Lambda/t}(y,\omega)+\Lambda_{O,\e,t}^\omega,\omega)dy\\
&\le t\fint_{O/\e}V(y,\Lambda/t+\nabla\varphi_{\Lambda/t}(y,\omega),\omega)dy+(1-t)\fint_{O/\e}V\left(y,\frac1{1-t}\Lambda_{O,\e,t}^\omega,\omega\right)dy.
\end{align*}
Since $0\in \inter \dom M$, there exists $\delta>0$ such that $\adh B_\delta\subset\inter\dom M$. As $t$ is rational and $\omega\in\Omega_\Lambda$, we have $\Lambda_{O,\e,t}^\omega\to0$ as $\e\downarrow0$ by the Birkhoff-Khinchin ergodic theorem in the form of~\eqref{eq:ergthoml0}.
Hence there exists $\e_{\Lambda,O,t}^{\omega}>0$ such that, for all $0<\e<\e_{\Lambda,O,t}^{\omega}$, we have
$$\left|\frac{1}{1-t}\Lambda_{O,\e,t}^\omega\right|<\delta,$$
and therefore,
$$\fint_{O/\e}V\left(y,\frac1{1-t}\Lambda_{O,\e,t}^\omega,\omega\right)dy\le \sup_{|\Lambda'|<\delta}M(\Lambda')<\infty.$$
This implies that
$$\limsup_{t\uparrow1, t\in\Q}\limsup_{\e\downarrow0}\widehat P_t^\e(\Lambda,\omega;O)\le\limsup_{t\uparrow1, t\in\Q}\limsup_{\e\downarrow0}\fint_{O/\e}V(y,\Lambda/t+\nabla\varphi_{\Lambda/t}(y,\omega),\omega)dy.$$
By the Birkhoff-Khinchin ergodic theorem in the form of~\eqref{eq:ergthoml} and the continuity of $\overline P$ in the interior of its domain (as a consequence of convexity), this yields
$$\limsup_{t\uparrow1, t\in\Q}\limsup_{\e\downarrow0}\widehat P_t^\e(\Lambda,\omega;O)\le \limsup_{t\uparrow1, t\in\Q}\E[V(0,\Lambda/t+\nabla\varphi_{\Lambda/t}(0,\cdot),\cdot)]=\limsup_{t\uparrow1, t\in\Q}\overline P(\Lambda/t)=\overline P(\Lambda).
$$
We have thus proved:
\begin{align*}
&\limsup_{t\uparrow1, t\in\Q}\limsup_{\e\downarrow0}\left(\left(\widehat P_t^\e(\Lambda,\omega;O)-\overline P(\Lambda)\right)^++\|\e\psi_{\Lambda,O,\e,t}(\cdot/\e,\omega)\|_{\Ld^p(O;\R^m)}\right)=0.
\end{align*}
By Attouch's diagonalization lemma (see~\cite[Corollary~1.16]{Attouch-84}), this implies the existence of a sequence $(\psi_{\Lambda,O,\e})_\e$ with $\psi_{\Lambda,O,\e}\in W^{1,p}(O/\e;\R^m)$ such that  $\int_{O/\e} \nabla \psi_{\Lambda,O,\e}=0$, $\limsup_\e\widehat P_t^\e(\Lambda,\omega;O)\le\overline P(\Lambda)$, and $\e\psi_{\Lambda,O,\e}(\cdot/\e,\omega)\to0$ in $\Ld^p(O;\R^m)$, for all $\omega\in\Omega_\Lambda$. 
By the choice of $\Lambda$, $\overline P(\Lambda)<\infty$, so that the lower bound \eqref{eq:cond12s} on $V$ implies that the sequence $(\nabla\psi_{\Lambda,O,\e}(\cdot/{\e},\omega))_\e$ is bounded in $\Ld^p(O;\R^m)$. We thus conclude that $\e\psi_{\Lambda,O,\e}(\cdot/\e,\omega)\cvf{}0$ weakly in $W^{1,p}(O;\R^m)$, as claimed.
\end{proof}

In the case of standard growth conditions (thus e.g. for the $V^k$'s), the corresponding inequality \eqref{eq:commut-dur} in Lemma~\ref{lem:ineqhomog} 
is indeed an equality. The following lemma gives equivalent definitions for the $\overline V^k$'s, which will be crucial in the sequel.
\begin{lem}\label{lem:equivk}
Let $O$ be a bounded Lipschitz domain of $\R^d$.
For all $\omega\in\Omega_1$, all $k$, and all $\Lambda\in\R^{m\times d}$, the following quantities are well-defined:
\begin{eqnarray}
\overline V_1^k(\Lambda)\label{eq:equivk1}&:=&\lim_{\e\downarrow0}\inf_{\phi\in W^{1,p}_0(O/\e;\R^m)}\fint_{O/\e}V^k(y,\Lambda+\nabla\phi(y),\omega)dy,\\
\overline V_{2}^k(\Lambda,\omega)\label{eq:equivk3}&:=&\lim_{\e\downarrow0}\inf_{\phi\in W^{1,p}(O/\e;\R^m)\atop\fint_{O/\e}\nabla\phi=0}\fint_{O/\e}V^k(y,\Lambda+\nabla\phi(y),\omega)dy,\\
\overline V_3^k(\Lambda)\label{eq:equivk4}&:=&\inf_{\tilde f\in F^p_\pot(\Omega)^m}\E[V^k(0,\Lambda+\tilde f,\cdot)],
\end{eqnarray}
and we have
$$
\overline V^k(\Lambda)=\overline V^k_1(\Lambda)=\overline V^k_2(\Lambda)=\overline V^k_3(\Lambda).
$$
\qed
\end{lem}
This result is standard (see for instance \cite[Chapter~15]{JKO94}) and we display its proof for completeness.
Note that the formulas \eqref{eq:equivk3} and \eqref{eq:equivk4} for $V$ will be shown to be equivalent to $\overline V$, whereas formula \eqref{eq:equivk1} is in general larger than $\overline V$.

\begin{proof}
Let $O\subset\R^d$ be a bounded Lipschitz domain, and $k\in \N$. 
By the definition of $\Gamma$-convergence for $J^k$ on $W^{1,p}_0(O;\R^m)$ and the convergence of infima with Dirichlet boundary conditions, for all $\Lambda$ and  $\omega\in\Omega_1$
we have
\begin{align}\label{eq:seconddefwbarkesp}
\overline V^k(\Lambda)&=\,\frac{1}{|O|}\inf_{\phi\in W^{1,p}_0(O;\R^m)} \int_O\overline V^k(\Lambda +\nabla \phi)\nonumber\\
&=\,\frac{1}{|O|}\lim_{\e \downarrow0}\inf_{\phi\in W^{1,p}_0(O;\R^m)}\fint_{O}V^k(y/\e,\Lambda+\nabla\phi(y),\cdot)dy\,=\,\overline V^k_1(\Lambda).
\end{align}
Likewise, the $\Gamma$-convergence result holds on $\{u \in W^{1,p}(O):\int_O \nabla u=0\}$ so that the identity
$$
\overline V^k(\Lambda)=\overline V^k_2(\Lambda)
$$
also follows from the convergence of infima.
Since Lemma~\ref{lem:ineqhomog} (applied to $V^k$ instead of $V$) yields $\overline V_{2}^k(\Lambda)\le\overline V_3^k(\Lambda)$, it remains to prove that $\overline V_3^k(\Lambda)\le\overline V_1^k(\Lambda)$ for all $\Lambda$.

Let $O'\subset\R^d$ be a bounded domain. By the coercivity and the lower semicontinuity of the integral functional $J^k$ (which follow from the growth condition \eqref{eq:SGC-trun} and the convexity of $V^k$), there exists a minimizer $\zeta\in \Ld^\infty(\Omega;W^{1,p}_0(O';\R^m))$ (where measurability follows from Proposition~\ref{prop:selmes}) such that, for almost all $\omega$,
$$
\fint_{O'}V^k(y,\Lambda+\nabla\zeta(y,\omega),\omega)dy=\inf_{\phi\in W^{1,p}_0({O'};\R^m)}\fint_{O'}V^k(y,\Lambda+\nabla\phi(y),\omega)dy.
$$
Set
$$
\xi(x,\omega):=\frac1{|O'|}\int_{\R^d}\zeta(x+z,\tau_{z}\omega)dz=\fint_{-x+O'}\zeta(x+z,\tau_{z}\omega)dz.
$$
Clearly, $\xi$ is well-defined and stationary, belongs to $W^{1,p}(\Omega;\R^m)$, and
$$
\nabla \xi(x,\omega)=\fint_{-x+O'}\nabla\zeta(x+z,\tau_z\omega)dz.
$$
Hence
$$
\overline V^k_3(\Lambda)\le\E\left[  V^k(0,\Lambda+\nabla \xi(0,\cdot),\cdot)\right]=\E\left[ V^k\left(0,\Lambda+\fint_{O'}\nabla\zeta(z,\tau_z\cdot)dz,\,\cdot\right)\right],
$$
and by convexity of $V^k$
$$
\overline V^k_3(\Lambda)\le \E\left[\fint_{O'} V^k(0,\Lambda+\nabla\zeta(z,\tau_z\cdot),\cdot)dz\right].
$$
By stationarity and the Fubini theorem, we may conclude
$$
\overline V^k_3(\Lambda)\le \fint_{O'}\E[V^k(z,\Lambda+\nabla\zeta(z,\cdot),\cdot)]dz=\E\left[\inf_{\phi\in W^{1,p}_0(O';\R^m)}\fint_{O'}V^k(y,\Lambda+\nabla\phi(y),\cdot)dy\right].
$$
With $O':=O/\e$, the claim $\overline V^k_3(\Lambda)\le\overline V^k_1(\Lambda)$ follows by the dominated convergence theorem and the growth condition from above \eqref{eq:SGC-trun}.
\end{proof}

The following result proves the equivalence between formulas~(i), (ii) and (iii) in Theorem~\ref{th:conv}.

\begin{prop}[Commutation of limits]\label{prop:commut}
For all bounded Lipschitz domains $O\subset\R^d$, and all $\Lambda\in\R^{m\times d}$, we have for almost all $\omega$
$$
\overline V(\Lambda)=\overline P(\Lambda)=\lim_{t\uparrow1}\lim_{\e\downarrow0}P^\e(t\Lambda,\omega;O).
$$
By convexity, for all $\Lambda\notin\partial\dom\overline V$, this takes the form
$\overline V(\Lambda)=\overline P(\Lambda)=\lim_{\e\downarrow0}P^\e(\Lambda,\omega;O)$.
\qed
\end{prop}

\begin{rem}
Although not stated explicitly, this result proves the commutation of truncation and homogenization.
By monotone convergence (cf. the proof of $\overline V\equiv\overline P$ below) we have for all $\e>0$
and almost every $\omega$,
\begin{align*}
&\inf_{\phi\in W^{1,p}(O/\e;\R^m)\atop \fint_{O/\e}\nabla\phi=0}\fint_{O/\e}V(y,\Lambda+\nabla\phi(y),\omega)dy=\sup_k \inf_{\phi\in W^{1,p}_0(O/\e;\R^m)\atop \fint_{O/\e}\nabla\phi=0}\fint_{O/\e}V^{k}(y,\Lambda+\nabla\phi(y),\omega)dy\,
\end{align*}
so that Proposition~\ref{prop:commut}, combined with~\eqref{eq:equivk3} in Lemma~\ref{lem:equivk}, yields
the desired commutation result
\begin{multline*}
\lim_{\e\downarrow0}\sup_k\inf_{\phi\in W^{1,p}(O/\e;\R^m)\atop\fint_{O/\e}\nabla\phi=0}\fint_{O/\e}V^{k}(y,\Lambda+\nabla\phi(y),\omega)dy
\\=\overline V(\Lambda)
=\sup_k\overline V^k(\Lambda)
=\sup_k\lim_{\e\downarrow0}\inf_{\phi\in W^{1,p}(O/\e;\R^m)\atop \fint_{O/\e}\nabla\phi=0}\fint_{O/\e}V^{k}(y,\Lambda+\nabla\phi(y),\omega)dy.
\end{multline*}
\qed
\end{rem}

\begin{proof}[Proof of Proposition~\ref{prop:commut}]
We split the proof into two steps.

\medskip

\step{1} Proof of $\overline V\equiv \overline P$.

Let $\Lambda\in\dom\overline V$.
By~\eqref{eq:equivk4} in Lemma~\ref{lem:equivk}, for all $k$,
$$
\overline V^k(\Lambda)=\inf_{f\in F^p_\pot(\Omega)^m}\E[ V^k(0,\Lambda+f)].
$$
By convexity, $f \mapsto \E[ V^k(0,\Lambda+f)]$ is lower semicontinuous on $F^p_\pot(\Omega)^m$, and by coercivity, the infimum is attained.
Hence there exists $g_k\in F^p_\pot(\Omega)^m$ such that
\begin{align*}
\overline V^k(\Lambda)=\E[ V^k(0,\Lambda +g_k)].
\end{align*}
By the uniform growth condition from below \eqref{eq:SGC-trun}, $(g_k)_k$ is bounded in $\Ld^p(\Omega;\R^{m\times d})$:
$$
\frac1C2^{-p+1}\E[| g_k|^p]-\frac1C |\Lambda|^p\le \frac1C\E[|\Lambda+ g_k|^p]\le \E[ V^k(0,\Lambda+ g_k)]= \overline V^k(\Lambda)\le\overline V(\Lambda).
$$
Let $g\in F^{p}_\pot(\Omega)^m$ be a cluster point of $(g_{k})$ for the weak convergence of $\Ld^{p}(\Omega;\R^{m\times d})$.
We have along the subsequence
$$
\overline V(\Lambda)=\sup_k\overline V^k(\Lambda)=\lim_{k \uparrow\infty}\E[ V^{k}(0,\Lambda+ g_{k})].
$$
Since $k\mapsto V^k$ is increasing and $f\mapsto \E[V^k(0,\Lambda+f)]$ is lower semicontinuous for the weak convergence of $\Ld^{p}(\Omega;\R^{m\times d})$, this yields for all $\ell$
$$
\overline V(\Lambda)\ge\liminf_{k\uparrow\infty}\E[ V^\ell(0,\Lambda+ g_{k})]\ge\E[ V^\ell(0,\Lambda+ g)].
$$
We then conclude by monotone convergence that
$$
\overline V(\Lambda)\ge\E[ V(0,\Lambda+ g)]\ge\inf_{f\in F^{p}_\pot(\Omega)^m}\E[ V(0,\Lambda+f)]= \overline P(\Lambda).
$$
For $\Lambda \notin \dom\overline V$, the above inequality is trivial so that $\overline V(\Lambda)\ge \overline P(\Lambda)$ for all $\Lambda \in \R^{m\times d}$.
For the converse inequality, note that for all $\Lambda$,
$$
\overline P(\Lambda)\ge\sup_k\inf_{f\in F^p_\pot(\Omega)^m}\E[ V^k(0,\Lambda+f)]=\sup_k\overline V^k(\Lambda)=\overline V(\Lambda).
$$
Hence, $\overline V\equiv \overline P$, as claimed. 

\medskip

\step{2} Proof of 
$
\lim_{t\uparrow1}\lim_{\e\downarrow0}P^\e(t\Lambda,\omega;O)=\overline V(\Lambda).
$

Since for $\Lambda\in\dom\overline V$ and $t\in[0,1)$,  $t\Lambda\in \inter\dom\overline V$, Lemma~\ref{lem:ineqhomog} and Step~1 yield
for almost all $\omega$
\begin{align}\label{eq:wbarpeps}
\overline V(t\Lambda)=\overline P(t\Lambda) \ge \limsup_{\e\downarrow0} P^\e(t\Lambda,\omega;O).
\end{align}
By~\eqref{eq:equivk3} in Lemma~\ref{lem:equivk}, for all $\Lambda\in\R^{m\times d}$ and almost all $\omega$,
\begin{align}\label{eq:wbarpeps2}
\liminf_{\e\downarrow0}P^\e(\Lambda,\omega;O)&\ge\sup_k \lim_{\e\downarrow0}\inf_{\phi\in W^{1,p}(O/\e;\R^m)\atop \fint_{O/\e}\nabla\phi=0}\fint_{O/\e}V^k(y,\Lambda+\nabla\phi(y),\omega)dy= \sup_k\overline V^k(\Lambda)=\overline V(\Lambda).
\end{align}
Combined with~\eqref{eq:wbarpeps}, this yields $\lim_{\e}P^\e(t\Lambda,\omega;O)=\overline V(t\Lambda)$ for almost all $\omega$, for all $\Lambda\in\dom\overline V$ and $t\in[0,1)$.
By convexity and lower semicontinuity of $\overline V$, this implies for all $\Lambda\in\dom\overline V$ 
\begin{align}\label{eq:equivwbarpeps}
\lim_{t\uparrow1}\lim_{\e\downarrow0}P^\e(t\Lambda,\omega;O)=\lim_{t\uparrow1}\overline V(t\Lambda)=\overline V(\Lambda),
\end{align}
and~\eqref{eq:wbarpeps2} ensures that this equality also holds for $\Lambda\notin \dom\overline V$. By convexity and by~\eqref{eq:equivwbarpeps}, the function $\Lambda\mapsto \lim_{\e}P^\e(\Lambda,\omega;O)$ is continuous outside $\partial\dom\overline V$, so that the limit $t\uparrow1$ can be omitted for $\Lambda\notin\partial\dom\overline V$.
\end{proof}

\subsection{Proof of Theorem~\ref{th:conv}: $\Gamma$-convergence with Neumann boundary data}

It only remains to prove the $\Gamma$-$\limsup$ inequality.

\begin{prop}[$\Gamma$-$\limsup$ inequality with Neumann boundary data]\label{prop:gammasupN}
Assume $p>d$. There exists a subset $\Omega'\subset\Omega_1$ of maximal probability with the following property: for all $\omega\in\Omega'$, all bounded Lipschitz domains $O\subset\R^d$, and all $u\in W^{1,p}(O;\R^m)$, there exists a sequence $(u_\e)_\e\subset W^{1,p}(O;\R^m)$ such that $u_\e\cvf{}u$ in $W^{1,p}(O;\R^m)$ and $J_\e(u_\e,\omega;O)\to J(u;O)$.
\qed
\end{prop}

\begin{proof}
We split the proof into three steps. We first treat the case of affine functions, then the case of continuous piecewise affine functions, and finally the general case.
The novelty of our approach is the careful gluing argument needed to pass from affine to piecewise affine functions.

\medskip

\step{1} Recovery sequence for affine functions.

In this step, we consider the case when $u=\Lambda\cdot x$ is an affine function. More precisely, we prove the existence of a subset $\Omega'\subset\Omega_1$ of maximal probability with the following property: given a bounded Lipschitz domain $O\subset\R^d$, for all $\omega\in\Omega'$ and all $\Lambda\in\inter\dom \overline V$, there exists a sequence $(u_{\Lambda,\e}^\omega)_\e\subset W^{1,p}(O;\R^m)$ with $u_{\Lambda,\e}^\omega\cvf{}\Lambda\cdot x$ weakly in $W^{1,p}(O;\R^m)$ such that, for all Lipschitz subdomains $O'\subset O$, we have $J_\e(u_{\Lambda,\e}^\omega,\omega;O')\to J(\Lambda\cdot x;O')$. By Corollary~\ref{cor:localizlim}, it suffices to prove this for $O'=O$.

By Lemma~\ref{lem:sublincor} and Proposition~\ref{prop:commut}, there exists a sequence $\varphi_{\Lambda}\in\Mes(\Omega;W^{1,p}_\loc(\R^d;\R^m))$ such that, for all $\omega\in\Omega_\Lambda$, we have $\e\varphi_{\Lambda}(\cdot/\e,\omega)\cvf{}0$ weakly in $W^{1,p}(O;\R^m)$
 and, by the Birkhoff-Khinchin ergodic theorem in the form of~\eqref{eq:ergthoml},
$$
\overline V(\Lambda)=\overline P(\Lambda)=\lim_{\e\downarrow0}\fint_{O/\e}V(y,\Lambda+\nabla\varphi_{\Lambda}(y,\omega),\omega)dy.
$$
In particular, by a change of variables, this yields
\begin{align*}
J(\Lambda\cdot x;O)=|O|\overline V(\Lambda)&=\lim_{\e\downarrow0}J_\e(\Lambda\cdot x+\e \varphi_{\Lambda}(\cdot/\e,\omega),\omega;O).
\end{align*}
The function $u_{\e}^{\Lambda,\omega}(x):=\Lambda\cdot x+\e\varphi_{\Lambda}(x/\e,\omega)$ thus satisfies $u_\e^{\Lambda,\omega}\cvf{}\Lambda\cdot x$ in $W^{1,p}(O;\R^m)$ and $J_\e(u_\e^{\Lambda,\omega},\omega;O)\to J(\Lambda\cdot x;O)$ as $\e\downarrow0$, for all $\omega\in\Omega_\Lambda$.

We then define $\Omega'\subset\Omega_1$ as the (countable) intersection of all $\Omega_\Lambda$'s with $\Lambda\in\Q^{m\times d}\cap \inter\dom\overline V$, which is still of maximal probability.
Let $\Lambda\in\inter\dom\overline V$ and $\omega\in\Omega'$ be fixed. Choose a sequence $(\Lambda_n)_n\subset\Q^{m\times d}\cap\inter\dom\overline V$ such that $\Lambda_n\to\Lambda$. For all $n$, we have already constructed a sequence $(u_{\e,n}^\omega)_\e\subset W^{1,p}(O;\R^m)$ such that $u_{\e,n}^\omega\rightharpoonup\Lambda_n\cdot x$ in $W^{1,p}(O;\R^m)$ and $J_\e(u_{\e,n}^\omega,\omega;O)\to J(\Lambda_n\cdot x;O)$. Since by convexity,  $\overline V$ is continuous on $\inter\dom\overline V$, we have
\begin{align*}
&\limsup_{n\uparrow\infty}\limsup_{\e\downarrow0}\left(|J_\e(u_{\e,n}^\omega,\omega;O)-J(\Lambda\cdot x;O)|+\|u_{\e,n}^\omega-\Lambda\cdot x\|_{\Ld^p(O;\R^m)}\right)\\
=~&\limsup_{n\uparrow\infty}\left(|J(\Lambda_n\cdot x;O)-J(\Lambda\cdot x;O)|+\|\Lambda_n\cdot x-\Lambda\cdot x\|_{\Ld^p(O;\R^m)}\right)\\
\le~&\limsup_{n\uparrow\infty}\left(|O||\overline V(\Lambda_n)-\overline V(\Lambda)|+C_O|\Lambda_n-\Lambda|\right)=0.
\end{align*}
By the Attouch diagonalization lemma (see~\cite[Corollary~1.16]{Attouch-84}), this implies the existence of a sequence $(v_{\e}^\omega)_\e$ such that $J_\e(v_{\e}^\omega,\omega;O)\to J(\Lambda\cdot x;O)$ and $v_{\e}^\omega\to\Lambda\cdot x$ in $\Ld^p(O;\R^m)$ for all $\omega\in\Omega'$. 
By the $p$-th order lower bound for $V$, we conclude that $v_{\e}^\omega$ converges weakly to $\Lambda\cdot x$ in $W^{1,p}(O;\R^m)$.

\medskip

\step{2} Recovery sequence for continuous piecewise affine functions.

Let $\omega\in\Omega'$, $O\subset \R^d$ be a bounded Lipschitz domain, and $u$ be a continuous piecewise affine function on $O$ such that $\nabla u\in\inter\dom\overline V$ pointwise. We shall prove that there exists a sequence $(u_{\e}^\omega)_\e\subset W^{1,p}(O;\R^m)$ with $u^\omega_{\e}\cvf{}u$ weakly in $W^{1,p}(O;\R^m)$, such that $J_\e(u^\omega_{\e},\omega;O)\to J(u;O)$. For that purpose, the major issue consists in gluing the recovery sequences for the different affine parts together, which requires a particularly careful treatment.

Consider the open partition $O=\biguplus_{l=1}^kO_l$ associated with $u$ (note that the $O_l$'s have piecewise flat boundary outside $\partial O$), and define $c_l+\Lambda_l\cdot x:= u|_{O_l}$, with $\Lambda_l\in\inter\dom\overline V$, for all $1\le l\le k$. 
Let $\calM:=(\bigcup_{l=1}^k\partial O_l)\setminus\partial O$ be the interior boundary of the partition of $O$, and for all $r>0$ set $\calM_r:=(\calM+B_r)\cap O=\{x\in O:\dist(x,\calM)<r\}$, the $r$-neighborhood of this interior boundary. 
By Proposition~\ref{prop:approxaffine}, for all $0<\kappa\le 1$ and $r>0$, there exists a continuous piecewise affine function $u_{\kappa,r}$ on $O$ with the following properties:
\begin{enumerate}[(i)]
\item $\nabla u_{\kappa,r}=\nabla u$ pointwise on $O\setminus \calM_r$, and 
\begin{align}\label{eq:convkaprhouappr}
\limsup_{r\downarrow0}\sup_{0<\kappa\le 1}\|u_{\kappa,r}-u\|_{\Ld^\infty(O)}=0;
\end{align}
\item $\nabla u_{\kappa,r}\in \conv(\{\Lambda_l:1\le l\le k\})\Subset\inter\dom\overline V$ pointwise (where $\conv(\cdot)$ denotes the convex hull);
\item denoting by $O:=\biguplus_{l=1}^{n_{\kappa,r}}O_{\kappa,r}^l$ the open partition associated with $u_{\kappa,r}$, and denoting $c^l_{\kappa,r}+\Lambda^l_{\kappa,r}\cdot x:= u_{\kappa,r}|_{O^l_{\kappa,r}}$ for all~$l$, then, for any $i,j$ with $\partial O^i_{\kappa,r}\cap\partial O^j_{\kappa,r}\ne\varnothing$, we have $|\Lambda^i_{\kappa,r}-\Lambda^j_{\kappa,r}|\le \kappa$.
\end{enumerate}
We shall approximate $u$ with these refined continuous piecewise affine functions $u_{\kappa,r}$ having smoother variations; in the sequel, we shall successively take the limits $\kappa\downarrow0$ and $r\downarrow0$.

Since $\omega\in\Omega'$ and $O\subset\R^d$ are fixed in the argument, we drop them from our notation.
Fix $\kappa,r>0$. 
By Step~1, for all $1\le i\le n_{\kappa,r}$ there exists a sequence $(u_{\e,\kappa,r}^{i})_\e\subset W^{1,p}_\loc(\R^d;\R^m)$ with $u_{\e,\kappa,r}^{i}\cvf{} c_{\kappa,r}^i+\Lambda_{\kappa,r}^i\cdot x$ in $W^{1,p}_{\loc}(\R^d;\R^m)$ and such that, for all Lipschitz subdomains $O'\subset O$, we have $J_\e(u_{\e,\kappa,r}^{i},\omega;O')\to J(\Lambda_{\kappa,r}^i\cdot x;O')$. For all $\eta>0$ and all $1\le i\le n_{\kappa,r}$, define the sets
\begin{align*}
O^{i+}_{\kappa,r,\eta}:=&\{x\in O:\dist(x,O^{i}_{\kappa,r})<\eta\}=O\cap(O^i_{\kappa,r}+B_{\eta}),\\
O^{i-}_{\kappa,r,\eta}:=&\{x\in O_{\kappa,r}^i:\dist(x,\partial O^{i}_{\kappa,r})>\eta\}
\end{align*}
Let then $\sum_{i=1}^{n_{\kappa,r}}\chi_{\kappa,r,\eta}^{i}=1$ be a partition of unity on $O$, where, for all $1\le i\le n_{\kappa,r}$, the smooth cut-off function $\chi_{\kappa,r,\eta}^{i}$ has values in $[0,1]$, equals $1$ on $O^{i-}_{\kappa,r,\eta}$ and vanishes outside $O^{i+}_{\kappa,r,\eta}$, and satisfies the bound $|\nabla\chi_{\kappa,r,\eta}^{i}|\le C'/\eta$ pointwise for some constant $C'>0$. We now set
$$
u_{\e,\kappa,r,\eta}:=u_{\kappa,r}+\sum_{i=1}^{n_{\kappa,r}}(u_{\e,\kappa,r}^{i}-(c^i_{\kappa,r}+\Lambda_{\kappa,r}^i\cdot x))~\chi_{\kappa,r,\eta}^{i}.
$$
By the Sobolev compact embedding for $p>d$, we have $u_{\e,\kappa,r}^{i}\to c_{\kappa,r}^i+\Lambda_{\kappa,r}^i\cdot x$ in $\Ld^\infty(O;\R^m)$ as $\e\downarrow0$, and hence $\limsup_{\eta}\limsup_{\e}\|u_{\e,\kappa,r,\eta}- u_{\kappa,r}\|_{\Ld^\infty(O)}=0$, so that~\eqref{eq:convkaprhouappr} yields
\begin{align}\label{eq:Linftyconvaffine}
\lim_{t\uparrow1}\limsup_{r\downarrow0}\limsup_{\kappa\downarrow0}\limsup_{\eta\downarrow0}\limsup_{\e\downarrow0}\|tu_{\e,\kappa,r,\eta}- u\|_{\Ld^\infty(O)}=0.
\end{align}

Let us now evaluate the integral functional $J_\e(\cdot,\omega;O)$ at $tu_{\e,\kappa,r,\eta}$ for $t\in[0,1)$.
Since 
\begin{multline*}
t\nabla u_{\e,\kappa,r,\eta}=\sum_{i=1}^{n_{\kappa,r}}t\chi_{\kappa,r,\eta}^{i}\nabla u_{\e,\kappa,r}^{i}\\
+(1-t)\frac t{1-t}\sum_{i=1}^{n_{\kappa,r}}\left((u_{\e,\kappa,r}^{i}-(c^i_{\kappa,r}+\Lambda_{\kappa,r}^i\cdot x))\nabla\chi_{\kappa,r,\eta}^{i}+(\nabla u_{\kappa,r}-\Lambda_{\kappa,r}^i)\chi_{\kappa,r,\eta}^i\right),
\end{multline*}
and $(1-t)+\sum_{i=1}^{n_{\kappa,r}}t\chi_{\kappa,r,\eta}^{i}=1$, we have by convexity and non-negativity of $V$
\begin{align}
J_\e(tu_{\e,\kappa,r,\eta},\omega;O)&\le(1-t)E_{\e,\kappa,r,\eta,t}+t\sum_{i=1}^{n_{\kappa,r}}\int_{O^{i+}_{\kappa,r,\eta}} \chi_{\kappa,r,\eta}^{i}(y)V(y/\e,\nabla u^{i}_{\e,\kappa,r}(y),\omega)dy\nonumber\\
&\le(1-t)E_{\e,\kappa,r,\eta,t}+\sum_{i=1}^{n_{\kappa,r}}J_\e(u^{i}_{\e,\kappa,r},\omega;O^{i+}_{\kappa,r,\eta}),\label{eq:boundconvglueaff}
\end{align}
where the error term reads
\begin{multline*}
E_{\e,\kappa,r,\eta,t}=\int_OV\bigg(y/\e,\frac t{1-t}\sum_{i=1}^{n_{\kappa,r}}\Big((u_{\e,\kappa,r}^{i}(y)-(c^i_{\kappa,r}+\Lambda_{\kappa,r}^i\cdot y))\nabla\chi_{\kappa,r,\eta}^{i}(y)\\
+(\nabla u_{\kappa,r}(y)-\Lambda_{\kappa,r}^i)\chi_{\kappa,r,\eta}^i(y)\Big),\omega\bigg)dy.
\end{multline*}
For all $i$, set $N_{\kappa,r,\eta}^{i}:=\{j: j\ne i,\,O^{j+}_{\kappa,r,\eta}\cap O^{i+}_{\kappa,r,\eta}\ne\varnothing\}$.
We then rewrite the argument of the energy density in the error term as
\begin{align*}
S_{\e,\kappa,r,\eta}(y):=~&\left|\sum_{i=1}^{n_{\kappa,r}}\Big((u_{\e,\kappa,r}^{i}(y)-(c^i_{\kappa,r}+\Lambda_{\kappa,r}^i\cdot y))\nabla\chi_{\kappa,r,\eta}^{i}(y)+(\nabla u_{\kappa,r}(y)-\Lambda_{\kappa,r}^i)\chi_{\kappa,r,\eta}^i(y)\Big)\right|\\
\le~&\frac{C'}\eta\sum_{i=1}^{n_{\kappa,r}}\|u^{i}_{\e,\kappa,r}-(c^i_{\kappa,r}+\Lambda_{\kappa,r}^i\cdot x)\|_{\Ld^\infty(O)}+\sup_{1\le i\le n_{\kappa,r}}\sup_{j\in N^i_{\kappa,r,\eta}}|\Lambda^j_{\kappa,r}-\Lambda^i_{\kappa,r}|.
\end{align*}
Since by definition $\limsup_{\eta\downarrow0}\sup_{j\in N^i_{\kappa,r,\eta}}|\Lambda^j_{\kappa,r}-\Lambda^i_{\kappa,r}|\le\kappa$ for all $i$, we have
$$
\limsup_{\kappa\downarrow0}\limsup_{\eta\downarrow0}\limsup_{\e\downarrow0} S_{\e,\kappa,r,\eta}(y)=0
$$ 
for all $r,\eta>0$. 
By assumption, there exists $\delta>0$ such that $\adh B_\delta\subset\inter\dom M$. 
Hence, for all $r,t>0$ there exists $\kappa_{r,t}>0$ such that for all $0<\kappa<\kappa_{r,t}$ there exists $\eta_{\kappa,r}>0$ such that for all $0<\eta<\eta_{\kappa,r}$ there exists $\e_{\kappa,r,\eta,t}>0$ with the following property: for all $0<\e<\e_{\kappa,r,\eta,t}$, we have
$$
\left\|\frac{t}{1-t}S_{\e,\kappa,r,\eta}\right\|_{\Ld^\infty(O)}<\delta.
$$
This yields the bound
$$
E_{\e,\kappa,r,\eta,t}\le |O|\sup_{|\Lambda'|<\delta}M(\Lambda') <\infty,
$$
and proves
$$
\lim_{t\uparrow1}\limsup_{r\downarrow0}\limsup_{\kappa\downarrow0}\limsup_{\eta\downarrow0}\limsup_{\e\downarrow0}(1-t)E_{\e,\kappa,r,\eta,t}=0,
$$
so that~\eqref{eq:boundconvglueaff} turns into
\begin{align}\nonumber
&\limsup_{t\uparrow1}\limsup_{r\downarrow0}\limsup_{\kappa\downarrow0}\limsup_{\eta\downarrow0}\limsup_{\e\downarrow0}J_\e(tu_{\e,\kappa,r,\eta},\omega;O)\\
&\hspace{4cm}\le\limsup_{r\downarrow0}\limsup_{\kappa\downarrow0}\sum_{i=1}^{n_{\kappa,r}}\limsup_{\eta\downarrow0}\limsup_{\e\downarrow0}J_\e(u^{i}_{\e,\kappa,r},\omega;O^{i+}_{\kappa,r,\eta}).\label{eq:boundlimsupglueaff}
\end{align}
For all $i$, we have by construction
$$
\lim_{\e\downarrow0}J_\e(u^{i}_{\e,\kappa,r},\omega;O^{i+}_{\kappa,r,\eta})=|O^{i+}_{\kappa,r,\eta}|\overline V(\Lambda^i_{\kappa,r}),
$$
so that, by definition of $O^{i+}_{\kappa,r,\eta}$,
$$
\lim_{\eta\downarrow0}\lim_{\e\downarrow0}J_\e(u^{i}_{\e,\kappa,r},\omega;O^{i+}_{\kappa,r,\eta})=|O^{i}_{\kappa,r}|\overline V(\Lambda^i_{\kappa,r}).
$$
Hence, summing over $i$, $1\le i\le n_{\kappa,r}$, yields
\begin{equation*}
\sum_{i=1}^{n_{\kappa,r}}\lim_{\eta\downarrow0}\lim_{\e\downarrow0}J_\e(u^{i}_{\e,\kappa,r},\omega;O^{i+}_{\kappa,r,\eta})\,=\,\sum_{i=1}^{n_{\kappa,r}}|O^{i}_{\kappa,r}|\overline V(\Lambda^i_{\kappa,r})
.
\end{equation*}
On the one hand, $\nabla u_{\kappa,r}=\nabla u$ holds on $O\setminus \calM_r$. On the other hand, for all $i,\kappa,r$, $\Lambda^i_{\kappa,r} \in K:=\conv(\{\Lambda_l:1\le l\le k\})$, which is a compact subset of $\inter\dom\overline V$.
Using in addition the non-negativity of the energy density, one may then turn the above equality into
\begin{align*}
\sum_{i=1}^{n_{\kappa,r}}\lim_{\eta\downarrow0}\lim_{\e\downarrow0}J_\e(u^{i}_{\e,\kappa,r},\omega;O^{i+}_{\kappa,r,\eta})=J(u_{\kappa,r};O)&=J(u;O\setminus \calM_r)+J(u_{\kappa,r};\calM_r)\\
&\le J(u;O)+|\calM_r|\sup_K \overline V.
\end{align*}
Combined with~\eqref{eq:boundlimsupglueaff}, this yields
\begin{align}\label{eq:gammalimsupglueup}
\limsup_{t\uparrow1}\limsup_{r\downarrow0}\limsup_{\kappa\downarrow0}\limsup_{\eta\downarrow0}\limsup_{\e\downarrow0}J_\e(tu_{\e,\kappa,r,\eta},\omega;O)\le J(u;O).
\end{align}
We are now in position to conclude.
By coercivity of $V$, the sequence $\nabla (tu_{\e,\kappa,r,\eta})$ is bounded in $\Ld^p(O;\R^{m\times d})$. Combined with \eqref{eq:Linftyconvaffine} (convergence in $\Ld^\infty(O;\R^m)$), this shows that any weakly converging subsequence of $(tu_{\e,\kappa,r,\eta})_{\e,\eta,\kappa,r,t}$  in  $W^{1,p}(O;\R^m)$ converges to $u$.
Hence the $\Gamma$-$\liminf$ inequality of Proposition~\ref{prop:gammaliminfneu} yields
$$
\liminf_{t\uparrow1}\liminf_{r\downarrow0}\liminf_{\kappa\downarrow0}\liminf_{\eta\downarrow0}\liminf_{\e\downarrow0}J_\e(tu_{\e,\kappa,r,\eta},\omega;O)\ge J(u;O).
$$
These last two inequalities combine to
$$
\limsup_{t\uparrow1}\limsup_{r\downarrow0}\limsup_{\kappa\downarrow0}\limsup_{\eta\downarrow0}\limsup_{\e\downarrow0}\left(|J_\e(tu_{\e,\kappa,r,\eta},\omega;O)-J(u;O)|+\|tu_{\e,\kappa,r,\eta}-u\|_{\Ld^p(O;\R^m)}\right)=0,
$$
and we conclude as before by the Attouch diagonalization lemma.

\medskip

\step{3} Recovery sequence for general functions.

We claim that, for all $\omega\in\Omega'$, all bounded Lipschitz domains $O\subset\R^d$ and all $u\in W^{1,p}(O;\R^m)$, there is a sequence $(u_\e)_\e\subset W^{1,p}(O;\R^m)$ with $u_\e\cvf{}u$ in $W^{1,p}(O;\R^m)$ and $J_\e(u_\e,\omega;O)\to J(u;O)$. 
By the locality of recovery sequences (cf. Corollary~\ref{cor:localizlim}), we may consider that $O$ is a ball of $\R^d$ to which we may apply the approximation result of Proposition~\ref{prop:approxw}.
By the $\Gamma$-$\liminf$ inequality of Proposition~\ref{prop:gammaliminfneu}, we can further assume that $u\in W^{1,p}(O;\R^m)$ satisfies
$$
J(u;O)=\int_O\overline V(\nabla u(y))dy<\infty,
$$
so that $\nabla u\in\dom\overline V$ almost everywhere. Let $u$ be such a function and let $\omega\in\Omega'$ be fixed.

Since $O$ is bounded, Lipschitz and strongly star-shaped, $\overline V$ is convex, and $0\in\inter\dom\overline V$, Proposition~\ref{prop:approxw}(ii) shows 
that there exists a sequence  $(u_{n})_n$ of continuous piecewise affine functions with $\nabla u_{n}\in\inter\dom\overline V$ pointwise such that $u_{n}\to u$ (strongly) in $W^{1,p}(O;\R^m)$ and $J( u_{n};O)\longrightarrow J(u;O)$ as $n\uparrow\infty$.
By Step~2, for all $n$, there exists a sequence $(u_{\e,n})_\e\subset W^{1,p}(O;\R^m)$ such that $u_{\e,n}\cvf{} u_{n}$ in $W^{1,p}(U;\R^m)$ and $J_\e(u_{\e,n},\omega;O)\to J(u_{n};O)$, as $\e\downarrow0$. In particular,
\begin{multline*}
\lim_{n\uparrow\infty}\lim_{\e\downarrow0}\left(|J_\e(u_{\e,n},\omega;O)-J(u;O)|+\|u_{\e,n}-u\|_{\Ld^p(O;\R^m)}\right)\\
=\lim_{n\uparrow\infty}\left(|J(u_{n};O)-J(u;O)|+\|u_{n}-u\|_{\Ld^p(O;\R^m)}\right)=0.
\end{multline*}
We then conclude as before by the Attouch diagonalization argument.
\end{proof}

\subsection{Proof of Corollary~\ref{cor:dir}(i): lifting Dirichlet boundary data}\label{chap:liftdirbd}
We split the proof into two steps. We first consider the case when $J(\alpha u;O)<\infty$ for some $\alpha>1$, and then turn to 
the case when in addition $\int_O M(\nabla u)<\infty$ or $\int_O M(\alpha\nabla u)<\infty$ for some $\alpha>1$.

\medskip

\step{1} Case when $J(\alpha u;O)<\infty$, for some $\alpha>1$.

As $v\in u+ W^{1,p}_0(O;\R^m)$ and $J(\alpha u;O)<\infty$, Proposition~\ref{prop:approxw}(ii)(a) yields the existence of a sequence $(v_k)_k\subset u+C^\infty_c(O;\R^m)$ with $v_k\to v$ in $W^{1,p}(O;\R^m)$ and $J(v_k;O)\to J(v;O)$. 
For all $r>0$, set $O_r^1:=\{x\in O:\dist(x,\partial O)>2r\}$, $O_r^2:=\{x\in O:\dist(x,\partial O)>r\}$, and choose smooth cut-off functions $\chi_r^1,\chi_r^2$ with the following properties: the functions take values in $[0,1]$, $\chi_r^1$ equals $1$ on $O_r^1$ and $0$ on $\R^d\setminus O_r^2$,  $\chi_r^2$ equals $1$ on $O_r^2$ and $0$ on $\R^d\setminus O$, and $|\nabla\chi_r^1|,|\nabla\chi_r^2|\le C'/r$ for some constant $C'$.
For all $\omega\in\Omega'$, Proposition~\ref{prop:gammasupN} provides sequences $(u_\e^\omega)_\e$ and $(v_{\e,r,k}^\omega)_\e$ in $W^{1,p}(O;\R^m)$ such that $u_\e^\omega\cvf{}u$ and $v_{\e,r,k}^\omega\cvf{}\chi_r^1v_k+(1-\chi_r^1)u$ in $W^{1,p}(O;\R^m)$, and  $J_\e(u_\e^\omega,\omega;O')\to J(u;O')$ and $J_\e(v_{\e,r,k}^\omega,\omega;O')\to J(\chi_r^1v_k+(1-\chi_r^1)u;O')$, for any subdomain $O'\subset O$. We then set $w_{\e,r,k}^\omega:=\chi_r^2v_{\e,r,k}^\omega+(1-\chi_r^2)u_\e^\omega$. Given $t\in[0,1)$, using the decomposition
$$
t\nabla w_{\e,r,k}^\omega=t\chi_r^2\nabla v_{\e,r,k}^\omega+t(1-\chi_r^2)\nabla u_\e^\omega+(1-t)\frac t{1-t}\nabla\chi_r^2(v_{\e,r,k}^\omega-u_\e^\omega),
$$
and convexity, we obtain
\begin{align}\label{eq:inequIeps}
J_\e(tw_{\e,r,k}^\omega,\omega;O)&\le (1-t)E_{\e,r,k,t}^\omega+J_\e(v_{\e,r,k}^\omega,\omega;O)+J_\e(u_\e^\omega,\omega;O\setminus O_r^2),
\end{align}
where the error term reads
$$
E_{\e,r,k,t}^\omega:=\int_O V\left(y/\e,\frac t{1-t}\nabla\chi_r^2(y)(v_{\e,r,k}^\omega(y)-u_\e^\omega(y)),\omega\right)dy.
$$
For all $y\in O\setminus O_r^2$, since $\chi_r^1(y)=0$, we have
\begin{align*}
|v_{\e,r,k}^\omega(y)-u_\e^\omega(y)|&\le \|v_{\e,r,k}^\omega-(\chi_r^1v_k+(1-\chi_r^1)u)\|_{\Ld^\infty(O)}+\|u_\e^\omega-u\|_{\Ld^\infty(O)}.
\end{align*}
By assumption, there is some $\delta>0$ with $\adh B_\delta\subset \inter\dom M$.
Hence, for all fixed $r,k,t$, there exists $\e_{r,k,t}>0$ such that for all $0<\e<\e_{r,k,t}$ we have
$$
\left\|\frac t{1-t}\nabla\chi_r^2(v_{\e,r,k}^\omega-u_\e^\omega)\right\|_{\Ld^\infty(O)}<\delta,
$$
and therefore
$$
\limsup_{\e\downarrow0}E_{\e,r,k,t}^\omega\le |O|\sup_{|\Lambda'|<\delta}M(\Lambda')<\infty.
$$
Inequality~\eqref{eq:inequIeps} then turns into
\begin{multline*}
\limsup_{t\uparrow1}\limsup_{k\uparrow\infty}\limsup_{r\downarrow0}\limsup_{\e\downarrow0}J_\e(tw_{\e,r,k}^\omega,\omega;O)\\
\le~\limsup_{k\uparrow\infty}\limsup_{r\downarrow0}J(\chi_r^1v_k+(1-\chi_r^1)u;O)+\limsup_{r\downarrow0}J(u;O\setminus O_r^2).
\end{multline*}
The second term in the right-hand side vanishes since $J(u;O)<\infty$, and it only remains to study the first term. 
By definition, for fixed $k$, we have $v_k\in u+C_c^\infty(O;\R^m)$, so that for all $r>0$ small enough, $\chi_r^1v_k+(1-\chi_r^1)u=v_k$ pointwise on $O$.
This implies
$$
\limsup_{k\uparrow\infty}\limsup_{r\downarrow0}J(\chi_r^1v_k+(1-\chi_r^1)u;O)=\limsup_{k\uparrow\infty}J(v_k;O)=J(v;O),
$$
and thus
\begin{align*}
\limsup_{t\uparrow1}\limsup_{k\uparrow\infty}\limsup_{r\downarrow0}\limsup_{\e\downarrow0}J_\e(tw_{\e,r,k}^\omega,\omega;O)\le J(v;O).
\end{align*}
Combined with the $\Gamma$-$\liminf$ inequality of Proposition~\ref{prop:gammaliminfneu} and a diagonalization argument, this proves the first part of the statement.

\medskip

\step{2} Cases when $\int_O M(\nabla u)<\infty$ or $\int_O M(\alpha\nabla u)<\infty$ for some $\alpha>1$.
 
If $u$ is chosen in such a way that $\int_OM(\nabla u(y))dy<\infty$, then we can repeat the argument of Step~1 with $u_\e^\omega:= u$, and bound the last term in the right-hand side of~\eqref{eq:inequIeps} by
$$
\limsup_{r\downarrow0}\limsup_{\e\downarrow0}J_\e(u,\omega;O\setminus O_r^2)\le\limsup_{r\downarrow0}\int_{O\setminus O_r^2}M(\nabla u(y))dy=0.
$$
We conclude with the case when $u$ satisfies $\int_OM(\alpha\nabla u(y))dy<\infty$ for some $\alpha>1$. 
Let $v_k,\chi_r^1,\chi_r^2$ be chosen as in Step~1, and let $\omega\in\Omega'$ be fixed. For any $t\in[0,1)$, Proposition~\ref{prop:gammasupN} shows the existence of a sequence $(v_{\e,r,k,t}^\omega)_\e$ in $W^{1,p}(O;\R^m)$ such that $v_{\e,r,k,t}^\omega\cvf{}\chi_r^1v_k+(1-\chi_r^1)u/t$ in $W^{1,p}(O;\R^m)$ and $J_\e(v_{\e,r,k,t}^\omega,\omega;O')\to J(\chi_r^1v_k+(1-\chi_r^1)u/t;O')$, for any subdomain $O'\subset O$. 
Set $w_{\e,r,k,t}^\omega:=\chi_r^2v_{\e,r,k,t}^\omega+(1-\chi_r^2)u/t$. As before, we obtain by convexity
\begin{multline*}
J_\e(tw_{\e,r,k,t}^\omega,\omega)\le J_\e(v_{\e,r,k,t}^\omega,\omega;O)+J_\e(u/t,\omega;O\setminus O_r^2)\\
+(1-t)\int_OM\left(\frac{t}{1-t}\nabla\chi_r^2(y)(v_{\e,r,k,t}^\omega(y)-u(y)/t)\right)dy,
\end{multline*}
and the conclusion then follows, using the convexity once more in the following form: for $t>1/\alpha$,
\begin{align*}
\limsup_{r\downarrow0}\limsup_{\e\downarrow0}J_\e(u/t,\omega;O\setminus O_r^2)&\le\limsup_{r\downarrow0}\int_{O\setminus O_r^2}M(\nabla u(y)/t)dy\\
&\le \frac1{\alpha t}\limsup_{r\downarrow0}\int_{O\setminus O_r^2}M(\alpha\nabla u(y))dy+\limsup_{r\downarrow0}|O\setminus O_r^2|M(0)=0.
\end{align*}
This completes the proof.\qed

\begin{rem}
As can be seen in the proof, the assumption that $J(\alpha u;O)<\infty$ can be relaxed to $J(\alpha u;O')<\infty$ for some open neighborhood $O'\subset O$ of $\partial O$ in $O$.
\qed
\end{rem}

\subsection{Proof of Corollary~\ref{cor:dir}(ii): soft buffer zone for Dirichlet boundary data}\label{chap:2dirbd}

We split the proof into two steps.
For all $s>0$ and $O\subset \R^d$, we use the notation $O_s:=\{x\in O:\dist(x,\partial O)>s\}$.

\medskip

\step{1} $\Gamma$-$\liminf$ inequality.

Let $\omega\in\Omega'$, let $O\subset\R^d$ be a bounded Lipschitz domain, let $u\in W^{1,p}(O;\R^m)$ with $J(u;O)<\infty$, and let $(u_\e)_\e\subset W^{1,p}(O;\R^m)$ be a sequence with $u_\e\cvf{} u$ in $W^{1,p}(O;\R^m)$. By the $\Gamma$-$\liminf$ inequality for $J_\e$ in Proposition~\ref{prop:gammaliminfneu},
\begin{align*}
\liminf_{\e\downarrow0}J_\e^\eta(u_\e,\omega;O)\ge \liminf_{\e\downarrow0}J_\e(u_\e,\omega;O_\eta)\ge J(u;O_\eta)=\int_{O_\eta}\overline{V}(\nabla u(y))dy,
\end{align*}
that is, using that $\int_{O\setminus O_\eta} \overline{V}(\nabla u(y))dy\to0$ as $\eta\downarrow0$,
\begin{align*}
\liminf_{\eta\downarrow0}\liminf_{\e\downarrow0}J_\e^\eta(u_\e,\omega;O)\ge J(u;O).
\end{align*}

\medskip

\step{2} $\Gamma$-$\limsup$ inequality.

Let $\omega\in\Omega'$, let $O\subset\R^d$ be a bounded Lipschitz domain, and let $u\in W^{1,p}(O;\R^m)$ with $J(u;O)<\infty$. By Proposition~\ref{prop:gammasupN}, there exists a sequence $(w_{\e})_\e\subset W^{1,p}(O;\R^m)$ such that $w_{\e}\cvf{}0$ in $W^{1,p}(O;\R^m)$ and $J_\e(u+w_{\e},\omega;O)\to J(u;O)$. Given $\eta>0$, choose a cut-off function $\chi_\eta$ with values in $[0,1]$, such that $\chi_\eta$ equals $1$ on $O_\eta$ and $0$ outside $O$, and that satisfies $|\nabla\chi_\eta|\le C'/\eta$ for some constant $C'>0$.
Set $v_{\e,\eta}:=\chi_{\eta}w_{\e}\in W^{1,p}_0(O;\R^m)$. For all $t\in[0,1)$, we have
$$
t\nabla u+t\nabla v_{\e,\eta}=t\chi_\eta \nabla (u+w_{\e})+t(1-\chi_\eta) \nabla u+(1-t)\frac{t}{1-t}w_{\e}\nabla\chi_\eta,
$$
so that by convexity and the definition of $V_\e^{O,\eta}$,
\begin{align}
J_\e^\eta(tu+tv_{\e,\eta},\omega;O)&\le (1-t)E_{\e,\eta,t}+J_\e^\eta(u+w_{\e},\omega;O)+\int_O(1-\chi_{\eta}(y)) V^{O,\eta}_\e(y,\nabla u(y),\omega)dy\nonumber\\
&\le (1-t)E_{\e,\eta,t}+J_\e(u+w_{\e},\omega;O)+\int_{O\setminus O_\eta}|\nabla u(y)|^pdy,\label{eq:boundenergydirmod}
\end{align}
where the error is defined by
\begin{align}
\label{eq:deferrorcorspeed}
E_{\e,\eta,t}:=~&\int_OV^{O,\eta}_\e\left(y,\frac{t}{1-t}w_\e(y)\nabla\chi_{\eta}(y),\omega\right)dy\nonumber\\
\le~&|O_\eta|M(0)+\int_{O\setminus O_{\eta}}\left|\frac{t}{1-t}w_\e(y)\nabla\chi_{\eta}(y)\right|^pdy.
\end{align}
By the Rellich-Kondrachov theorem, $w_{\e}\to 0$ (strongly) in $\Ld^p(O)$, so that $\limsup_\e E_{\e,\eta,t}\le |O|M(0)$ for all $t,\eta$.
Passing to the limit in inequality~\eqref{eq:boundenergydirmod} thus yields
\begin{align*}
\limsup_{t\uparrow1}\limsup_{\eta\downarrow0}\limsup_{\e\downarrow0}J_\e^\eta(tu+tv_{\e},\omega;O,\eta)&\le\limsup_{\e\downarrow0}J_\e(u+w_{\e},\omega;O)=J(u;O).
\end{align*}
We then conclude by the same diagonalization argument as before and Step~1. This proves the first part of the statement.

Now consider the case when $u$ satisfies $J(\alpha u;O)<\infty$ for some $\alpha>1$. Then, for all $t\in[0,1)$, Proposition~\ref{prop:gammasupN} provides a sequence $(w_{\e,t})_\e\subset W^{1,p}(O;\R^m)$ such that $w_{\e,t}\cvf{}0$ in $W^{1,p}(O;\R^m)$ and $J_\e(u/t+w_{\e,t},\omega;O)\to J(u/t;O)$ as $\e\downarrow0$. Define $v_{\e,t,\eta}:=\chi_\eta w_{\e,t}$, where $\chi_\eta$ is the same cut-off function as above. We then have now
$$
\nabla u+t\nabla v_{\e,t,\eta}=t\chi_\eta \nabla (u/t+w_{\e,t})+t(1-\chi_\eta) \nabla u/t+(1-t)\frac{t}{1-t}w_{\e,t}\nabla\chi_\eta,
$$
so that by convexity and definition of $V^{O,\eta}_\e$,
\begin{align}
J_\e^\eta(u+tv_{\e,\eta},\omega;O)&\le (1-t)E'_{\e,\eta,t}+J_\e(u/t+w_{\e,t},\omega;O)+{t^{1-p}}\int_{O\setminus O_\eta}|\nabla u(y)|^pdy,\label{eq:boundenergydirmod2}
\end{align}
where the error is defined by
\begin{align*}
E'_{\e,\eta,t}:=&|O_\eta|M(0)+\int_{O\setminus O_{\eta}}\left|\frac{t}{1-t}w_{\e,t}(y)\nabla\chi_{\eta}(y)\right|^pdy.
\end{align*}
By the Rellich-Kondrachov theorem, $w_{\e,t}\to 0$ (strongly) in $\Ld^p(O)$ for all $t$, so that  $\limsup_\e E'_{\e,\eta,t}=|O_\eta|M(0)$ for all $t,\eta$.
Passing to the limit in inequality~\eqref{eq:boundenergydirmod2} then yields
\begin{align*}
\limsup_{t\uparrow1}\limsup_{\eta\downarrow0}\limsup_{\e\downarrow0}J_\e^\eta(u+tv_{\e,t,\eta},\omega;O)&\le\limsup_{t\uparrow1}\limsup_{\e\downarrow0}J_\e(u/t+w_{\e,t},\omega;O)=\limsup_{t\uparrow1}J(u/t;O).
\end{align*}
Since $u$ satisfies $J(\alpha u;O)<\infty$, we deduce by convexity that the map $t\mapsto J(u/t;O)$ is continuous on $(1/\alpha,1]$.
This implies $\limsup_{t\uparrow1}J(u/t;O)=J(u;O)$, and the conclusion follows.\qed

\begin{rem}
As can be seen in the proof, the assumption that $J(\alpha u;O)<\infty$ can be relaxed to $J(\alpha u;O')<\infty$ for some open neighborhood $O'\subset O$ of $\partial O$ in $O$.
\qed
\end{rem}

\section{Proof of the results for nonconvex integrands}\label{sec:nonconv}
In this section, we study the case when $W$ is nonconvex but admits a two-sided estimate by a convex function (which may depend on the space variable) and prove Theorem~\ref{th:nonconv}. Let $W$ be a (nonconvex) $\tau$-stationary normal random integrand, which is further assumed to be ru-usc (in the sense of Definition~\ref{def:ruusc}, with respect to some $\tau$-stationary integrable random field $a$). Up to a translation, for simplicity of notation, we can restrict to the following stronger version of~\eqref{eq:convhomb} and~\eqref{eq:aslowerbound0}: for almost all $\omega$, $y$, we have, for all $\Lambda$,
\begin{align}\label{eq:condnc}
\frac1C|\Lambda|^p\le V(y,\Lambda,\omega)\le W(y,\Lambda,\omega)\le C(V(y,\Lambda,\omega)+1),
\end{align}
for some $C>0$ and $d<p<\infty$, and for some convex $\tau$-stationary normal random integrand $V$. Also assume that $0$ belongs to the interior of the domain of the convex function $M:=\supess_{y,\omega}V(y,\cdot,\omega)$. We can then apply Theorem~\ref{th:conv} to $V$, yielding an homogenized energy density $\overline V$ with the following property: defining
\[J_\e(u,\omega;O)=\int_OV(y/\e,\nabla u(y),\omega)dy,\qquad J(u;O)=\int_O\overline V(\nabla u(y))dy,\]
for almost all $\omega$, the integral functionals $J_\e(\cdot,\omega;O)$ $\Gamma$-converge to $J(\cdot;O)$ on $W^{1,p}(O;\R^m)$, for any bounded Lipschitz domain $O\subset\R^d$.
Let $\Omega_{0}\subset\Omega$ be a subset of maximal probability on which all these assumptions and properties (of $V,W$) are simultaneously pointwise satisfied.

\subsection{Definition of the homogenized energy density}\label{chap:preliminnonconv}
We need to define in this section a candidate for the homogenized energy density $\overline W$. As before, the standard homogenization formula with Dirichlet boundary conditions does not hold because of the generality of the growth conditions considered here. Instead, we use the corrector for the convex problem as a boundary condition for the nonconvex problem, which is indeed admissible because of the two-sided growth condition~\eqref{eq:condnc}.

More precisely, for all $\Lambda\in\R^{m\times d}$, Lemma~\ref{lem:sublincor} yields a function $\varphi_\Lambda\in\Mes(\Omega;W^{1,p}_\loc(\R^d;\R^m))$ such that $\nabla\varphi_\Lambda(0,\cdot)\in F^{p}_{\pot}(\Omega)^m$ and
\[\overline V(\Lambda)=\E[V(0,\Lambda+\nabla\varphi_\Lambda(0,\cdot),\cdot)].\]
Now, for any $t\in[0,1)$, consider the function $\mu^t_\Lambda$ defined by
\begin{align*}
\mu_\Lambda^t(O,\omega):=\inf_{v\in W^{1,p}_0(O;\R^m)}\int_OW(y,t\Lambda+t\nabla\varphi_\Lambda(y,\omega)+\nabla v(y),\omega)dy.
\end{align*}
As this quantity is stationary and subadditive, the Ackoglu-Krengel ergodic theorem implies:

\begin{lem}[Definition of the homogenized energy density]\label{lem:Whomnc}
Let $t\in[0,1)$ be fixed. Then, there exists a function $\overline W_t:\dom\overline V\to[0,\infty)$ such that, for all $\Lambda \in\dom\overline V$, the following holds for almost all $\omega\in\Omega_0$: for all bounded Lipschitz domains $O\subset\R^d$,
\begin{align}\label{eq:defWhomnc}
\overline W_t(\Lambda)=\lim_{\e\downarrow0}\frac{\mu^t_\Lambda(O/\e,\omega)}{|O/\e|},
\end{align}
where convergence also holds for expectations. Now define $\overline W(\Lambda):=\liminf_{t\uparrow1}\liminf_{\Lambda'\to\Lambda}\overline W_t(\Lambda')$ for any $\Lambda\in\dom\overline V$, and further set $\overline W(\Lambda)=\infty$ for all $\Lambda\notin \dom\overline V$. Then, $\overline W$ satisfies $\overline V\le\overline W\le C(1+\overline V)$ on the whole of $\R^{m\times d}$, and for all $\Lambda\in\R^{m\times d}$ the following holds for almost all $\omega$ and all bounded Lipschitz domain $O\subset\R^d$:
\begin{align}\label{eq:defWhomncter}
\overline W(\Lambda)=\liminf_{t\uparrow1}\liminf_{\Lambda'\to\Lambda}\lim_{\e\downarrow0}\frac{\mu^t_{\Lambda'}(O/\e,\omega)}{|O/\e|},
\end{align}
where the $\liminf$ as $t\uparrow1$ can further be restricted to $t\in\Q$. Finally, in the particular case when $W$ is convex, then $\overline W$ coincides with the various definitions for the homogenized integrand as given by Theorem~\ref{th:conv}.\qed
\end{lem}

\begin{proof}
We split the proof into three steps.

\medskip
\step{1} Definition of $\overline W_t(\Lambda)$ and proof of~\eqref{eq:defWhomnc}.

First consider the case when $\Lambda\in\dom\overline V$. Let $t\in[0,1)$ be fixed. The upper bound in~\eqref{eq:condnc} then implies $\E[\mu^t_\Lambda(O,\cdot)]\le C|O|(1+t\overline V(\Lambda)+(1-t)M(0))<\infty$. As the function $\mu_\Lambda^t$ is obviously stationary and subadditive, and as $\mu_\Lambda^t(O,\cdot)$ is measurable by Hypothesis~\ref{hypo:asmeasinf}, the Ackoglu-Krengel subadditive ergodic theorem (see e.g.~\cite[Section~6.2]{Krengel-85}) can be applied and asserts the existence of some $\overline W_t(\Lambda)\in[0,\infty)$ such that, for almost all $\omega$, we have
\begin{align*}
\overline W_t(\Lambda)=\lim_{n\uparrow\infty}\frac{\mu^t_\Lambda(I_n,\omega)}{|I_n|},
\end{align*}
for any regular sequence $(I_n)_n\subset\I:=\{[a,b):a,b\in\Z^d\}$ such that $\lim_{n\uparrow\infty}I_n=\R^d$ (in the usual sense of~\cite[Section~6.2]{Krengel-85}), and moreover this convergence also holds for expectations. In particular, we easily see that the same result must hold for the choice $I_n=nQ_0$, where $Q_0$ is any cube aligned with the axes. Further note that, for all bounded Lipschitz subsets $O'\subset O\subset \R^d$, we can estimate, as $W^{1,p}_0(O'/\e;\R^m)\subset W^{1,p}_0(O/\e;\R^m)$,
\begin{align*}
\e^d\mu_\Lambda^t(O/\e,\omega)&\le \e^d\mu^t_\Lambda(O'/\e,\omega)+\e^d\int_{(O\setminus O')/\e}W(y,t\Lambda+t\nabla\varphi_\Lambda(y,\omega),\omega)dy\\
&\le \e^d\mu^t_\Lambda(O'/\e,\omega)+C|O\setminus O'|\left(1+\fint_{(O\setminus O')/\e}V(y,\Lambda+\nabla\varphi_\Lambda(y,\omega),\omega)dy\right),
\end{align*}
where the last expression in brackets converges to $1+\overline V(\Lambda)<\infty$ as $\e\downarrow0$. Now based on this estimate, an easy approximation argument (see e.g.~\cite[Step~4 of the proof of Theorem~3.1]{Gloria-Penrose-13}) allows us to conclude as follows: for almost all $\omega$ (for all $\omega\in\Omega_\Lambda$, for some subset $\Omega_\Lambda\subset\Omega$ of maximal probability, say), we have for all bounded Lipschitz domains $O\subset\R^d$ and all $t\in(0,1)$
\begin{align}\label{eq:plopthomas}
\lim_{\e\downarrow0}\frac{\mu_\Lambda^t(O/\e,\omega)}{|O/\e|}=\overline W_t(\Lambda).
\end{align}

\medskip
\step2 Definition of $\overline W$ and proof of the bounds $\overline V(\Lambda)\le\overline W(\Lambda)\le C(1+\overline V(\Lambda))$ for $\Lambda\in\dom\overline V$.

Let $\Lambda\in\dom\overline V$ be fixed, and let $O\subset\R^d$ be some bounded Lipschitz domain. We define $\overline W(\Lambda)=\liminf_t\liminf_{\Lambda'\to\Lambda}\overline W_t(\Lambda')$. The bounds $\overline V(\Lambda)\le\overline W(\Lambda)\le C(1+\overline V(\Lambda))$ directly follow from the two-sided estimate~\eqref{eq:condnc} together with the following equality, for almost all $\omega$,
\begin{align}\label{eq:rewriteVhomnc}
\overline V(\Lambda)=\overline V_0(\Lambda,\omega),
\end{align}
where we have defined
\[\overline V_0(\Lambda,\omega):=\liminf_{t\uparrow1}\liminf_{\Lambda'\to\Lambda}\lim_{\e\downarrow0}\inf_{v\in W^{1,p}_0(O;\R^m)}\fint_OV(y/\e,t\Lambda'+t\nabla\varphi_{\Lambda'}(y/\e,\omega)+\nabla v(y),\omega)dy.\]
Let us give the argument for~\eqref{eq:rewriteVhomnc}. On the one hand, we can estimate
\begin{align*}
\overline V_0(\Lambda,\omega)\ge\liminf_{t\uparrow1}\liminf_{\Lambda'\to\Lambda}\lim_{\e\downarrow0}\inf_{v\in W^{1,p}(O;\R^m)\atop\fint_O\nabla v=0}\fint_OV(y/\e,t\Lambda'+t\Lambda_\e'(\omega)+\nabla v(y),\omega)dy,
\end{align*}
where we have set $\Lambda_\e'(\omega):=\fint_O\nabla\varphi_{\Lambda'}(\cdot/\e,\omega)$. For almost all $\omega$, since $\Lambda_\e'(\omega)\to0$, we can write, for any $\kappa>0$,
\begin{align*}
\overline V_0(\Lambda,\omega)\ge~&\inf_{\Lambda':|\Lambda'-\Lambda|\le\kappa}\liminf_{t\uparrow1}\liminf_{\e\downarrow0}\inf_{v\in W^{1,p}(O;\R^m)\atop\fint_O\nabla v=0}\fint_OV(y/\e,t\Lambda'+\nabla v(y),\omega)dy,
\end{align*}
so that formula~\eqref{eq:formconvexif} yields $\overline V_0(\Lambda,\omega)\ge\inf_{\Lambda':|\Lambda'-\Lambda|\le\kappa}\overline V(\Lambda')$. Passing to the limit $\kappa\downarrow0$, the lower semicontinuity of $\overline V$ directly gives $\overline V_0(\Lambda,\omega)\ge\overline V(\Lambda)$. On the other hand, the convexity of $V$, the Birkhoff-Khinchin ergodic theorem and the definition of $\varphi_{\Lambda'}$ give for all $t\in[0,1]$ and all $\Lambda'\in\R^{m\times d}$
\begin{align*}
&\lim_{\e\downarrow0}\inf_{v\in W^{1,p}_0(O;\R^m)}\fint_OV(y/\e,t\Lambda'+t\nabla\varphi_{\Lambda'}(y/\e,\omega)+\nabla v(y),\omega)dy\\
\le~& \lim_{\e\downarrow0}\fint_OV(y/\e,\Lambda'+\nabla\varphi_{\Lambda'}(y/\e,\omega),\omega)dy+(1-t)M(0)=\overline V(\Lambda')+(1-t)M(0).
\end{align*}
Passing to the limit $\Lambda'\to\Lambda$ and $t\uparrow1$, and using the lower semicontinuity of $\overline V$ in the form of $\overline V(\Lambda)=\liminf_{\Lambda'\to\Lambda}\overline V(\Lambda')$, this gives $\overline V_0(\Lambda,\omega)\le\overline V(\Lambda)$. The desired identity~\eqref{eq:rewriteVhomnc} is proven.

\medskip
\step3 Case when $\Lambda\notin \dom\overline V$.

For $\Lambda\notin\dom\overline V$, arguing as in Step~2 above, we can estimate, using the pointwise bound $V\le W$,
\[\liminf_{t\uparrow1}\liminf_{\Lambda'\to\Lambda}\liminf_{\e\downarrow0}\inf_{v\in W^{1,p}_0(O;\R^m)}\fint_OW(y/\e,t\Lambda+t\nabla\varphi_\Lambda(y/\e,\omega)+\nabla v(y),\omega)dy\ge \overline V(\Lambda)=\infty,\]
so that~\eqref{eq:defWhomncter} trivially holds with $\overline W(\Lambda):=\infty$. Moreover, the bounds $\overline V\le \overline W\le C(1+\overline V)$ holds as well.
\end{proof}

Although the definition of the homogenized energy density $\overline W(\Lambda)$ may a priori depend on the choice of a corrector $\varphi_\Lambda$, it would follow a posteriori from the $\Gamma$-convergence result that the value of $\overline W(\Lambda)$ is independent of that choice. As this independence will actually be useful in the proof of the $\Gamma$-$\limsup$ inequality (see the proof of Lemma~\ref{lem:lemma}(c)), we display a direct proof.

\begin{lem}[Independence upon the choice of a corrector]\label{lem:Whomncindep}
Assume $p>d$, and let $t\in[0,1)$ and $\Lambda\in\dom\overline V$ be fixed. For almost all $\omega$, given a bounded domain $O\subset\R^d$, if $(u_\e)_\e\subset W^{1,p}(O;\R^m)$ satisfies $\|u_\e\|_{\Ld^\infty(O)}\to0$ and $\limsup_\e \int_{D}V(y/\e,\Lambda+\nabla u_\e(y),\omega)dy\le C_\Lambda|D|$ for all subdomains $D\subset O$ and some constant $C_\Lambda>0$, then, for all Lipschitz subdomains $D\subset O$,
\begin{align}\label{eq:defWhomncbis}
\overline W_t(\Lambda)=\lim_{\e\downarrow0}\inf_{v\in W^{1,p}_0(D;\R^m)}\fint_{D}W(y/\e,t\Lambda+t\nabla u_\e(y)+\nabla v(y),\omega)dy,
\end{align}
where the limit is well-defined.\qed
\end{lem}

\begin{proof}
Let $t\in[0,1)$ and $\Lambda\in\dom\overline V$ be fixed. Let $\omega\in\Omega$ be fixed such that~\eqref{eq:defWhomnc} holds on all bounded Lipschitz domains and such that moreover, for all bounded domains $D\subset\R^d$,
\[\|\e\varphi_\Lambda(\cdot/\e,\omega)\|_{\Ld^\infty(D)}\to0,\qquad \int_DV(y/\e,\Lambda+\nabla \varphi_\Lambda(y/\e,\omega),\omega)dy\to\overline V(\Lambda),\]
which follows almost surely from Lemma~\ref{lem:sublincor}, the Sobolev embedding and the Birkhoff-Khinchin ergodic theorem. Let $(u_\e)_\e$ be as in the statement of the lemma.
Also denote $v_\e^\omega:=\e\varphi_\Lambda(\cdot/\e,\omega)\in W_\loc^{1,p}(\R^d;\R^m)$. By the choice of $\omega$, the sequence $(v_\e^\omega)_\e$ satisfies the same properties as $u_\e$ on any bounded domain, with $C_\Lambda$ replaced by $C'_\Lambda=\overline V(\Lambda)$, and moreover, for all bounded Lipschitz domains $D\subset \R^d$,
\begin{align}\label{eq:defWhomth}
\overline W_t(\Lambda)=\lim_{\e\downarrow0}\inf_{v\in W^{1,p}_0(D;\R^m)}\fint_DW(y/\e,t\Lambda+t\nabla v_\e^\omega(y)+\nabla v(y),\omega)dy.
\end{align}

Let $D\subset O$ be some fixed Lipschitz subdomain. On the one hand, define
\begin{align}\label{eq:defWhomthbis}
\overline W'_t(\Lambda,\omega;D)=\limsup_{\e\downarrow0}\inf_{v\in W^{1,p}_0(D;\R^m)}\fint_DW(y/\e,t\Lambda+t\nabla u_\e(y)+\nabla v(y),\omega)dy.
\end{align}
Given $\eta>0$, set $D_\eta:=\{x\in D:\dist(x,\partial D)>\eta\}$ and consider the difference
\begin{align}
\Delta_{\e,t,\eta}^\omega:=&\inf_{v\in W^{1,p}_0(D;\R^m)}\int_DW(y/\e,t\Lambda+t\nabla u_\e(y)+\nabla v(y),\omega)dy\label{eq:defDeltaetetaom}\\
&\qquad-\inf_{w\in W^{1,p}_0(D_\eta;\R^m)}\int_{D_\eta}W(y/\e,t\Lambda+t\nabla v_\e^\omega(y)+\nabla w(y),\omega)dy.\nonumber
\end{align}
Choose a smooth cut-off function $\chi_\eta$ such that $\chi_\eta$ equals $1$ on $D_\eta$ and vanishes outside $D$, with $|\nabla\chi_\eta|\le C'/\eta$ for some constant $C'>0$, and define $w_{\e,\eta}^\omega:=\chi_\eta v_\e^\omega+(1-\chi_\eta)u_\e$. Restricting the first infimum in~\eqref{eq:defDeltaetetaom} to those $v$'s that are equal to $t(v_\e-u_\e)$ on $\partial D_\eta$, we obtain
\begin{align*}
\Delta^\omega_{\e,t,\eta}&\le \inf_{v\in W^{1,p}_0(D\setminus D_\eta;\R^m)}\int_{D\setminus D_\eta}W(y/\e,t\Lambda+t\nabla w_{\e,\eta}^\omega(y)+\nabla v(y),\omega)dy.
\end{align*}
Hence, choosing $v=0$, using the upper bound $W\le C(1+V)$ and decomposing
\[t\nabla w_{\e,\eta}^\omega=t\chi_\eta\nabla v_\e^\omega+t(1-\chi_\eta)\nabla u_\e+(1-t)\frac t{1-t}\nabla \chi_\eta(v_\e^\omega-u_\e),\]
we obtain by convexity
\begin{align*}
\Delta^\omega_{\e,t,\eta}&\le C|D\setminus D_\eta|\left(1+\fint_{D\setminus D_\eta}V(y/\e,\Lambda+\nabla u_\e(y),\omega)dy\right.\\
&\hspace{3cm}\left.+\fint_{D\setminus D_\eta}V(y/\e,\Lambda+\nabla v_\e^\omega(y),\omega)dy+E^\omega_{\e,t,\eta}\right),
\end{align*}
where the error reads
\[E_{\e,t,\eta}^\omega:=\fint_{D\setminus D_\eta}V\left(y/\e,\frac t{1-t}\nabla \chi_\eta(y)(v_\e^\omega(y)-u_\e(y)),\omega\right)dy.\]
Since $v_\e^\omega$ and $u_\e$ go to $0$ in $\Ld^\infty(D;\R^m)$, we can prove that, for any $t\in(0,1)$,
\[\limsup_{\eta\downarrow0}\limsup_{\e\downarrow0}\Delta_{\e,t,\eta}^\omega\le \limsup_{\eta\downarrow0}C|D\setminus D_\eta|(1+C_\Lambda+C'_\Lambda)=0.\]
In view of equalities~\eqref{eq:defWhomth} and~\eqref{eq:defWhomthbis}, this implies $\overline W'_t(\Lambda,\omega;D)\le \overline W_t(\Lambda)$.

On the other hand, define
\begin{align}\label{eq:defWhomthbis}
\overline W''_t(\Lambda,\omega;D)=\liminf_{\e\downarrow0}\inf_{v\in W^{1,p}_0(D;\R^m)}\fint_DW(y/\e,t\Lambda+t\nabla u_\e(y)+\nabla v(y),\omega)dy,
\end{align}
and repeat the same argument as above with $D^\eta:=\{x:\dist(x, D)<\eta\}$ and
\begin{align*}
\tilde\Delta_{\e,t,\eta}^\omega:=&\inf_{v\in W^{1,p}_0(D^\eta;\R^m)}\int_{D^\eta}W(y/\e,t\Lambda+t\nabla v_\e^\omega(y)+\nabla v(y),\omega)dy\\
&\qquad-\inf_{w\in W^{1,p}_0(D;\R^m)}\int_{D}W(y/\e,t\Lambda+t\nabla u_\e(y)+\nabla w(y),\omega)dy,
\end{align*}
which then yields $\overline W''_t(\Lambda,\omega;D)\ge \overline W_t(\Lambda)$. This shows that $\overline W''_t(\Lambda,\omega;D)=\overline W'_t(\Lambda,\omega;D)= \overline W_t(\Lambda)$, and the result is proven.
\end{proof}

Let $\Omega_1\subset\Omega_0$ be a subset of maximal probability such that \eqref{eq:defWhomnc} holds for all $\omega\in\Omega_1$, $t\in\Q\cap[0,1)$ and $\Lambda\in\Q^{m\times d}\cap\dom\overline V$, such that \eqref{eq:defWhomncter} holds for all $\omega\in\Omega_1$ and $\Lambda\in\Q^{m\times d}$, and such that we further have, for all $\omega\in\Omega_1$, $\Lambda\in\Q^{m\times d}$, and all bounded domains $O\subset\R^d$,
\[\overline V(\Lambda)=\lim_{\e\downarrow0}\fint_{O/\e}V(y,\Lambda+\nabla\varphi_{\Lambda}(y,\omega),\omega)dy.\]

\subsection{$\Gamma$-$\liminf$ inequality by blow-up}\label{chap:ncliminf}
In this section, we prove the $\Gamma$-$\liminf$ inequality for Theorem~\ref{th:nonconv} by adapting the blow-up method introduced by~\cite{Fonseca-Muller-92} (see also~\cite[Section~4.1]{Anza-Mandallena-11} and~\cite{Braides-08}). In the present context, a subtle use of the corrector for the convex problem is needed.

\begin{prop}[$\Gamma$-$\liminf$ inequality]\label{prop:gammaliminfnc}
For any $\omega\in\Omega_1$, any bounded Lipschitz domain $O\subset\R^d$, and any sequence $(u_\e)_\e\subset W^{1,p}(O;\R^m)$ with $u_\e\cvf{} u$ in $W^{1,p}(O;\R^m)$, we have
\[\liminf_{\e\downarrow0}I_\e(u_\e,\omega;O)\ge I(u;O).\]
\qed
\end{prop}

\begin{proof}
For all $r>0$ and $x\in\R^d$, define $Q_r(x)=x+rQ$ and $S_{r,\kappa}(x)=Q_r(x)\setminus Q_{r\kappa}(x)$ for all $\kappa>0$. For all $\e>0$ and all $\Lambda$, $\omega$, define $\chi_{\e,\Lambda}^\omega=\e\varphi_\Lambda(\cdot/\e,\omega)$. For all
$\omega\in\Omega_1$ and $\Lambda\in\Q^{m\times d}$, the sequence $(\chi^\omega_{\e,\Lambda})_\e$ satisfies $\chi^\omega_{\e,\Lambda}\cvf{}{}0$ in $W^{1,p}(O;\R^m)$ and $J_\e(\chi^\omega_{\e,\Lambda}+\Lambda\cdot x,\,\omega;\,O')\to J(\Lambda\cdot x;\,O')=|O'|\overline V(\Lambda)$ as $\e\downarrow0$, for any subdomain $O'\subset O$.

From now on, let $\omega\in\Omega_1$ be {\it fixed}, let $O\subset\R^d$ be some bounded Lipschitz subset, and let $(u_\e)_\e\subset W^{1,p}(O;\R^m)$ be some fixed sequence with $u_\e\cvf{}u$ in $W^{1,p}(O;\R^{m})$. We need to prove
\begin{align}\label{eq:liminf}
\liminf_{\e\downarrow0}I_\e(u_\e,\omega;O)\ge I(u;O).
\end{align}
It does not restrict generality to assume $\liminf_\e I_\e(u_\e,\omega;O)=\lim_\e I_\e(u_\e,\omega;O)<\infty$ and also $\sup_\e I_\e(u_\e,\omega;O)<\infty$. Hence, $\nabla u_\e(x)\in\dom W(x/\e,\cdot,\omega)=\dom V(x/\e,\cdot,\omega)$ for almost all $x$. Furthermore, the $\Gamma$-convergence result for $V$ yields $J(u;O)\le\liminf_\e J_\e(u_\e,\omega;O)\le\lim_\e I_\e(u_\e,\omega;O)<\infty$, so that $\nabla u(x)\in\dom \overline V$ for almost all $x$.

\medskip
\step1 Localization by blow up: we prove that it suffices to show for almost all $x_0$ that
\begin{align}\label{eq:toproveliminf}
\liminf_{t\uparrow1}\liminf_{r\downarrow0}\lim_{\e\downarrow0}\fint_{Q_r(x_0)}W(y/\e,t\nabla u_\e(y),\omega)dy\ge\overline W(\nabla u(x_0)).
\end{align}

For all $\e>0$, consider the positive Radon measure on $O$ defined by $d\rho_{\e}(x)=W(x/\e,\nabla u_\e(x),\omega)dx$. As $\sup_\e\rho_{\e}(\adh O)<\infty$ by hypothesis, the Prokhorov theorem asserts the convergence $\rho_{\e}\cvf*{}\rho$ up to extraction of a subsequence, for some positive Radon measure $\rho$ on $\adh O$. (The extraction will remain implicit in our notation in the sequel.) By Lebesgue's decomposition theorem, we can consider the absolutely continuous part of the positive measure $\rho$, and the Radon-Nikodym theorem allows to define the density $f\in\Ld^1(U)$ of the latter. As $O$ is open, we then have by the portmanteau theorem (see e.g. \cite[Theorem~2.1]{BillCPM})
\[\liminf_\e I_\e(u_\e,\omega;O)=\liminf_\e\rho_{\e}(O)\ge\rho(O)\ge\int_Of(x)dx.\]
Hence, in order to prove~\eqref{eq:liminf}, it suffices to show that $f(x)\ge\overline W(\nabla u(x))$ for almost all $x$. Since $\rho(\adh O)<\infty$, we have $\rho(\partial Q_r(x))=0$ for all $r\in(0,1)\setminus D_x$, where $D_x$ is at most countable, so that, for almost all $x$, Lebesgue's differentiation theorem and the portmanteau theorem successively give
\[f(x)=\lim_{r\downarrow0\atop r\notin D_x}\frac{\rho (Q_r(x))}{r^d}=\lim_{r\downarrow0\atop r\notin D_x}\lim_{\e\downarrow0}\frac{\rho_{\e} (Q_r(x))}{r^d}.\]
Hence, it suffices to show that, for almost all $x_0$,
\begin{align*}
\liminf_{r\downarrow0}\lim_{\e\downarrow0}\fint_{Q_r(x_0)}W(y/\e,\nabla u_\e(y),\omega)dy\ge\overline W(\nabla u(x_0)).
\end{align*}
Using the ru-usc assumption on $W$, we easily deduce the following inequality:
\begin{align}\label{eq:consruusc}
&\limsup_{t\uparrow1}\liminf_{r\downarrow0}\lim_{\e\downarrow0}\fint_{Q_r(x_0)}W(y/\e,t\nabla u_\e(y),\omega)dy\le\liminf_{r\downarrow0}\lim_{\e\downarrow0}\fint_{Q_r(x_0)}W(y/\e,\nabla u_\e(y),\omega)dy.
\end{align}
Indeed, as $\nabla u_\e(y)\in\dom W(y/\e,\cdot,\omega)$ for almost all $y$, we can write
\begin{align*}
\fint_{Q_r(x_0)}W(y/\e,t\nabla u_\e(y),\omega)dy&\le(1+\Delta^a_W(t))\fint_{Q_r(x_0)}W(y/\e,\nabla u_\e(y),\omega)dy\\
&\hspace{1cm}+\Delta^a_W(t)\fint_{Q_r(x_0)}a(y/\e,\omega)dy,
\end{align*}
and thus, by $\tau$-stationarity of $a$, the Birkhoff-Khinchin ergodic theorem yields
\begin{align*}
&\liminf_{r\downarrow0}\lim_{\e\downarrow0}\fint_{Q_r(x_0)}W(y/\e,t\nabla u_\e(y),\omega)dy\\
\le~&(1+\Delta^a_W(t))\liminf_{r\downarrow0}\lim_{\e\downarrow0}\fint_{Q_r(x_0)}W(y/\e,\nabla u_\e(y),\omega)dy+\Delta^a_W(t)\E [a(0,\cdot)],
\end{align*}
so that inequality~\eqref{eq:consruusc} now directly follows from the ru-usc assumption with respect to $a$ (meaning indeed that $\limsup_{t\uparrow1}\Delta^a_W(t)\le0$). Using~\eqref{eq:consruusc}, we finally conclude that it is sufficient to show~\eqref{eq:toproveliminf} for almost all $x_0$.

\medskip
\step2 Proof of~\eqref{eq:toproveliminf} by truncation.

The idea is to truncate $u_\e$ at the boundary, in order to make appear in the left-hand side of~\eqref{eq:toproveliminf} precisely $\overline W(t\nabla u(x_0))$, which will then allow us to conclude.

Let $t,\kappa\in(0,1)$ be fixed. Since $p>d$, the Sobolev embedding yields $u_\e\to u$ in $\Ld^\infty(O;\R^m)$. Moreover, combining the Lebesgue differentiation theorem for $\nabla u$ and the Sobolev embedding for $p>d$, we deduce that, for all $x_0\notin \Nc$ (for some null set $\Nc\subset\R^d$, $|\Nc|=0$),
\begin{align}\label{eq:derppu0}
\lim_{r\downarrow0}\frac{1}{r}\|u-u(x_0)-\nabla u(x_0)\cdot(x-x_0)\|_{\Ld^\infty(Q_r(x_0))}=0.
\end{align}
Enlarging the null set $\Nc$, we can also assume that $\nabla u(x_0)\in\dom\overline V$ for any $x_0\notin \Nc$. From now on, let $x_0\in O\setminus \Nc$ be fixed and write for simplicity $\Lambda:=\nabla u(x_0)$. Since $\overline V$ is convex and lower semicontinuous, we have $\overline V(\Lambda)=\liminf_{\Lambda'\to\Lambda,\,\Lambda'\in\Q^{m\times d}}\overline V(\Lambda')$, and a diagonalization argument then allows us to choose a sequence $(\Lambda_r)_r\subset\Q^{m\times d}$ such that $\Lambda_r\to\Lambda$ and $\overline V(\Lambda_r)\to\overline V(\Lambda)$ as $r\downarrow0$, and simultaneously
\begin{align}\label{eq:derppu}
\lim_{r\downarrow0}\frac{1}{r}\|u-u(x_0)-\Lambda_r\cdot(x-x_0)\|_{\Ld^\infty(Q_r(x_0))}=0.
\end{align}
Let $\phi_{r,\kappa}$ be a smooth cut-off function with values in $[0,1]$, such that $\phi_{r,\kappa}$ equals $1$ on $Q_{r\kappa}(x_0)$, vanishes outside $Q_r(x_0)$, and satisfies $\|\nabla\phi_{r,\kappa}\|_{\Ld^\infty}\le\frac2{r(1-\kappa)}$. We then set
\[v_{\e,r,\kappa}:=\phi_{r,\kappa} u_\e+(1-\phi_{r,\kappa})(u(x_0)+\Lambda_r\cdot(x-x_0)+\chi_{\e,\Lambda_r}^\omega(x)).\]
Since $v_{\e,r,\kappa}$ coincides with $u_\e$ on $Q_{r\kappa}(x_0)$ and $0\le W\le C(1+V)$, we have
\begin{align}
&\fint_{Q_r(x_0)}W(y/\e,t\nabla v_{\e,r,\kappa}(y),\omega)dy\nonumber\\
\le~&\fint_{Q_r(x_0)}W(y/\e,t\nabla u_\e(y),\omega)dy+\frac C{r^d}\int_{S_{r,\kappa}(x_0)}V(y/\e,t\nabla v_{\e,r,\kappa}(y),\omega)dy+C(1-\kappa^d).\label{eq:bound1ploplop}
\end{align}
Defining $\Psi_{\e,r,\kappa}(x):=\nabla\phi_{r,\kappa}(x)\otimes(u_\e(x)-u(x_0)-\Lambda_r\cdot(x-x_0)-\chi_{\e,\Lambda_r}^\omega)$ and decomposing
\[t\nabla v_{\e,r,\kappa}=t\phi_{r,\kappa}\nabla u_\e+t(1-\phi_{r,\kappa})(\Lambda_r+\nabla\chi_{\e,\Lambda_r}^\omega)+(1-t)\frac t{1-t}\Psi_{\e,r,\kappa},\]
we obtain by convexity of $V$
\begin{align}
V(y/\e,t\nabla v_{\e,r,\kappa}(y),\omega)&\le t\phi_{r,\kappa} V(y/\e,\nabla u_\e(y),\omega)+t(1-\phi_{r,\kappa})V(y/\e,\Lambda_r+\nabla\chi_{\e,\Lambda_r}^\omega(y),\omega)\nonumber\\
&\qquad+(1-t)V\left(y/\e,\frac t{1-t}\Psi_{\e,r,\kappa}(y),\omega\right)\nonumber\\
&\le W(y/\e,\nabla u_\e(y),\omega)+V(y/\e,\Lambda_r+\nabla\chi_{\e,\Lambda_r}^\omega(y),\omega)\nonumber\\
&\qquad+(1-t)V\left(y/\e,\frac t{1-t}\Psi_{\e,r,\kappa}(y),\omega\right).\label{eq:boundVboundary}
\end{align}
Combined with
\begin{align*}
\|\Psi_{\e,r,\kappa}\|_{\Ld^\infty(S_{r,\kappa}(x_0))}&\le\frac{2}{r(1-\kappa)}\Big(\|u_\e-u\|_{\Ld^\infty(O)}+\|u-u(x_0)-\Lambda_r\cdot(x-x_0)\|_{\Ld^\infty(Q_r(x_0))}\\
&\hspace{4cm}+\|\chi_{\e,\Lambda_r}^\omega\|_{\Ld^\infty(O)}\Big),
\end{align*}
the convergences $u_\e\to u$ and $\chi_{\e,\Lambda_r}^\omega\to0$ in $\Ld^\infty(O;\R^m)$ as $\e\downarrow0$, and \eqref{eq:derppu} yield
\[\limsup_{r\downarrow0}\limsup_{\e\downarrow0}\|\Psi_{\e,r,\kappa}\|_{\Ld^\infty(S_{r,\kappa}(x_0))}=0.\]
By assumption, we can find $\delta>0$ with $\adh B_\delta\subset\inter\dom M$. Hence, for all $t,\kappa\in(0,1)$, there exists $r_{\kappa,t}>0$ such that, for all $0<r<r_{\kappa,t}$, there exists some $\e_{r,\kappa,t}>0$ such that for all $0<\e<\e_{r,\kappa,t}$
\[\left\|\frac t{1-t}\Psi_{\e,r,\kappa}\right\|_{\Ld^\infty(S_{r,\kappa}(x_0))}<\delta.\]
This implies
\[\int_{S_{r,\kappa}(x_0)}V\left(y/\e,\frac t{1-t}\Psi_{\e,r,\kappa}(y),\omega\right)dy\le |S_{r,\kappa}(x_0)|\sup_{|\Lambda'|<\delta}M(\Lambda')=r^d(1-\kappa^d)|\sup_{|\Lambda'|<\delta}M(\Lambda'),\]
where the supremum is finite, by virtue of the convexity of $M$ and our choice of $\delta>0$. Combined with inequality~\eqref{eq:boundVboundary} and the definition of the correctors $\chi_{\e,\Lambda_r}^\omega$ (with $\lim_r\overline V(\Lambda_r)=\overline V(\Lambda)<\infty$), this yields
\begin{align*}
&\liminf_{\kappa\uparrow1}\liminf_{t\uparrow1}\liminf_{r\downarrow0}\liminf_{\e\downarrow0}\int_{S_{r,\kappa}(x_0)}V(y/\e,t\nabla v_{\e,r,\kappa}(y),\omega)dy\\
\le~& \liminf_{\kappa\uparrow1}\liminf_{r\downarrow0}\liminf_{\e\downarrow0}\int_{S_{r,\kappa}(x_0)} W(y/\e,\nabla u_\e(y),\omega)dy.
\end{align*}
This turns inequality~\eqref{eq:bound1ploplop} into
\begin{align}
&\liminf_{\kappa\uparrow1}\liminf_{t\uparrow1}\liminf_{r\downarrow0}\liminf_{\e\downarrow0}\fint_{Q_r(x_0)}W(y/\e,t\nabla v_{\e,r,\kappa}(y),\omega)dy\nonumber\\
\le~& \liminf_{t\uparrow1}\liminf_{r\downarrow0}\liminf_{\e\downarrow0}\fint_{Q_r(x_0)}W(y/\e,t\nabla u_\e(y),\omega)dy\label{eq:bound1ploplop2}\\
&\qquad+\liminf_{\kappa\uparrow1}\liminf_{r\downarrow0}\liminf_{\e\downarrow0}\frac C{r^d}\int_{S_{r,\kappa}(x_0)} W(y/\e,\nabla u_\e(y),\omega)dy.\nonumber
\end{align}
Since we have chosen $\omega\in\Omega_1$, $t\in\Q\cap(0,1)$ and $\Lambda_r\in\Q^{m\times d}$, \eqref{eq:defWhomnc} holds and reads
\[\overline W_t(\Lambda_r)=\liminf_{\e\downarrow0}\inf_{v\in W^{1,p}_0(O;\R^m)}\fint_{Q_r(x_0)}W(y/\e,t\Lambda_r+t\nabla \varphi_{\Lambda_r}(\cdot/\e,\omega)+\nabla v(y),\omega)dy.\]
Hence, since $v_{\e,r,\kappa}-u(x_0)-\Lambda_r\cdot(x-x_0)\in \chi_{\e,\Lambda_r}^\omega+W^{1,p}_0(Q_r(x_0);\R^m)$ with $\chi_{\e,\Lambda_r}^\omega=\e\varphi_{\Lambda_r}(\cdot/\e,\omega)$, \eqref{eq:bound1ploplop2}~yields
\begin{align}
\overline W(\nabla u(x_0))&=\overline W(\Lambda)\le\liminf_{t\uparrow1}\liminf_{r\downarrow0}\overline W_t(\Lambda_r)\nonumber\\
&\le\liminf_{t\uparrow1}\liminf_{r\downarrow0}\liminf_{\e\downarrow0}\fint_{Q_r(x_0)}W(y/\e,t\nabla v_{\e,r,\kappa}( y),\omega)dy\nonumber\\
&\le\liminf_{t\uparrow1}\liminf_{r\downarrow0}\liminf_{\e\downarrow0}\fint_{Q_r(x_0)}W(y/\e,t\nabla u_\e(y),\omega)dy\label{eq:boundfinwhatloc}\\
&\qquad+\limsup_{\kappa\uparrow1}\limsup_{r\downarrow0}\limsup_{\e\downarrow0}\frac C{r^d}\int_{S_{r,\kappa}(x_0)}W(y/\e,\nabla u_\e(y),\omega)dy.\nonumber
\end{align}
It remains to get rid of the second term of the right-hand side of~\eqref{eq:boundfinwhatloc}. By the portmanteau theorem,
\begin{align*}
\limsup_{\e\downarrow0}\frac 1{r^d}\int_{S_{r,\kappa}(x_0)}W(y/\e,\nabla u_\e(y),\omega)dy&=\limsup_{\e\downarrow0}\frac{\rho_{\e}(S_{r,\kappa}(x_0))}{r^d}\\
&\le\limsup_{\e\downarrow0}\frac{\rho_{\e}(\adh S_{r,\kappa}(x_0))}{r^d}\le\frac{\rho(\adh S_{r,\kappa}(x_0))}{r^d}.
\end{align*}
Since the singular part of $\rho$ must be supported in a closed subset of $\adh O$ of measure $0$, we deduce, for almost all $x_0\in O\setminus \Nc$, the existence of some $r_0>0$ sufficiently small such that $\adh Q_r(x_0)$ has no intersection with that support for all $0<r<r_0$. Hence, for all $0<r<r_0$,
\begin{align*}
&\limsup_{\e\downarrow0}\frac 1{r^d}\int_{S_{r,\kappa}(x_0)}W(y/\e,\nabla u_\e(y),\omega)dy\\
\le~&\frac{\rho(\adh S_{r,\kappa}(x_0))}{r^d}=\frac1{r^d}\int_{S_{r,\kappa}(x_0)}f(y)dy=\fint_{Q_{r}(x_0)}f(y)dy-\kappa^d\fint_{Q_{r\kappa}(x_0)}f(y)dy,
\end{align*}
where, for almost all $x_0$, the right-hand side converges to $(1-\kappa^d)f(x_0)$ as $r\downarrow0$ by Lebesgue's differentiation theorem. Hence, for almost all $x_0$, \eqref{eq:boundfinwhatloc} turns into
\begin{align*}
\overline W(\nabla u(x_0))&\le\limsup_{t\uparrow1}\liminf_{r\downarrow0}\liminf_{\e\downarrow0}\fint_{Q_r(x_0)}W(y/\e,t\nabla u_\e(y),\omega)dy,
\end{align*}
and the desired result~\eqref{eq:toproveliminf} is proven.
\end{proof}

\subsection{$\Gamma$-$\limsup$ inequality with Neumann boundary data}\label{chap:nclimsup}

In this section, we prove the $\Gamma$-$\limsup$ inequality, first considering the affine case, and then deducing the general case by approximation. For this approximation argument to hold, we would however need to know a priori that the homogenized energy $\overline W$ satisfies good regularity properties (i.e. lower semicontinuity on $\R^{m\times d}$ and continuity on $\inter\dom\overline V$). Since this is not clear at all a priori, our strategy (inspired by~\cite{Anza-Mandallena-11}) consists in introducing some relaxations of $\overline W$ that enjoy the required properties, and then in deducing a posteriori from $\Gamma$-convergence (or a weaker form of it) the equality of $\overline W$ with its relaxations (so that $\overline W$ itself has all the desired properties). Motivated by the work of Fonseca~\cite{Fonseca-88} (see also~\cite{Anza-Mandallena-11}), we thus consider the following relaxation of $\overline W$:
\[\Zc W(\Lambda):=\inf\left\{\fint_O\overline W(\Lambda+\nabla \phi(y))dy\,:\,\text{$\phi$ continuous piecewise affine on $O$ and $\phi|_{\partial O}=0$}\right\},\]
where the definition does clearly not depend on the chosen underlying (nonempty) bounded Lipschitz domain $O\subset\R^d$. Also write $\widehat\Zc W$ for the lower semicontinuous envelope of $\Zc W$ (defined by $\widehat\Zc W(\Lambda):=\liminf_{\Lambda'\to\Lambda}\Zc W(\Lambda')$ for all $\Lambda$). Now define the integral functionals corresponding to all these relaxed integrands: for any bounded domain $O\subset\R^d$ and $u\in W^{1,p}(O;\R^m)$,
\begin{align*}
\Zc I(u;O)&:=\int_O\Zc W(\nabla u(y))dy,\qquad\widehat\Zc I(u;O):=\int_O\widehat\Zc W(\nabla u(y))dy.
\end{align*}
The following result gives some properties of these relaxations, which will be crucial in the sequel:
\begin{lem}[Properties of relaxations]\label{lem:lemma}
Assume $p>d$. Then the following holds:
\begin{enumerate}[(a)]
\item $\Zc W$ (and thus also $\widehat\Zc W$) is continuous on $\inter\dom\,\Zc W$.
\item $\overline V\le\widehat\Zc W\le\overline W\le C(1+\overline V)$.
\item $\overline W$ and $\Zc W$ are ru-usc.
\item For any $t\in(0,1)$, we have $t\,\adh\,\dom\Zc W\subset\inter\dom\Zc W$, and the following representation result holds:
\begin{align*}
\widehat\Zc W(\Lambda)&=\liminf_{t\to1}\Zc W(t\Lambda)=\begin{cases}
\Zc W(\Lambda),&\text{if $\Lambda\in\inter\dom\Zc W$;}\\
\lim_{t\uparrow1}\Zc W(t\Lambda),&\text{if $\Lambda\in\partial\,\dom\Zc W$;}\\
\infty,&\text{otherwise;}\end{cases}
\end{align*}
where, in particular, the limit in the second line does exist.
\item Let $\Lambda\in\R^{m\times d}$ and let $O\subset\R^d$ be a bounded Lipschitz domain. Then, there exists a sequence $(\phi_k)_k\subset W^{1,p}_0(O;\R^m)$ of piecewise affine functions such that $\phi_k\to0$ in $\Ld^\infty(O;\R^m)$ and
\[\lim_{k\uparrow\infty}\fint_O\overline W(\Lambda+\nabla\phi_k(y))dy=\Zc W(\Lambda).\]
\end{enumerate}\qed
\end{lem}

\begin{proof}
Continuity of $\Zc W$ on $\inter\dom\,\Zc W$ is a result due to~\cite{Fonseca-88}, which yields part~(a) (even without any ru-usc assumption on $W$).
The inequalities stated in part~(b) directly follow from the definitions of $\overline V$, $\overline W$ and $\widehat\Zc W$. Part~(e) is standard (see~\cite[Proposition~3.17]{Anza-Mandallena-11} for details). It remains to prove properties~(c) and~(d).

\medskip
\step1 Proof of~(c).

We first prove that $\overline W_t$ is ru-usc, for any fixed $t\in[0,1)$. Let $s>0$, $\Lambda\in\dom\overline W=\dom\overline V$, and let $O\subset\R^d$ be some bounded Lipschitz domain. For almost all $\omega$, note that by convexity, for all subdomains $D\subset O$,
\begin{align*}
&\limsup_{\e\downarrow0}\int_DV(y/\e,s\Lambda+s\nabla\varphi_\Lambda(y/\e,\omega),\omega)dy\\
\le~& \lim_{\e\downarrow0}\int_DV(y/\e,\Lambda+\nabla\varphi_\Lambda(y/\e,\omega),\omega)dy+(1-s)M(0)=\overline V(\Lambda)+(1-s)M(0)<\infty.
\end{align*}
Hence for almost all $\omega$ we can apply equality~\eqref{eq:defWhomncbis} at $s\Lambda$ with $u_\e=s\e\varphi_\Lambda(\cdot/\e,\omega)$, which yields
\[\overline W_t(s\Lambda)=\lim_{\e\downarrow0}\inf_{v\in W^{1,p}_0(O;\R^m)}\fint_{O}W(y/\e,st\Lambda+st\nabla\varphi_\Lambda(y/\e,\omega)+\nabla v(y),\omega)dy.\]
Given $\omega\in\Omega$ such that this convergence and~\eqref{eq:defWhomnc} both hold, and choosing a sequence $(v_\e^\omega)_\e\subset W_0^{1,p}(O;\R^d)$ such that
\begin{align*}
\overline W_t(\Lambda)=\lim_{\e\downarrow0}\fint_{O}W(y/\e,t\Lambda+t\nabla\varphi_\Lambda(y/\e,\omega)+\nabla v_\e^\omega(y),\omega)dy,
\end{align*}
we deduce
\begin{align}
\overline W_t(s\Lambda)-\overline W_t(\Lambda)&\le\lim_{\e\downarrow0}\fint_{O}\Big(W(y/\e,s(t\Lambda+t\nabla\varphi_{\Lambda}(y/\e,\omega)+\nabla v_\e^\omega(y)),\omega)\nonumber\\
&\hspace{3cm}-W(y/\e,t\Lambda+t\nabla\varphi_{\Lambda}(y/\e,\omega)+\nabla v_\e^\omega(y),\omega)\Big)dy\nonumber\\
&\le\Delta_W^a(s)\lim_{\e\downarrow0}\fint_{O}(a(y/\e,\omega)+W(y/\e,t\Lambda+t\nabla\varphi_{\Lambda}(y/\e,\omega)+\nabla v_\e^\omega(y),\omega))dy\nonumber\\
&=\Delta_W^a(s)(\E [a(0,\cdot)]+\overline W_t(\Lambda)),\label{eq:ruuscWhomt}
\end{align}
using the Birkhoff-Khinchin ergodic theorem for the stationary field $a$. As $\overline W_t$ and the field $a$ are nonnegative, as $\alpha:=\E[a(0,\cdot)]$ is finite and as $\limsup_{s\uparrow1}\Delta_W^a(s)\le 0$ by assumption, we deduce that $\overline W_t$ is also ru-usc. Now rewriting inequality~\eqref{eq:ruuscWhomt} in the form
\[\overline W_t(s\Lambda)\le \alpha\Delta^a_W(s)+(1+(-1)\vee\Delta^a_W(s))\overline W_t(\Lambda),\]
and taking the suitable $\liminf$, we directly deduce $\overline W(s\Lambda)-\overline W(\Lambda)\le (-1)\vee \Delta^a_W(s)(\alpha+\overline W(\Lambda))$ for all $\Lambda\in\dom\overline V$ and $s\in[0,1)$, proving that $\overline W$ is itself ru-usc with $\Delta_{\overline W}^\alpha=(-1)\vee \Delta^a_W(s)$.

We now show that $\Zc W$ is also ru-usc. Take $s>0$ and $\Lambda\in\dom\overline V$. By definition, there exists a sequence of piecewise affine functions $(\phi_k)_k\subset W^{1,p}_0(O)$ such that
\[\Zc W(\Lambda)=\lim_{k\uparrow\infty}\fint_O\overline W(\Lambda+\nabla\phi_k(y))dy.\]
As $\Lambda\in\dom\overline V$, the left-hand side is finite, and we can thus assume $\Lambda+\nabla\phi_k\in \dom\overline W$ almost everywhere. Hence the ru-usc property satisfied by $\overline W$ gives
\begin{align*}
\Zc W(s\Lambda)-\Zc W(\Lambda)&\le\lim_{k\uparrow\infty}\fint_O(\overline W(s(\Lambda+\nabla\phi_k(y)))-\overline W(\Lambda+\nabla\phi_k(y)))dy\\
&\le \Delta_{\overline W}^\alpha(s)\lim_{k\uparrow\infty}\fint_O(\alpha+\overline W(\Lambda+\nabla\phi_k(y))dy=\Delta_{\overline W}^\alpha(s)(\alpha+\Zc W(\Lambda)).
\end{align*}

\medskip
\step2 Proof of (d).

Since $\dom\Zc W=\dom\overline V$ is a convex set containing $0$, it is clear that, for all $t\in[0,1)$, $t\,\adh\dom\Zc W$ is contained in $\inter\dom\Zc W$. We first show that the limit $\lim_{t\uparrow1}\Zc W(t\Lambda)$ exists for all $\Lambda\in\adh\dom\overline V$. Given some fixed $\Lambda\in\adh\dom\overline V$, choose two sequences $s_n\uparrow1$ and $t_n\uparrow1$ with $t_n/s_n\uparrow1$ such that
\[\lim_{n\uparrow\infty}\Zc W(s_n\Lambda)=\liminf_{t\uparrow1}\Zc W(t\Lambda)\qquad\text{and}\qquad\lim_{n\uparrow\infty}\Zc W(t_n\Lambda)=\limsup_{t\uparrow1}\Zc W(t\Lambda).\]
As $s_n\Lambda,t_n\Lambda\in \dom\overline V$ for all $n$, and as $\Zc W$ is ru-usc, we have
\begin{align*}
\lim_{n\uparrow\infty}\Zc W(t_n\Lambda)&\le \limsup_{n\uparrow\infty} (\alpha+\Zc W(s_n\Lambda))\Delta_{W}^a(t_n/s_n)+\lim_{n\uparrow\infty}\Zc W(s_n\Lambda)\\
&\le\lim_{n\uparrow\infty}\Zc W(s_n\Lambda)\le \lim_{n\uparrow\infty}\Zc W(t_n\Lambda),
\end{align*}
which thus proves the existence of the limit $\lim_{t\uparrow1}\Zc W(t\Lambda)$ for all $\Lambda\in\adh\dom\overline V$.

We now prove the claimed representation result. First, if $\Lambda\in\inter\dom\overline V$, then $\liminf_{t\to1}\Zc W(t\Lambda)=\Zc W(\Lambda)=\widehat\Zc W(\Lambda)$ follows from part~(a). Second, if $\Lambda\notin\adh\dom\overline V$, then $\Zc W(t\Lambda)=\infty$ for any $t$ sufficiently close to $1$, and thus $\liminf_{t\to1}\Zc W(t\Lambda)=\infty=\widehat \Zc W(\Lambda)$. Now it only remains to consider $\Lambda\in\partial\,\dom\Zc W$. Then $\Zc W(t\Lambda)=\infty$ whenever $t>1$, so that we simply have $\liminf_{t\to1}\Zc W(t\Lambda)=\liminf_{t\uparrow1}\Zc W(t\Lambda)=\lim_{t\uparrow1}\Zc W(t\Lambda)$, since we have already proven the existence of this limit. Hence, it suffices to prove that $\widehat\Zc W(\Lambda)=\liminf_{t\uparrow1}\Zc W(t\Lambda)$. By definition of the lower semicontinuous envelope $\widehat\Zc W$ of $\Zc W$, this equality would follow if we could show that, for any sequence $\Lambda_n\to\Lambda$, we have
\begin{align}\label{eq:scirep}
\liminf_{n\uparrow\infty}\Zc W(\Lambda_n)\ge\liminf_{t\uparrow1}\Zc W(t\Lambda).
\end{align}
It is of course sufficient to assume $\liminf_n\Zc W(\Lambda_n)=\lim_n\Zc W(\Lambda_n)<\infty$ and $\sup_n\Zc W(\Lambda_n)<\infty$. Hence, $\Lambda_n\in\dom\overline V$ for all $n$, and thus, for all $t\in[0,1)$, $t\Lambda\in\inter\dom\overline V$, so that, using part~(a) as well as the ru-usc property satisfied by $\Zc W$, we have
\[\Zc W(t\Lambda)=\lim_{n\uparrow\infty}\Zc W(t\Lambda_n)\le \lim_{n\uparrow\infty}\Zc W(\Lambda_n)+\Delta_W^a(t)\lim_{n\uparrow\infty}(\alpha+\Zc W(\Lambda_n)).\]
This yields
\[\limsup_{t\uparrow1}\Zc W(t\Lambda)\le\liminf_{n\uparrow\infty}\Zc W(\Lambda_n),\]
and proves~\eqref{eq:scirep}.
\end{proof}

Combining the $\Gamma$-$\liminf$ inequality for $I_\e$ towards $I$ with a $\Gamma$-$\limsup$ argument, we prove the following a priori surprising equality of $\overline W$ with its relaxations.

\begin{lem}[Regularity of the homogenized energy density]\label{lem:equZc}
Assume $p>d$. Then $\overline W(\Lambda)=\Zc W(\Lambda)=\widehat \Zc W(\Lambda)$ for all $\Lambda\in\R^{m\times d}$. In particular, $\overline W$ is lower semicontinuous on $\R^{m\times d}$ and is continuous on $\inter\dom\overline V$.\qed
\end{lem}

\begin{proof}
We split the proof into four steps.

\medskip
\step1 Recovery sequence for $I(\Lambda\cdot x;O)$.

Let $\Lambda\in\inter\dom\overline V$ and let $t\in[0,1)$. In this step, for almost all $\omega$, for all bounded Lipschitz domain $O\subset\R^d$, we prove the existence of sequences $t_\e\uparrow1$, $\Lambda_\e\to\Lambda$ and $(w_\e)_\e\subset W^{1,p}_0(O;\R^m)$ such that $\e\varphi_{\Lambda_\e}(\cdot/\e,\omega)\cvf{}0$, $w_\e\cvf{}0$ in $W^{1,p}(O;\R^m)$ and $I_\e(t_\e\Lambda_\e\cdot x+\e t_\e\varphi_{\Lambda_\e}(\cdot/\e,\omega)+w_\e,\omega;O)\to |O|\overline W(\Lambda)=I(\Lambda\cdot x;O)$.

By definition of $\overline W$ and a diagonalization argument, for almost all $\omega$ and all bounded Lipschitz domains $O\subset\R^d$, it suffices to prove the existence of a sequence $(v_\e)_\e\subset W^{1,p}_0(O;\R^m)$ such that $v_\e\cvf{}0$ in $W^{1,p}(O;\R^m)$ and $I_\e(t\Lambda\cdot x+\e t\varphi_\Lambda(\cdot/\e,\omega)+v_\e,\omega;O)\to |O|\overline W_t(\Lambda)$ as $\e\downarrow0$.

Let $O$ be some fixed bounded Lipschitz domain. Given $\e>0$, consider the cubes of the form $k(z+Q)$, $z\in\Z^d$, that are contained in $O/\e$, and denote by $z_{j}\in \Z^d$, $j=1,\ldots,N_{\e,k}$, the centers of these cubes (the enumeration of which can be chosen independent of $\e,k$). Since $O$ is Lipschitz, we have $N_{\e,k}(\e k)^d\to|O|$ as $\e\downarrow0$, for all $k$. For all $j,\omega$, we can choose a sequence $(v_{k}^{j,\omega})_k$ with $v_{k}^{j,\omega}\in W^{1,p}_0(k(z_j+Q);\R^m)$ such that
\begin{align*}
&\fint_{k(z_j+Q)}W(y,t\Lambda+t\nabla\varphi_\Lambda(y,\omega)+\nabla v_{k}^{j,\omega}(y),\omega)dy\\
\le~&\frac1k+\inf_{v\in W^{1,p}_0(k(z_j+Q);\R^m)}\fint_{k(z_j+Q)}W(y,t\Lambda+t\nabla\varphi_\Lambda(y,\omega)+\nabla v(y),\omega)dy.
\end{align*}
For all $\e,k,\omega$, we then consider the function $v_{\e,k}^\omega:=\sum_{j=1}^{N_{\e,k}}v_k^{j,\omega}\mathds1_{k(z_j+Q)}\in W^{1,p}_0(O/\e;\R^m)$, and we define $w_{\e,k}^\omega:=\e v_{\e,k}^\omega(\cdot/\e)\in W^{1,p}_0(O;\R^m)$. Up to a diagonalization argument, it suffices to show that, for almost all $\omega$ (independent of the choice of $O$, as it is clear in the proof below),
\begin{align}\label{eq:blablaplop1}
&\limsup_{k\uparrow\infty}\limsup_{\e\downarrow0}\Big(\big|I_\e(t\Lambda\cdot x+\e t\varphi_\Lambda(\cdot/\e,\omega)+w_{\e,k}^\omega,\omega;O)-|O|\overline W_t(\Lambda)\big|+\|w_{\e,k}^\omega\|_{\Ld^p(O)}\Big)=0.
\end{align}
First we argue that, for almost all $\omega$,
\begin{align}\label{eq:blablaplop2}
\limsup_{k\uparrow\infty}\limsup_{\e\downarrow0} I_\e(t\Lambda\cdot x+\e t\varphi_\Lambda(\cdot/\e,\omega)+w_{\e,k}^\omega,\omega;O)\le |O|\overline W_t(\Lambda).
\end{align}
Indeed, by definition of $w_{\e,k}^\omega$,
\begin{align*}
&I_\e(t\Lambda\cdot x+\e t\varphi_\Lambda(\cdot/\e,\omega)+w_{\e,k}^\omega,\omega;O)\\
\le~&\frac1k(\e k)^dN_{\e,k}+(\e k)^d\sum_{j=1}^{N_{\e,k}}\inf_{v\in W^{1,p}_0(k(z_j+Q);\R^m)}\fint_{k(z_j+Q)}W(y,t\Lambda+t\nabla\varphi_\Lambda(y,\omega)+\nabla v(y),\omega)dy\\
&+\e^d\int_{(O/\e)\setminus \bigcup_{j=1}^{N_{\e,k}}k(z_j+Q)}W(y,t\Lambda+t\nabla\varphi_\Lambda(y,\omega),\omega)dy.
\end{align*}
Since $W\le C(1+V)$, the last term of the right-hand side goes to $0$ as $\e\downarrow0$ for almost all $\omega$ by construction of the cubes $k(z_j+Q)$ and definition of $\varphi_\Lambda$. The Birkhoff-Khinchin ergodic theorem (which we apply to a measurable map by Hypothesis~\ref{hypo:asmeasinf}) then gives, for almost all $\omega$,
\begin{align}
&\limsup_{\e\downarrow0}I_\e(t\Lambda\cdot x+\e t\varphi_\Lambda(\cdot/\e,\omega)+w_{\e,k}^\omega,\omega;O)\nonumber\\
\le~&\frac {|O|}k+|O|\,\E\left[\inf_{v\in W^{1,p}_0(kQ;\R^m)}\fint_{kQ}W(y,t\Lambda+t\nabla\varphi_\Lambda(y,\cdot)+\nabla v(y),\cdot)dy\right].\label{eq:boundasienc}
\end{align}
Lemma~\ref{lem:Whomnc} then yields the desired result~\eqref{eq:blablaplop2} as $k\uparrow\infty$. On the other hand, by definition~\eqref{eq:defWhomnc} of $\overline W_t$, for all $k$ and almost all $\omega$, we have
\begin{align}
&\liminf_{\e\downarrow0}I_\e(t\Lambda\cdot x+\e t\varphi_\Lambda(\cdot/\e,\omega)+w_{\e,k}(\cdot,\omega),\omega;O)\nonumber\\
\ge~& |O|\liminf_{\e\downarrow0}\inf_{v\in W^{1,p}_0(O;\R^m)}\fint_{O/\e}W(y,t\Lambda+t\nabla\varphi_\Lambda(y,\omega)+\nabla v(y),\omega;O)= |O|\overline W_t(\Lambda).\label{eq:blablaplop3}
\end{align}
We now show that $w_{\e,k}^\omega\to0$ in $\Ld^p(O;\R^m)$ as $\e\downarrow0$, for almost all $\omega$. Combining inequality~\eqref{eq:boundasienc} with the bound $W\le C(1+V)$, the $p$-th order lower bound for $W$ and the convexity of $V$, we indeed have
\begin{align*}
\limsup_{\e\downarrow0}\|t\Lambda+t\nabla\varphi_\Lambda(\cdot/\e,\omega)+\nabla w_{\e,k}^\omega\|_{\Ld^p(O)}^p\le \frac{|O|}k+C|O|(1+\overline V(\Lambda)+(1-t)M(0))<\infty.
\end{align*}
For almost all $\omega$, the weak $\Ld^p$ convergence of the sequence $(\nabla\varphi_\Lambda(\cdot/\e,\omega))_\e$ to $0$ implies the boundedness of this sequence in $\Ld^p(O;\R^{m\times d})$, so that $(\nabla w_{\e,k}^\omega)_\e$ is also bounded in $\Ld^p(O;\R^{m\times d})$, for any fixed $k$. By Poincaré's inequality on cubes of side length $k\e$, this implies
\begin{align*}
\|w_{\e,k}^\omega\|_{\Ld^p(O)}\le C_k(\omega)\e,
\end{align*}
for some (random) constant $C_k(\omega)$.
Combined with~\eqref{eq:blablaplop2} and~\eqref{eq:blablaplop3}, this proves~\eqref{eq:blablaplop1}.

\medskip
\step2 Recovery sequence for $\Zc I(\Lambda\cdot x;O)$.

Let $\Lambda\in\inter\dom\overline V$ and let $O\subset\R^d$ be a bounded Lipschitz domain. In this step, for almost all $\omega$, we prove the existence of a sequence $(u_\e)_\e\subset W^{1,p}(O;\R^m)$ such that $u_\e\cvf{} 0$ in $W^{1,p}(O;\R^m)$ and $I_\e(\Lambda\cdot x+u_\e,\omega;O)\to \Zc I(\Lambda\cdot x;O)$.

Lemma~\ref{lem:lemma}(e) gives a sequence $(\phi_k)_k\subset W^{1,p}_0(O;\R^m)$ of piecewise affine functions such that $\phi_k\to0$ in $\Ld^p(O;\R^m)$ (and even in $\Ld^\infty(O;\R^m)$), and
\[I(\phi_k+\Lambda\cdot x;O)=\int_O\overline W(\Lambda+\nabla\phi_k(y))dy\xrightarrow{k\uparrow\infty}|O|\Zc W(\Lambda).\]
Denote by $(P_k^i)_{i=1}^{n_k}$ the partition of $O$ associated with the piecewise affine function $\phi_k$. For all $k$ and $1\le i\le n_k$, considering $\Lambda_{k}^{i}:=\Lambda+\nabla \phi_k|_{P_k^i}$ on $P_k^i$, Step~1 above gives, for almost all $\omega$, a sequence $t_\e\uparrow1$, a sequence $\Lambda_{k,\e}^i\to\Lambda_k^i$ and a sequence $(v_{\e,k}^i)_\e\subset W^{1,p}_0(P_k^i;\R^m)$ such that $\e \varphi_{\Lambda_{\e,k}^i}(\cdot/\e,\omega),v_{\e,k}^i\cvf{}0$ in $W^{1,p}(P_k^i;\R^m)$ and $I_\e(t_\e\Lambda_{\e,k}^i\cdot x+\e t_\e\varphi_{\Lambda_{\e,k}^i}(\cdot/\e,\omega)+v_{\e,k}^i,\omega;P_k^i)\to I(\Lambda_k^i\cdot x;P_k^i)=I(\phi_k+\Lambda\cdot x;P_ k^i)$. As the $v_{\e,k}^i$'s satisfy Dirichlet boundary conditions, they can be directly glued together, while for the $\varphi_\Lambda$'s we need to repeat the more complicated gluing argument of Step~2 of the proof of Proposition~\ref{prop:gammasupN}, with $p>d$. Although the functional $I_\e$ is not convex here, as in the proof of Proposition~\ref{prop:gammaliminfnc}, the idea is to use the bound $W\le C(1+V)$ at all points where the cut-off functions are different from $1$ or $0$, then use the convexity of $V$ and estimate the corresponding error terms as before. We leave the details to the reader.

\medskip
\step3 Recovery sequence for $\widehat\Zc I(\Lambda\cdot x;O)$.

Let $\Lambda\in\dom\overline V$ and let $O\subset\R^d$ be a bounded Lipschitz domain. In this step, for almost all $\omega$, we prove the existence of a sequence $(u_\e^\omega)_\e\subset W^{1,p}(O;\R^m)$ such that $u_\e^\omega\cvf{} 0$ in $W^{1,p}(O;\R^m)$ and $I_\e(\Lambda\cdot x+u_\e^\omega,\omega;O)\to \widehat\Zc I(\Lambda\cdot x;O)$.

By Lemma~\ref{lem:lemma}(d), $\Zc W$ and $\widehat \Zc W$ coincide on $\inter\dom\overline V$, and hence the result on $\inter\dom\overline V$ already follows from Step~2. Let now $\Lambda\in\partial \dom\Zc W$. Lemma~\ref{lem:lemma}(d) then asserts $\widehat \Zc W(\Lambda)=\lim_{t\uparrow1}\Zc W(t\Lambda)$. By convexity of $\dom\overline V$, for all $t\in[0,1)$, we have $t\Lambda\in\inter\dom\overline V$, and hence, for almost all $\omega$, Step~2 above gives a sequence $(u_{\e,t})_\e\subset W^{1,p}(O;\R^m)$ such that $u_{\e,t}\cvf{}0$ in $W^{1,p}(O;\R^m)$ and $I_\e(t\Lambda\cdot x+u_{\e,t},\omega;O)\to \Zc I(t\Lambda\cdot x;O)=|O|\Zc W(t\Lambda)$. The conclusion then follows from a diagonalization argument.

\medskip
\step4 Conclusion.

Let $\Lambda\in\dom\overline V$, let $O\subset\R^d$ be a bounded Lipschitz domain, and let $(u_\e^\omega)_\e$ be the sequence given by Step~3 above. As $u_\e^\omega\cvf{}0$ in $W^{1,p}(O;\R^m)$, the $\Gamma$-$\liminf$ inequality (see Proposition~\ref{prop:gammaliminfnc}) gives, for almost all $\omega$,
\[|O|\widehat \Zc W(\Lambda)=\lim_{\e\downarrow0}I_\e(\Lambda\cdot x+u_\e^\omega,\omega;O)\ge I(\Lambda\cdot x;O)=|O|\overline W(\Lambda).\]
This being true for any $\Lambda\in\dom\overline V$, we conclude that $\widehat \Zc W=\overline W$ everywhere.
\end{proof}

With Lemma~\ref{lem:equZc} at hands, we may prove the $\Gamma$-$\limsup$ inequality.

\begin{prop}[$\Gamma$-$\limsup$ inequality]\label{prop:gammalimsupnc}
Assume $p>d$. There exists a subset $\Omega'\subset\Omega_1$ of maximal probability with the following property: for all $\omega\in\Omega'$, all strongly star-shaped (in the sense of Proposition~\ref{prop:approxw}) bounded Lipschitz domains $O\subset\R^d$ and all $u\in W^{1,p}(O;\R^m)$, there exist a sequence $(u_\e)_\e\subset W^{1,p}(O;\R^m)$ and a sequence $(v_\e)_\e\subset W^{1,p}_0(O;\R^m)$ such that $u_\e\cvf{}u$ and $v_\e\cvf{}0$ in $W^{1,p}(O;\R^m)$, and such that $I_\e(u_\e+v_\e,\omega;O)\to I(u;O)$  and $J_\e(u_\e,\omega;O)\to J(u;O)$ as $\e\downarrow0$.\qed
\end{prop}

Recall that the $\Gamma$-$\liminf$ inequality implies the locality of recovery sequences (see the proof of Corollary~\ref{cor:localizlim}). Hence, the $\Gamma$-convergence result on a Lipschitz domain $D$ for Neumann boundary conditions follows from the $\Gamma$-$\limsup$ on a ball $B\supset D$ and the $\Gamma$-$\liminf$ inequality on $B\setminus D$. For the adaptation of Corollary~\ref{cor:dir}, the approach is similar and we leave the details to the reader.

\begin{proof}

\step1 Recovery sequence for affine functions.

In this step, we consider the case when $u=\Lambda\cdot x$. is an affine function. More precisely, we prove the existence of a subset $\Omega'\subset\Omega_1$ of maximal probability with the following property: given a bounded Lipschitz domain $O\subset\R^d$, for all $\omega\in\Omega'$ and all $\Lambda \in\inter\dom\overline W$, there exist sequences $(u_\e)_\e\subset W^{1,p}_\loc(\R^d;\R^m)$ and $(v_\e)_\e\subset W^{1,p}_0(O;\R^m)$ such that $u_\e\cvf{}\Lambda\cdot x$ in $W^{1,p}_{\loc}(\R^d;\R^m)$ and $v_\e\cvf{}0$ in $W^{1,p}(O;\R^m)$, and such that $I_\e(u_\e+v_\e,\omega;O)\to I(\Lambda\cdot x;O)$ and $J_\e(u_\e,\omega;O')\to J(\Lambda\cdot x;O')$ for all bounded domains $O'\subset \R^d$.

Let $\Lambda\in \inter\dom \overline V$. For almost all $\omega\in\Omega_1$, and all bounded domains $O'\subset\R^d$, we have by convexity, the Birkhoff-Khinchin ergodic theorem, definition of $\varphi_{\Lambda'}$, and continuity of $\overline V$ at $\Lambda$:
\begin{align*}
\limsup_{t\uparrow1,\Lambda'\to\Lambda}\lim_{\e\downarrow0} J_\e(t\Lambda'\cdot x+\e t\varphi_{\Lambda'}(\cdot/\e,\omega),\omega;O')&\le\limsup_{\Lambda'\to\Lambda}\lim_{\e\downarrow0} J_\e(\Lambda'\cdot x+\e \varphi_{\Lambda'}(\cdot/\e,\omega),\omega;O')\\
&=|O'|\lim_{\Lambda'\to\Lambda} \overline V(\Lambda')=|O'|\overline V(\Lambda)=J(\Lambda\cdot x;O').
\end{align*}
Combined with the $\Gamma$-$\liminf$ inequality for $J_\e(\cdot,\omega;O')$ towards $J(\cdot;O')$ (for $\omega\in\Omega_1$), this yields
\begin{align}\label{eq:convconvnc}
\lim_{t\uparrow1,\Lambda'\to\Lambda}\lim_{\e\downarrow0} J_\e(t\Lambda'\cdot x+\e t\varphi_{\Lambda'}(\cdot/\e,\omega),\omega;O')&= J(\Lambda\cdot x;O').
\end{align}
By definition of $\overline W$, we may choose sequences $\Lambda_n\to\Lambda$ and $t_n\uparrow1$ such that $\overline W_{t_n}(\Lambda_n)\to\overline W(\Lambda)$. For this choice, \eqref{eq:convconvnc} yields for almost all $\omega$ and all bounded domain $O'\subset\R^d$
\[\lim_{n\uparrow\infty}\lim_{\e\downarrow0} J_\e(t_n\Lambda_n\cdot x+\e t_n\varphi_{\Lambda_n}(\cdot/\e,\omega),\omega;O')=J(\Lambda\cdot x;O').\]
For all $n$ and almost all $\omega$, set $u_{\e,n}^\omega:=t_n\Lambda_n\cdot x+\e t_n\varphi_{\Lambda_n}(\cdot/\e,\omega)$. By Step~1 of the proof of Lemma~\ref{lem:equZc}, for any bounded Lipschitz domains $O\subset\R^d$, there exists a sequence $(v_{\e,n}^\omega)_\e\subset W^{1,p}_0(O;\R^m)$ such that $v_{\e,n}^\omega\cvf{}0$ in $W^{1,p}(O;\R^m)$ and $I_\e(u_{\e,n}^\omega+v_{\e,n}^\omega,\omega;O)\to |O|\overline W_{t_n}(\Lambda_n)$.

By a diagonalization argument, we then conclude that for almost all $\omega$ and all bounded Lipschitz domains $O\subset\R^d$ there exist sequences $(u_\e)_\e\subset W^{1,p}_\loc(\R^d;\R^m)$ and $(v_\e)_\e\subset W^{1,p}_0(O;\R^m)$ such that $u_\e\cvf{}\Lambda\cdot x$ and $v_\e\cvf{}0$ in $W^{1,p}$, and such that $I_\e(u_\e+v_\e,\omega;O)\to I(\Lambda\cdot x;O)$ and $J_\e(u_\e,\omega;O')\to J(\Lambda\cdot x;O')$ for all bounded domains $O'\subset \R^d$.

Now define $\Omega'\subset\Omega_1$ as a subset of maximal probability such that this result holds for all $\Lambda\in \Q^{m\times d}\cap\inter\dom\overline V$ and all $\omega\in\Omega'$. Arguing as in the end of Step~1 of the proof of Proposition~\ref{prop:gammasupN}, and using the continuity of both $\overline W$ and $\overline V$ in the interior of the domain (see Lemma~\ref{lem:equZc}), the conclusion follows.

\medskip
\step2 Recovery sequence for continuous piecewise affine functions.

We now show that, for any $\omega\in\Omega'$, any bounded Lipschitz domain $O\subset\R^d$, and any continuous piecewise affine function $u$ on $O$ with $\nabla u\in\inter\dom\overline V$ pointwise, there exist a sequence $(u_\e)_\e\subset W^{1,p}(O;\R^m)$ and a sequence $(v_\e)_\e\subset W^{1,p}_0(O;\R^m)$ such that $u_\e\cvf{} u$ and $v_\e\cvf{}0$ in $W^{1,p}(O;\R^m)$, and such that $I_\e(u_\e+v_\e,\omega;O)\to I(u;O)$ and $J_\e(u_\e,\omega;O)\to J(\Lambda\cdot x;O)$. This follows from an immediate adaptation of Step~2 of the proof of Proposition~\ref{prop:gammasupN}. Again, the functional $I_\e$ is not convex, but we may use the bound $W\le C(1+V)$ at all points where the cut-off functions are different from $1$ or $0$, and use the convexity of $V$ to estimate the corresponding error terms. We leave the details to the reader.

\medskip
\step3 Recovery sequence for general functions.

We show that, for all $\omega\in\Omega'$, all strongly star-shaped bounded Lipschitz domains $O\subset\R^d$, and all $u\in W^{1,p}(O;\R^m)$, there exist a sequence $(u_\e)_\e\subset W^{1,p}(O;\R^m)$ and a sequence $(v_\e)_\e\subset W^{1,p}_0(O;\R^m)$ such that $u_\e\cvf{}u$ and $v_\e\cvf{} 0$ in $W^{1,p}(O;\R^m)$, and such that $I_\e(u_\e+v_\e,\omega;O)\to I(u;O)$ and $J_\e(u_\e,\omega;O)\to J(u;O)$. Let $O\subset \R^d$ be some fixed strongly star-shaped bounded Lipschitz domain. By the $\Gamma$-$\liminf$ inequality of Proposition~\ref{prop:gammaliminfnc}, we can restrict attention to those $u\in W^{1,p}(O;\R^m)$ that satisfy
\[ I(u;O)=\int_O\overline W(\nabla u(y))dy<\infty,\]
so that $\nabla u\in\dom\overline V$ almost everywhere. Let $u$ be such a function and let $\omega\in\Omega'$ be fixed.

Let $t\in (0,1)$. Since $O$ is Lipschitz and strongly star-shaped, and since $\overline W$ is lower semicontinuous on $\R^{m\times d}$, continuous on $\inter\dom\overline V$, ru-usc, and satisfies $\overline V\le \overline W\le C(1+\overline V)$ (see indeed Lemmas~\ref{lem:lemma} and~\ref{lem:equZc}), the nonconvex approximation result of Proposition~\ref{prop:approxw}(ii)(c) yields a sequence  $(u_{n})_n$ of continuous piecewise affine functions such that $u_{n}\to u$ (strongly) in $W^{1,p}(O;\R^m)$, $ I( u_{n};O)\to  I(u;O)$ and $J( u_{n};O)\to J(u;O)$ as $n\uparrow\infty$, and such that $\nabla u_{n}\in\inter\dom\overline V$ pointwise. Now Step~2 above gives, for any $n$, sequences $(u_{\e,n})_\e\subset W^{1,p}(O;\R^m)$ and $(v_{\e,n})_\e\subset W^{1,p}_0(O;\R^m)$ such that $u_{\e,n}\cvf{}u_{n}$ and $v_{\e,n}\cvf{}0$ in $W^{1,p}(O;\R^m)$, and such that $I_\e(u_{\e,n}+v_{\e,n},\omega;O)\to I(u_{n};O)$ and $J_\e(u_{\e,n},\omega;O)\to J(u_{n};O)$ as $\e\downarrow0$. The result then follows from a diagonalization argument.
\end{proof}

\section{Proof of the improved results}\label{chap:improved}

\subsection{Subcritical case $1<p\le d$}
In this section, we prove Corollary~\ref{cor:subcritp}. We shall use truncations in place of the Sobolev compact embedding. For such truncation arguments to work, we need to restrict to the scalar case and to assume that the domain is fixed, i.e. $\dom V(y,\cdot,\omega)=\dom M$ for almost all $y,\omega$.

\begin{proof}[Proof of Corollary~\ref{cor:subcritp}]
In the proof of Theorem~\ref{th:conv} and Corollary~\ref{cor:dir}, the Sobolev compact embedding into bounded functions is used both in Step~2 of the proof of Proposition~\ref{prop:gammasupN} and in Step~1 of the proof of Corollary~\ref{cor:dir}(i) (see Section~\ref{chap:liftdirbd}). We only display the argument for Proposition~\ref{prop:gammasupN} (the argument for Corollary~\ref{cor:dir}(i) is similar).

We use the notation of Step~2 of the proof of Proposition~\ref{prop:gammasupN}.
For all $s>0$, define the truncation map $T_s:\R\to\R$ as follows:
\begin{align}\label{eq:truncation1D}
T_s(x)=\sign(x)|x|\wedge s=\begin{cases}
s,&\text{if $x\ge s$;}\\x,&\text{if $-s\le x\le s$;}\\-s,&\text{if $x\le-s$;}
\end{cases}
\end{align}
and for all $s>0$ consider the following $s$-truncation of $u_{\e,\kappa,r,\eta}$:
\begin{align}\label{eq:deftruncueps}
u_{\e,\kappa,r,\eta,s}:=T_s(u_{\e,\kappa,r,\eta}-u_{\kappa,r})+u_{\kappa,r}\in W^{1,p}(O;\R).
\end{align}
Since $|tu_{\e,\kappa,r,\eta,s}-u|\le s+|tu_{\kappa,r}-u|$, we may replace~\eqref{eq:Linftyconvaffine} by
\[\lim_{t\uparrow1}\limsup_{r\downarrow0}\limsup_{\kappa\downarrow0}\limsup_{\eta\downarrow0}\limsup_{s\downarrow0}\limsup_{\e\downarrow0}\|tu_{\e,\kappa,r,\eta,s}-u\|_{\Ld^\infty(O)}=0.\]
Since $\nabla u_{\e,\kappa,r,\eta,s}=T_s'(u_{\e,\kappa,r,\eta}-u_{\kappa,r})\nabla u_{\e,\kappa,r,\eta}+(1-T_s'(u_{\e,\kappa,r,\eta}-u_{\kappa,r}))\nabla u_{\kappa,r}$, we deduce by convexity, noting that $T_s'$ takes values in $[0,1]$,
\begin{align}
J_\e(tu_{\e,\kappa,r,\eta,s},\omega;O)&\le \int_OT_s'(u_{\e,\kappa,r,\eta}(y)-u_{\kappa,r}(y))V(y/\e,t\nabla u_{\e,\kappa,r,\eta},\omega)dy\label{eq:useconvextrunc}\\
&\hspace{2cm}+\int_O(1-T_s'(u_{\e,\kappa,r,\eta}(y)-u_{\kappa,r}(y)))V(y/\e,t\nabla u_{\kappa,r},\omega)dy\nonumber\\
&\le J_\e(tu_{\e,\kappa,r,\eta},\omega;O)+|\{y\in O:|u_{\e,\kappa,r,\eta}(y)-u_{\kappa,r}(y)|>s\}|\max_{1\le l\le k}M(t\Lambda_l)\nonumber\\
&\le J_\e(tu_{\e,\kappa,r,\eta},\omega;O)+s^{-p}\|u_{\e,\kappa,r,\eta}-u_{\kappa,r}\|_{\Ld^p(O)}^p\left((1-t)M(0)+\max_{1\le l\le k}M(\Lambda_l)\right).\nonumber
\end{align}
Since by definition (and by the Rellich-Kondrachov theorem) we have $u_{\e,\kappa,r,\eta}\to u_{\kappa,r}$ (strongly) in $\Ld^p(O)$ as $\e\downarrow0$, since the $\Lambda_l$'s all belong to $\dom \overline V$, and since by assumption $\dom M=\dom\overline V$, we deduce, combining this with~\eqref{eq:gammalimsupglueup}, that
\begin{align*}
\lim_{t\uparrow1}\limsup_{r\downarrow0}\limsup_{\kappa\downarrow0}\limsup_{\eta\downarrow0}\limsup_{s\downarrow0}\limsup_{\e\downarrow0}J_\e(tu_{\e,\kappa,r,\eta,s},\omega;O)\le J(u;O).
\end{align*}
The rest of the proof is unchanged.
\end{proof}

\subsection{Minimal soft buffer zone for Dirichlet boundary data}
In this section, we prove Corollary~\ref{cor:dir-bis}. In view of the error term~\eqref{eq:deferrorcorspeed} in the proof of Corollary~\ref{cor:dir}, it seems that the speed of convergence of $\eta$ to $0$ with respect to $\e$ must depend quantitatively on the speed of convergence of $w_\e$ to $0$ in $\Ld^\infty(O;\R^m)$. In the case when the target function is affine ($\Lambda \cdot x$, say), then $w_\e:=\e \varphi_\Lambda(\cdot/\e,\cdot)$ is the rescaling of the corrector and its convergence to zero is strictly related to the sublinearity of $\varphi_\Lambda$ at infinity, cf. Lemma~\ref{lem:sublincor}.
Even in the linear scalar case when  $V(y,\Lambda)=\Lambda \cdot A(y)\Lambda$ for some matrix random field $A$, this sublinear growth can be arbitrary, and we expect that
for all $\gamma<1$, there exists a field $A$ such that $\E[|\varphi_\Lambda(x)|^2]^\frac{1}{2}\sim |x|^\gamma$ as $|x|\gg 1$ (see recent results in~\cite{Glotto-Neukamm-15a} and the example of Gaussian fields with non-integrable covariance).
Yet, if instead of using the corrector $\varphi_\Lambda$ itself, which is in general not stationary and well-behaved, we use a proxy that is stationary, then the size of the buffer zone can be (optimally) reduced, at least for affine target functions, as the following proposition shows.
\begin{prop}\label{prop:statcorbuffer}
If for all $\Lambda\in \R^{m\times d}$ we have
\begin{align}\label{eq:Whomstat}
\overline V(\Lambda)=\inf_{\phi\in W^{1,p}(\Omega;\R^m)}\E[V(0,\Lambda+D\phi,\cdot)],
\end{align}
then the conclusion of Corollary~\ref{cor:dir-bis} holds (and we can further replace $\theta\e$ by any sequence $\eta_\e\downarrow0$ satisfying $\liminf_\e\eta_\e/\e>0$).
\qed
\end{prop}
Identity~\eqref{eq:Whomstat} is essentially a regularity statement on quasi-minimizers of $f\mapsto \E[V(0,\Lambda+f,\cdot)]$ on $\Ld^p_\pot(\Omega)^m$. By Poincaré's inequality, periodic gradients with mean-value zero are gradients of periodic functions, and hence in that case the space $F^p_\pot(\Omega)$ coincide with $\{D\phi:\phi\in W^{1,p}_\tau(\Omega)^m\}$, so that~\eqref{eq:Whomstat} is trivially satisfied. This already proves Corollary~\ref{cor:dir-bis} under the additional assumption~(1).

On the other hand, the following result shows that~\eqref{eq:Whomstat} is also satisfied in the scalar case $m=1$ if the domain of $V$ is fixed, in which case truncations are available. This proves Corollary~\ref{cor:dir-bis} under the additional assumption~(2).
\begin{lem}\label{lem:reg-quasi}
If $m=1$ and if $\dom V(y,\cdot,\omega)=\dom M$ is open for almost all $y,\omega$, then \eqref{eq:Whomstat} holds true for all $\Lambda$.\qed
\end{lem}
\begin{proof}[Proof of Proposition~\ref{prop:statcorbuffer}]
For all $\Lambda\in\R^{m\times d}$, by assumption~\eqref{eq:Whomstat}, for all $\delta>0$, there exists a stationary random field $\varphi_{\Lambda,\delta}\in W^{1,p}(\Omega;\R^m)$ such that $\E[V(0,\Lambda+\nabla\varphi_{\Lambda,\delta}(0,\cdot),\cdot)]\le\overline V(\Lambda)+\delta$. Set $u_{\e}^{\Lambda,\delta,\omega}:=\e\varphi_{\Lambda,\delta}(\cdot/\e,\omega)$. By stationarity, for almost all $\omega$, the Birkhoff-Khinchin ergodic theorem asserts that, for any bounded domain $O\subset\R^d$,
\begin{align}\label{eq:convepsunabla}
\lim_{\e\downarrow0}\int_O|u_\e^{\Lambda,\delta,\omega}/\e|^p=|O|\E[|\varphi_{\Lambda,\delta}|^p],\qquad\lim_{\e\downarrow0}\int_O|\nabla u_\e^{\Lambda,\delta,\omega}|^p=|O|\E[|\nabla\varphi_{\Lambda,\delta}|^p].
\end{align}
Let $O\subset\R^d$ be some fixed bounded domain. For $\eta>0$, set $O_\eta:=\{x\in O:\dist(x,\partial O)>\eta\}$. For any sequence $\eta_\e\downarrow0$, \eqref{eq:convepsunabla} yields
\begin{align}\label{eq:convepsunabla2}
\lim_{\e\downarrow0}\int_{O\setminus O_{\eta_\e}}|u_\e^{\Lambda,\delta,\omega}/\e|^p=0=\lim_{\e\downarrow0}\int_{O\setminus O_{\eta_\e}}|\nabla u_\e^{\Lambda,\delta,\omega}|^p.
\end{align}
Fix such a sequence $\eta_\e\downarrow0$. As in Step~1 of the proof of Proposition~\ref{prop:gammasupN}, for all $\Lambda\in\R^{m\times d}$, we obtain the following, for some subset $\Omega_\Lambda\subset\Omega$ of maximal probability: for all $\omega\in\Omega_\Lambda$ and all $\delta>0$, there is a sequence $(u_\e^{\Lambda,\delta,\omega})_\e\subset W^{1,p}(O;\R^m)$ such that $u_\e^{\Lambda,\delta,\omega}\cvf{}0$ in $W^{1,p}(O;\R^m)$ as $\e\downarrow0$, $\limsup_\e J_\e(\Lambda\cdot x+u_\e^{\Lambda,\delta,\omega},\omega;O)\le J(\Lambda\cdot x;O)+\delta$, and such that~\eqref{eq:convepsunabla2} is satisfied.

Let $\Lambda$ and $\omega\in\Omega_\Lambda$ be fixed, and let $(u_\e^{\Lambda,\delta,\omega})_\e$ be as above. For all $\e>0$, choose a smooth cut-off function $\chi_\e$ with values in $[0,1]$, equal to $1$ on $O_{\eta_\e}=\{x\in O:d(x,\partial O)>\eta_\e\}$, vanishing outside $O$, and with $|\nabla \chi_\e|\le C'/\eta_\e$ for some constant $C'$. Defining $v_\e^{\Lambda,\delta,\omega}:=\chi_\e u_\e^{\Lambda,\delta,\omega}\in W^{1,p}_0(O;\R^m)$, we obtain
\begin{align*}
J_\e^{\eta_\e}(\Lambda\cdot x+v_\e^{\Lambda,\delta,\omega},\omega;O)&=J_\e(\Lambda\cdot x+v_\e^{\Lambda,\delta,\omega},\omega;O_{\eta_\e})+\int_{O\setminus O_{\eta_\e}}|\Lambda+\nabla v_\e^{\Lambda,\delta,\omega}|^p\\
&\le J_\e(\Lambda\cdot x+v_\e^{\Lambda,\delta,\omega},\omega;O)+3^{p-1}|\Lambda|^p|O\setminus O_{\eta_\e}|\\
&\hspace{2cm}+3^{p-1}\int_{O\setminus O_{\eta_\e}}(|\nabla u_\e^{\Lambda,\delta,\omega}|^p+|C'u_\e^{\Lambda,\delta,\omega}/\eta_\e|^p),
\end{align*}
and hence, if the sequence $\eta_\e\downarrow0$ is further chosen such that $\liminf_\e\eta_\e/\e>0$,
\begin{align*}
\limsup_{\e\downarrow0}J_\e^{\eta_\e}(\Lambda\cdot x+v_\e^{\Lambda,\delta,\omega},\omega;O)&\le \limsup_{\e\downarrow0}J_\e(\Lambda\cdot x+v_\e^{\Lambda,\delta,\omega},\omega;O)\le J(\Lambda\cdot x;O)+\delta.
\end{align*}
Therefore, $\limsup_\delta\limsup_\e J_\e^{\eta_\e}(\Lambda\cdot x+v_\e^{\Lambda,\delta,\omega},\omega;O)\le J(\Lambda\cdot x;O)$. Combined with the $\Gamma$-$\liminf$ inequality of Proposition~\ref{prop:gammaliminfneu} and a diagonalization argument, this proves the result.
\end{proof}

\begin{proof}[Proof of Lemma~\ref{lem:reg-quasi}]
We split the proof into two steps.

\medskip
\step1 Preliminary: we claim that it suffices to prove that, for all $\Lambda\in \dom\overline V$,
\begin{align}\label{eq:resstep1dens}
\limsup_{t\uparrow1}\inf_{\phi\in W^{1,p}(\Omega)}\E[V(0,t\Lambda+D \phi,\cdot)]\le \overline V(\Lambda).
\end{align}

Define $\overline V'(\Lambda):=\inf_{\phi\in W^{1,p}(\Omega)}\E[V(0,\Lambda+D\phi,\cdot)]$. By definition $\overline V'(\Lambda)\ge\overline V(\Lambda)$ for all $\Lambda$, and hence property~\eqref{eq:resstep1dens} together with the lower semicontinuity of $\overline V$ directly yields $\overline V(\Lambda)=\lim_{t\uparrow1}\overline V'(t\Lambda)$ for all $\Lambda$ (and in particular the limit exists). Since $\overline V'$ is obviously convex, it is continuous on the interior of its domain. Since the domain is assumed to be open, this yields $\overline V(\Lambda)=\overline V'(\Lambda)$ for all $\Lambda$.

\medskip
\step2 Proof of~\eqref{eq:resstep1dens}.

Let $\Lambda\in \dom\overline V$ be fixed. Lemma~\ref{lem:sublincor} gives a measurable corrector $u:=\varphi_\Lambda\in\Mes(\Omega;W^{1,p}_{\loc}(\R^d))$ such that $\nabla u\in F^p_\pot(\Omega)$ and $\overline V(\Lambda)=\E[V(0,\Lambda+ \nabla u(0,\cdot),\cdot)]$. For all $R>r>0$, choose a smooth cut-off function $\chi_{R,r}$ taking values in $[0,1]$, equal to $1$ on $Q_{R-r}$, vanishing outside $Q_{R}$ and satisfying $|\nabla \chi_{R,r}|\le 2/r$. Also recall the definition~\eqref{eq:truncation1D} of the truncation $T_s$. We then set
\[u_{R}(x,\omega)=u(x,\omega)-\fint_{Q_R}u(\cdot,\omega),\qquad v_{R,r}^s(x,\omega)=\chi_{R,r}(x)\,T_su_R(x,\omega),\]
and
\[w_{R,r}^s(x,\omega)=\frac1{|Q_R|}\int_{\R^d}v_{R,r}^s(x+y,\tau_y\omega)dy=\fint_{-x+Q_R}v_{R,r}^s(x+y,\tau_y\omega)dy.\]
Clearly, $w_{R,r}^s$ is well-defined, stationary, and belongs to $W^{1,p}(\Omega)$, with
\[\nabla w_{R,r}^s(x,\omega)=\fint_{-x+Q_R}\nabla v_{R,r}^s(x+y,\tau_y\omega)dy.\]
Let $t\in[0,1)$. By Jensen's inequality,
\begin{align*}
K_{R,r}^s(t):=\E[V(0,t\Lambda+t\nabla w_{R,r}^s(0,\cdot),\cdot)]&=\E\left[V\left(0,t\Lambda+t\fint_{Q_R}\nabla v_{R,r}^s(y,\tau_y\cdot)dy,\cdot\right)\right]\\
&\le\E\left[\fint_{Q_R}V(0,t\Lambda+t\nabla v_{R,r}^s(y,\tau_y\cdot),\cdot)dy\right],
\end{align*}
and hence, by stationarity and the Fubini theorem,
\begin{align*}
K_{R,r}^s(t)&\le\E\left[\fint_{Q_R}V(y,t\Lambda+t\nabla v_{R,r}^s(y,\cdot),\cdot)dy\right].
\end{align*}
Decomposing
\begin{align*}
t\Lambda+t\nabla v_{R,r}^s(y,\omega)&=t\Lambda+t\chi_{R,r}(y)\nabla T_su_R(y,\omega)+(1-t)\frac t{1-t}\nabla \chi_{R,r}(y) T_su_R(y,\omega)\\
&=t\chi_{R,r}(y) T_s'(u_R(y,\omega))(\Lambda+\nabla u(y,\omega))+t(1-\chi_{R,r}(y)) T_s'(u_R(y,\omega))\Lambda\\
&\qquad+t(1-T_s'(u_R(y,\omega)))\Lambda+(1-t)\frac t{1-t}\nabla \chi_{R,r}(y) T_su_R(y,\omega),
\end{align*}
with $T_s'$ taking values in $[0,1]$, we may then bound by convexity
\begin{align}\label{eq:K_R,r^s(t)}
K_{R,r}^s(t)&\le\E\left[\fint_{Q_R}V(y,\Lambda+\nabla u(y,\cdot),\cdot)dy\right]+(1-t)E_{R,r}^s(t)\nonumber\\
&\qquad+M(\Lambda)\fint_{Q_R}(1-\chi_{R,r})+M(\Lambda)\E\left[\fint_{Q_R}(1-T_s'(u_R(y,\cdot)))\right],
\end{align}
where the error term reads
\[E_{R,r}^s(t)=\E\left[\fint_{Q_R}M\left(\frac{t}{1-t}\nabla \chi_{R,r}(y) T_su_R(y,\omega)\right)dy\right].\]
By stationarity of $\nabla u$, note that
\begin{align}\label{eq:K_R,r^s(t)-1}
\E\left[\fint_{Q_R}V(y,\Lambda+\nabla u(y,\cdot),\cdot)dy\right]=\E[V(0,\Lambda+\nabla u(0,\cdot),\cdot)]=\overline V(\Lambda).
\end{align}
For the error term, note that
\begin{align}\label{eq:K_R,r^s(t)-2}
\left\|\frac t{1-t}\nabla \chi_{R,r}(\cdot)T_su_R(\cdot,\omega)\right\|_{\Ld^\infty(O)}\le\frac {2t}{1-t}\frac sr.
\end{align}
It remains to treat the last two terms of~\eqref{eq:K_R,r^s(t)}. Noting that $\fint_{Q_R}(1-\chi_{R,r})=R^{-d}(R^d-(R-r)^d)\le dr/R$, we obtain
\begin{align}\label{eq:K_R,r^s(t)-3}
\fint_{Q_R}(1-\chi_{R,r})+\E\left[\fint_{Q_R}(1-T_s'(u_R(y,\cdot)))\right]&\le \frac{dr}R+\E\left[\fint_{Q_R}\mathds1_{|u_R(y)|\ge s}dy\right]\\
&\le \frac{dr}R+\int_Q\p\left[\frac1R\left|u(Ry,\cdot)-\fint_Qu(Rz,\cdot)dz\right|\ge \frac sR\right]dy.\nonumber
\end{align}
Lemma~\ref{lem:sublincor} (together with the Rellich-Kondrachov theorem) gives $\frac1R|u(R\cdot,\omega)-\fint_Qu(Rz,\omega)dz|\to0$ (strongly) in $\Ld^p(Q)$ as $R\uparrow\infty$, for almost all $\omega$. Hence up to an extraction in $R$ (implicit in the sequel) we deduce that, for almost all $y\in Q$, $\frac1R|u(Ry,\cdot)-\fint_Qu(Rz,\cdot)dz|\to0$ almost surely. Since almost sure convergence implies convergence in probability, we deduce by dominated convergence, for all $\e>0$, 
\[\lim_{R\uparrow\infty}\int_Q\p\left[\frac1R\left|u(Ry,\cdot)-\fint_Qu(Rz,\cdot)dz\right|\ge \e\right]dy=0.\]
A diagonalization argument then gives a sequence $\e_R\downarrow0$ such that
\begin{align}\label{eq:K_R,r^s(t)-3'}
\lim_{R\uparrow\infty}\int_Q\p\left[\frac1R\left|u(Ry,\cdot)-\fint_Qu(Rz,\cdot)dz\right|\ge \e_R\right]dy=0.
\end{align}
Choose $s=s_R:=R\e_R$ and $r=r_R:=R\sqrt{\e_R}$. By assumption, there exists some $\delta>0$ with $\adh B_\delta\subset\inter\dom M$. By~\eqref{eq:K_R,r^s(t)-2}, for all $t\in[0,1)$, there is some $R_t>0$ such that for all $R>R_t$
\begin{align}\label{eq:K_R,r^s(t)-2'}
\left\|\frac t{1-t}\nabla \chi_{R,r_R}(\cdot)T_{s_R}u_R(\cdot,\omega)\right\|_{\Ld^\infty(O)}\le\frac {2t}{1-t}\sqrt{\e_R}<\delta.
\end{align}
Combining this with~\eqref{eq:K_R,r^s(t)}, \eqref{eq:K_R,r^s(t)-1}, \eqref{eq:K_R,r^s(t)-3}, \eqref{eq:K_R,r^s(t)-3'}, and noting that $M(\Lambda)<\infty$ follows from the choice $\Lambda\in\dom\overline V$, we obtain
\[\limsup_{t\uparrow1}\limsup_{R\uparrow\infty}K_{R,r_R}^{s_R}(t)\le \overline V(\Lambda),\]
and the result~\eqref{eq:resstep1dens} follows.
\end{proof}

\subsection{Approximation by periodization in law}
This last section is devoted to the proof of Corollary~\ref{cor:per}. For that purpose, we first prove the following version of the $\Gamma$-$\limsup$ inequality of Proposition~\ref{prop:gammasupN} in expectation. As Lemma~\ref{lem:sublincor} only gives almost sure (and not $\Ld^1(\Omega)$) control of the sublinearity of the corrector, we need to use truncations to prove this result (and therefore restrict to the scalar case with fixed domain).

\begin{prop}[$\Gamma$-$\limsup$ inequality in expectation with Neumann boundary data]\label{prop:gammasupNexpect}
Let $V,J_\e,J,M$ be as in Theorem~\ref{th:conv} for some $p>1$. Also assume that we have $m=1$ and $\dom V(y,\cdot,\omega)=\dom M$ for almost all $y,\omega$.
Then, for all bounded Lipschitz domains $O\subset\R^d$ and all $u\in W^{1,p}(O)$, there exists a sequence $(u_\e)_\e\subset \Mes(\Omega;W^{1,p}(O))$ such that $\E[J_\e(u_\e,\cdot;O)]\to J(u;O)$ and $u_\e\to u$ in $\Ld^\infty(\Omega;\Ld^\infty(O))$.
\qed
\end{prop}

\begin{proof}
We split the proof into three steps.

\medskip
\step1 Preliminary.

We claim that, for all bounded domains $O\subset\R^d$ and all $\Lambda\in\inter\dom\overline V$, there exists a sequence $(u_{\Lambda,\e})_\e \subset \Mes(\Omega;W^{1,p}(O))$ with $u_{\Lambda,\e}(\cdot,\omega)\cvf{}\Lambda\cdot x$ weakly in $W^{1,p}(O)$ for almost all $\omega$ and $\E[J_\e(u_{\Lambda,\e},\cdot;O)]\to J(\Lambda\cdot x;O')$ for all subdomains $O'\subset O$. This is indeed a simple reformulation of Lemma~\ref{lem:sublincor} with the notation $u_{\Lambda,\e}(x,\omega):=\e\varphi_\Lambda(x/\e,\omega)$.

\medskip
\step2 Recovery sequence for continuous piecewise affine functions.

Let $O\subset\R^d$ be a bounded Lipschitz domain and let $u$ be a continuous piecewise affine function on $O$ such that $\nabla u\in\inter\dom\overline V$ pointwise. We prove the existence of a sequence $(u_\e)_\e\subset\Mes(\Omega;W^{1,p}(O))$ with $u_\e\to u$ in $\Ld^\infty(\Omega;\Ld^\infty(O))$ and $\E[J_\e(u_\e,\cdot;O)]\to J(u;O)$.

Let us modify $u$ as in Step~2 of the proof of Proposition~\ref{prop:gammasupN} to make its variations smoother (using Proposition~\ref{prop:approxaffine}), and use the same notation. By Step~1, for all $1\le i\le n_{\kappa,r}$, there exists a sequence $(u^i_{\e,\kappa,r})_\e\subset \Mes(\Omega;W^{1,p}_{\loc}(\R^d))$ with $u^i_{\e,\kappa,r}(\cdot,\omega)\cvf{}c_{\kappa,r}^i+\Lambda_{\kappa,r}^i\cdot x$ in $W^{1,p}_{\loc}(\R^d)$ for almost all $\omega$ and such that, for all Lipschitz subdomains $O'\subset O$, we have $\E[J_\e(u^i_{\e,\kappa,r},\cdot;O')]\to J(\Lambda^i_{\kappa,r}\cdot x;O')$. Consider the same partition of unity $\sum_{i=1}^{n_{\kappa,r}}\chi_{\kappa,r,\eta}^i$ as in Step~2 of the proof of Proposition~\ref{prop:gammasupN}, and also recall the definition~\eqref{eq:truncation1D} of the truncation $T_s$. We now set, for $s>0$,
\[u_{\e,\kappa,r,\eta,s}:=u_{\kappa,r}+\sum_{i=1}^{n_{\kappa,r}}\chi_{\kappa,r,\eta}^iT_s(u_{\e,\kappa,r}^i-(c_{\kappa,r}^i+\Lambda_{\kappa, r}^i\cdot x))\quad\in\Mes(\Omega;W^{1,p}(O)).\]
On the one hand, since $|tu_{\e,\kappa,r,\eta}- u|\le s+|tu_{\kappa,r}-u|$ pointwise, we deduce
\begin{align}\label{eq:Linftyconvaffine7}
\lim_{t\uparrow1}\limsup_{r\downarrow0}\limsup_{\kappa\downarrow0}\limsup_{\eta\downarrow0}\limsup_{s\downarrow0}\limsup_{\e\downarrow0}\|tu_{\e,\kappa,r,\eta}- u\|_{\Ld^\infty(\Omega;\Ld^\infty(O))}=0.
\end{align}
On the other hand, since
\begin{align*}
t\nabla u_{\e,\kappa,r,\eta,s}&=\sum_{i=1}^{n_{\kappa,r}}t\chi_{\kappa,r,\eta}^iT_s'(u_{\e,\kappa,r}^i-(c_{\kappa,r}^i+\Lambda_{\kappa,r}^i\cdot x))\nabla u_{\e,\kappa,r}^i\\
&\qquad+\sum_{i=1}^{n_{\kappa,r}}t\chi_{\kappa,r,\eta}^i(1-T_s'(u_{\e,\kappa,r}^i-(c_{\kappa,r}^i+\Lambda_{\kappa,r}^i\cdot x)))\Lambda_{\kappa,r}^i+(1-t)S_{\e,\kappa,r,\eta,s,t},
\end{align*}
where we have set
\[S_{\e,\kappa,r,\eta,s,t}:=\frac t{1-t}\sum_{i=1}^{n_{\kappa,r}}T_s(u_{\e,\kappa,r}^i-(c_{\kappa,r}^i+\Lambda_{\kappa,r}^i\cdot x))\nabla \chi_{\kappa,r,\eta}^i+\frac t{1-t}\sum_{i=1}^{n_{\kappa,r}}\chi_{\kappa,r,\eta}^i(\nabla u_{\kappa,r}-\Lambda_{\kappa,r}^i),\]
we deduce by convexity
\begin{align}\label{eq:decompexpecterrglue}
\E[J_\e(tu_{\e,\kappa,r,\eta,s},\cdot;O)]&\le(1-t)E_{\e,\kappa,r,\eta,s,t}+\sum_{i=1}^{n_{\kappa,r}}\E[J_\e(u_{\e,\kappa,r}^i,\cdot;O_{\kappa,r,\eta}^{i+})]\\
&\qquad+\sum_{i=1}^{n_{\kappa,r}}M(\Lambda_{\kappa,r}^i)\E[|\{y\in O_{\kappa,r,\eta}^{i+}:|u_{\e,\kappa,r}^i(y,\omega)-(c_{\kappa,r}^i+\Lambda_{\kappa,r}^i\cdot y)|\ge s\}|]\nonumber
\end{align}
where the error reads
\begin{align*}
E_{\e,\kappa,r,\eta,s,t}&=\E\left[\int_OV(y/\e,S_{\e,\kappa,r,\eta,s,t}(y,\cdot),\cdot)dy\right].
\end{align*}
For all $i$, set $N_{\kappa,r,\eta}^{i}:=\{j: j\ne i,\,O^{j+}_{\kappa,r,\eta}\cap O^{i+}_{\kappa,r,\eta}\ne\varnothing\}$.
We estimate $S_{\e,\kappa,r,\eta,s,t}$ for all $y,\omega$ as follows:
\begin{align*}
|S_{\e,\kappa,r,\eta,s,t}(y,\omega)|&\le \frac t{1-t}\frac{C'n_{\kappa,r}}\eta s+\frac{t}{1-t}\sup_{1\le i\le n_{\kappa,r}}\sup_{j\in N^i_{\kappa,r,\eta}}|\Lambda_{\kappa,r}^j-\Lambda_{\kappa,r}^i|,
\end{align*}
and hence, since $\limsup_{\eta\downarrow0}\sup_{j\in N^i_{\kappa,r,\eta}}|\Lambda^j_{\kappa,r}-\Lambda^i_{\kappa,r}|\le\kappa$ for all $i$,
\begin{align}\label{eq:decompexpecterrglue-1}
\limsup_{\kappa\downarrow0}\limsup_{\eta\downarrow0}\limsup_{s\downarrow0}\limsup_{\e\downarrow0} \|S_{\e,\kappa,r,\eta,s,t}\|_{\Ld^\infty(\Omega;\Ld^\infty(O))}=0.
\end{align}
The last term of the right-hand side of~\eqref{eq:decompexpecterrglue} is estimated by
\begin{align*}
\E[|\{y\in O_{\kappa,r,\eta}^{i+}:|u_{\e,\kappa,r}^i(y,\omega)-(c_{\kappa,r}^i+\Lambda_{\kappa,r}^i\cdot y)|\ge s\}|]&\le \int_O\p[|u_{\e,\kappa,r}^i(y,\omega)-(c_{\kappa,r}^i+\Lambda_{\kappa,r}^i\cdot y)|\ge s]dy.
\end{align*}
For almost all $\omega$, the Rellich-Kondrachov theorem shows that $u_{\e,\kappa,r}^i(\cdot,\omega)\to c_{\kappa,r}^i+\Lambda_{\kappa,r}^i\cdot x$ (strongly) in $\Ld^p(O)$, and hence, up to an extraction in $\e$ (implicit in the sequel), this convergence holds almost surely, almost everywhere on $O$. The dominated convergence theorem then yields
\begin{align}\label{eq:decompexpecterrglue-2}
\limsup_{\e\downarrow0}\E[|\{y\in O_{\kappa,r,\eta}^{i+}:|u_{\e,\kappa,r}^i(y,\omega)-(c_{\kappa,r}^i+\Lambda_{\kappa,r}^i\cdot y)|\ge s\}|]&=0.
\end{align}
Finally we note that by construction,
\[\lim_{\e\downarrow0}\sum_{i=1}^{n_{\kappa,r}}\E[J_\e(u_{\e,\kappa,r}^i,\cdot;O_{\kappa,r,\eta}^{i+})]=\sum_{i=1}^{n_{\kappa,r}}J(\Lambda_{\kappa,r}^i\cdot x;O_{\kappa,r,\eta}^{i+})=|O_{\kappa,r,\eta}^{i+}|\overline V(\Lambda_{\kappa,r}^i).\]
Hence, arguing as in Step~2 of the proof of Proposition~\ref{prop:gammasupN}, combining this with~\eqref{eq:decompexpecterrglue}, \eqref{eq:decompexpecterrglue-1} and~\eqref{eq:decompexpecterrglue-2}, and recalling that $\dom\overline V=\dom M$, we obtain
\begin{align}\label{eq:decompexpecterrglue-3}
\lim_{t\uparrow1}\limsup_{r\downarrow0}\limsup_{\kappa\downarrow0}\limsup_{\eta\downarrow0}\limsup_{s\downarrow0}\limsup_{\e\downarrow0}\E[J_\e(tu_{\e,\kappa,r,\eta,s},\cdot;O)]\le J(u;O),
\end{align}
so that the conclusion follows as in Step~2 of the proof of Proposition~\ref{prop:gammasupN} (note that by Fatou's lemma the $\Gamma$-$\liminf$ inequality also holds in expectation).

\medskip
\step3 Conclusion. The existence of recovery sequences for general functions can now be deduced as in Step~3 of the proof of Proposition~\ref{prop:gammasupN}, using the result of Step~2 above.
\end{proof}

We now turn to the proof of Corollary~\ref{cor:per} itself. In case~(1), we use Proposition~\ref{prop:gammasupNexpect} to control the energy close to the boundary of the cube, where the periodization in law may have modified the integrand (which however is the same in law locally and can therefore be handled by taking the expectation). In case~(2), we use the particular geometric structure to explicitly estimate the energy density near the boundary and prove its uniform integrability.

\begin{proof}[Proof of Corollary~\ref{cor:per}]
Let $(V^R)_{R>0}$ be an admissible periodization in law for $V$ in the sense of Definition~\ref{def:admperlaw}, and for all $\e>0$ and all $\Lambda \in \R^{m\times d}$ denote by $J^{\per}_\e(\cdot,\Lambda,\cdot)$ the following random integral functional on the unit cube $Q=[-\frac{1}{2},\frac{1}{2})^d$:
$$
J^{\per}_\e(u,\Lambda,\omega):=\int_Q V^{1/\e}({y}/{\e},\Lambda+\nabla u(y),\omega)dy, \qquad u\in W^{1,p}_{\per}(Q;\R^m).
$$
We split the proof into three main steps.

\medskip
\step{1} $\Gamma$-$\liminf$ inequality. The results of this step hold without the additional assumptions (1) and (2). Let $\Omega''\subset\Omega'$ be a subset of maximal probability such that the almost sure property of Definition~\ref{def:admperlaw}(ii) holds pointwise on $\Omega''$. We claim that for all $u\in W^{1,p}_\per(Q;\R^m)$, all $u_\e\cvf{}u$ in $W^{1,p}_\per(Q)$ and all $\Lambda_\e\to\Lambda$, we have $\liminf_{\e\downarrow0} J_\e^{\per}(u_\e,\Lambda_\e,\omega)\ge J(u+\Lambda \cdot x;Q)$ for all $\omega\in\Omega''$.

For all $\theta\in(0,1)$ and $\omega\in \Omega''$, we thus have for all $\e>0$ small enough
$$
J_\e^{\per}(u_\e,\Lambda_\e,\omega)\,=\,J_\e(u_\e+\Lambda_\e \cdot x,\omega;Q_\theta)+\int_{Q\setminus Q_\theta} V^{1/\e}({y}/{\e},\Lambda_\e+\nabla u_\e(y),\omega)dy.
$$
By the $\Gamma$-$\liminf$ inequality for $J_\e(\cdot,\omega;Q_\theta)$ (see Proposition~\ref{prop:gammaliminfneu}) and by the non-negativity of $V^{1/\e}$, this turns into
$$
\liminf_{\e\downarrow0} J_\e^{\per}(u_\e,\Lambda_\e,\omega)\,\ge \, J(u+\Lambda \cdot x;Q_\theta),
$$
and the result follows by the arbitrariness of $\theta<1$ and the  monotone convergence theorem.

\medskip
\step{2} $\Gamma$-$\limsup$ inequality under assumption~(1). We assume $m=1$ and $\dom V(y,\cdot,\omega)=\dom M$ for almost all $y,\omega$. We prove the existence of a subset $\Omega'''\subset\Omega''$ of maximal probability such that, for all $u\in W^{1,p}_{\per}(Q)$, all $\Lambda\in\R^{1\times d}$, and all $\omega\in\Omega'''$, there exists a sequence $(u_\e)_\e\subset \Mes(\Omega;W^{1,p}_{\per}(Q))$ such that $u_\e\to0$ in $\Ld^\infty(\Omega;\Ld^\infty(Q))$ and
\[\lim_{t\uparrow1}\liminf_{\e\downarrow0}J_\e^{\per}(tu_\e(\cdot,\omega)+tu,t\Lambda,\omega)=\lim_{t\uparrow1}\limsup_{\e\downarrow0}J_\e^{\per}(tu_\e(\cdot,\omega)+tu,t\Lambda,\omega)= J(u+\Lambda\cdot x;Q).\]
Moreover, the limits $t\uparrow1$ can be dropped if we have $J(\alpha u+\alpha\Lambda\cdot x;Q)<\infty$ for some $\alpha>1$. We split the proof of this step into three parts.

\medskip
\step{2.1} Local recovery sequence for affine functions.

For all $R>0$, define the event
\[\Omega_R:=\{\omega\in\Omega:V^r(\cdot,\cdot,\omega)|_{Q_{\theta_0r}\times\R^{m\times d}}=V(\cdot,\cdot,\omega)|_{Q_{\theta_0r}\times\R^{m\times d}},\,\text{for all $r> R$}\},\]
for some fixed $\theta_0>0$. Definition~\ref{def:admperlaw}(ii) ensures that $\p[\Omega_R]\uparrow1$ as $R\uparrow\infty$. Let $\Lambda\in\R^{1\times d}$, let $O\subset Q$ be a domain of diameter smaller than $\theta_0/2$, and let $R>0$. By the diameter condition, there exists $x_0\in Q$ such that $O\subset x_0+\theta_0Q$. We claim that there is a sequence $(u_\e)_\e\subset \Mes(\Omega;W^{1,p}(O))$ such that $u_\e\to0$ in $\Ld^\infty(\Omega;\Ld^\infty(O))$ and $\E[\mathds1_{\tau_{x_0/\e}\Omega_R}\int_OV^{1/\e}(\cdot/\e,\Lambda+\nabla u_\e,\cdot)]\to \p[\Omega_R]J(\Lambda\cdot x;O)$.

By definition, for all $\omega\in\Omega_R$ we have $V^{1/\e}(\cdot,\cdot,\omega)|_{Q_{\theta_0/\e}\times\R^{m\times d}}=V(\cdot,\cdot,\omega)|_{Q_{\theta_0/\e}\times\R^{m\times d}}$ for all $\e<1/R$. Therefore, applying Proposition~\ref{prop:gammasupNexpect} on $-x_0+O$ restricted to $\Omega_R$ yields the following: there exists a sequence $(u_\e)_\e\subset\Mes(\Omega;W^{1,p}(-x_0+O))$ with $u_\e\to0$ in $\Ld^\infty(\Omega;\Ld^\infty(-x_0+O))$ and
\begin{align*}
\lim_{\e\downarrow0}\E\left[\mathds1_{\Omega_R}\int_{-x_0+O}V^{1/\e}(\cdot/\e,\Lambda+\nabla u_\e,\cdot)\right]&=\lim_{\e\downarrow0}\E\left[\mathds1_{\Omega_R}\int_{-x_0+O}V(\cdot/\e,\Lambda+\nabla u_\e,\cdot)\right]\\
&= \p[\Omega_R]J(\Lambda\cdot x;-x_0+O)=\p[\Omega_R]J(\Lambda\cdot x;O).
\end{align*}
Now define $u_\e':=u_\e(-x_0+\cdot,\tau^{1/\e}_{-x_0/\e}\cdot)\in \Mes(\Omega;W^{1,p}(-x_0+O))$ for all $\e$. By construction $u_\e'\to0$ in $\Ld^\infty(\Omega;\Ld^\infty(O))$, and the stationarity property of Definition~\ref{def:admperlaw}(i) further gives
\begin{align*}
&\lim_{\e\downarrow0}\E\left[\mathds1_{\tau_{x_0/\e}^{1/\e}\Omega_R}\int_OV^{1/\e}(y/\e,\Lambda+\nabla u'_\e(y,\cdot),\cdot)dy\right]\\
=~&\lim_{\e\downarrow0}\E\left[\mathds1_{\Omega_R}\int_OV^{1/\e}(y/\e,\Lambda+\nabla u_\e(-x_0+y,\cdot),\tau^{1/\e}_{x_0/\e}\cdot)dy\right]\\
=~&\lim_{\e\downarrow0}\E\left[\mathds1_{\Omega_R}\int_OV^{1/\e}((-x_0+y)/\e,\Lambda+\nabla u_\e(-x_0+y,\cdot),\cdot)dy\right]\\
=~&\lim_{\e\downarrow0}\E\left[\mathds1_{\Omega_R}\int_{-x_0+O}V^{1/\e}(y/\e,\Lambda+\nabla u_\e(y,\cdot),\cdot)dy\right]= \p[\Omega_R]J(\Lambda\cdot x;O).
\end{align*}

\medskip
\step{2.2} Global recovery sequence for (piecewise) affine functions.

Let $\Lambda\in\R^{1\times d}$ and let $u$ be a continuous piecewise affine $Q$-periodic function on $\R^d$ with $\Lambda+\nabla u\in\inter\dom \overline V$. We claim that there exists a sequence $(u_\e)_\e\subset\Mes(\Omega;W^{1,p}_{\per}(Q))$ such that $u_\e\to u$ in $\Ld^\infty(\Omega;\Ld^\infty(Q))$, and, for almost all $\omega$,
\begin{align}\label{eq:limsupper}
&\lim_{t\uparrow1}\liminf_{\e\downarrow0}J_\e^{\per}(tu_\e(\cdot,\omega)+tu,t\Lambda,\omega)\nonumber\\
=~&\lim_{t\uparrow1}\limsup_{\e\downarrow0}J_\e^{\per}(tu_\e(\cdot,\omega)+tu,t\Lambda,\omega)= J(u+\Lambda\cdot x;Q).
\end{align}

Consider the partition of $Q$ associated with $u$, and let $Q=\biguplus_{l=1}^k O_l$ be a refined partition such that all the $O_l$'s have diameter at most $\theta_0/2$. For all $l$, define $c_l+\Lambda_l\cdot x:=\nabla u|_{O_l}$ with $\Lambda+\Lambda_l\in\inter\dom\overline V$, and choose $x_l\in Q$ such that $O_l\subset x_l+\theta_0Q$. Given $R>0$, define
\[\chi_{\e,R}:=\prod_{l=1}^k\mathds1_{\tau_{x_l/\e}^{1/\e}\Omega_R}.\]
Step~2.1 then gives a sequence $(u_\e^l)_\e\subset\Mes(\Omega;W^{1,p}(O_l))$ such that $u_\e^l\to0$ in $\Ld^\infty(\Omega;\Ld^\infty(O_l))$ and
\begin{align*}
\limsup_{\e\downarrow0}\E\left[\chi_{\e,R}\int_{O_l}V^{1/\e}(\cdot/\e,\Lambda+\Lambda_l+\nabla u_\e^l,\cdot)\right]&\le\lim_{\e\downarrow0}\E\left[\mathds1_{\tau_{x_l/\e}^{1/\e}\Omega_R}\int_{O_l}V^{1/\e}(\cdot/\e,\Lambda+\Lambda_l+\nabla u_\e^l,\cdot)\right]\\
&=\p[\Omega_R]J(u+\Lambda\cdot x;O_l).
\end{align*}
We can then repeat the argument of Step~2 of the proof of Proposition~\ref{prop:gammasupN}, and glue the recovery sequences for the (small) affine parts, using this time Proposition~\ref{prop:approxaffineper} instead of Proposition~\ref{prop:approxaffine} (applied to the function $u$). This allows to deduce the following: there exists a sequence $(u_\e)_\e\subset \Mes(\Omega;W^{1,p}_{\per}(O))$ such that $u_\e\to0$ in $\Ld^\infty(\Omega;\Ld^\infty(Q))$ and
\begin{align}\label{eq:primgammsuplim}
\limsup_{t\uparrow1}\limsup_{\e\downarrow0}\E[\chi_{\e,R}J_\e^{\per}(tu_\e+tu,t\Lambda,\cdot)]\le \p[\Omega_R]J(u+\Lambda\cdot x;Q).
\end{align}

Now the $\Gamma$-$\liminf$ inequality of Step~1 yields $\liminf_t\liminf_\e J_\e^{\per}(u_\e(\cdot,\omega)+u,t\Lambda,\omega)\ge J(u+\Lambda\cdot x;Q)$ for almost all $\omega$. Assume by contradiction that this inequality is strict for all $\omega$ in some set of positive probability (even for some subsequence in $\e,t$). In this case we would find $\e_0,\delta>0$ and $t_0\in(0,1)$ such that $\p[\Omega_{\e_0,t_0}^\delta]>0$, where we have defined the event
\[\Omega_{\e_0,t_0}^\delta:=\{\omega\in\Omega: J_\e^{\per}(tu_\e(\cdot,\omega)+tu,t\Lambda,\omega)>\delta+J(u+\Lambda\cdot x;Q),\,\text{for all $\e<\e_0$, $t>t_0$}\},\]
By definition we have
\begin{align*}
&\liminf_{t\uparrow1}\liminf_{\e\downarrow0}\E[\chi_{\e,R}J_\e^{\per}(tu_\e+tu,t\Lambda,\cdot)]\\
\ge~& J(u+\Lambda\cdot x;Q)\liminf_{\e\downarrow0}\E[\chi_{\e,R}]+\liminf_{t\uparrow1}\liminf_{\e\downarrow0}\E[\chi_{\e,R}(J_\e^{\per}(tu_\e+tu,t\Lambda,\cdot)-J(u+\Lambda\cdot x;Q))]\\
\ge~& J(u+\Lambda\cdot x;Q)\liminf_{\e\downarrow0}\E[\chi_{\e,R}]+\delta\liminf_{\e\downarrow0}\E[\mathds1_{\Omega_{\e_0,t_0}^\delta}\chi_{\e,R}]\\
&\quad+\liminf_{t\uparrow1}\liminf_{\e\downarrow0}\E[\mathds1_{\Omega\setminus \Omega_{\e_0,t_0}^\delta}\chi_{\e,R}(J_\e^{\per}(tu_\e+tu,t\Lambda,\cdot)-J(u+\Lambda\cdot x;Q))].
\end{align*}
Now a union bound argument yields
\begin{align*}
\E[\chi_{\e,R}]\ge 1-k\p[\Omega\setminus\Omega_R],\qquad\E[\mathds1_{\Omega_{\e_0,t_0}^\delta}\chi_{\e,R}]\ge \p[\Omega_{\e_0,t_0}^{\delta}]-k\p[\Omega\setminus\Omega_R],
\end{align*}
and hence, by Fatou's lemma and by the $\Gamma$-$\liminf$ inequality of Step~1,
we obtain
\begin{align*}
&\liminf_{t\uparrow1}\liminf_{\e\downarrow0}\E[\chi_{\e,R}J_\e^{\per}(tu_\e+tu,t\Lambda,\cdot)]\\
\ge~& J(u+\Lambda\cdot x;Q)(1-k\p[\Omega\setminus\Omega_R])+\delta(\p[\Omega_{\e_0,t_0}^{\delta}]-k\p[\Omega\setminus\Omega_R]).
\end{align*}
However, taking the limit $R\uparrow\infty$, and recalling $\p[\Omega_R]\uparrow1$, this contradicts~\eqref{eq:primgammsuplim}. The result~\eqref{eq:limsupper} is then proven.

\medskip
\step{2.3} Recovery sequence for general target functions.

For all $\Lambda\in\inter\dom\overline V$, let $\Omega_\Lambda\subset\Omega''$ be a subset of maximal probability such that equalities~\eqref{eq:limsupper} hold for all $\omega\in\Omega_\Lambda$ for the choice $u=0$. We then define $\Omega'''\subset\Omega''$ as the (countable) intersection of all $\Omega_\Lambda$'s with $\Lambda\in\Q^{1\times d}\cap\inter\dom\overline V$. Repeating on this basis the gluing argument of Step~2 of the proof of Proposition~\ref{prop:gammasupN}, for all $\Lambda\in\R^{1\times d}$ and for all continuous piecewise affine $Q$-periodic functions $u$ on $\R^d$ with $\Lambda+\nabla u\in\Q^{1\times d}\cap\inter\dom \overline V$, we may prove that~\eqref{eq:limsupper} holds for all $\omega\in\Omega'''$.

Let now $u\in W^{1,p}_{\per}(Q)$ and let $\Lambda\in\R^{1\times d}$.
By Step~1, we may assume $J(u+\Lambda\cdot x;Q)<\infty$, so that by convexity $\Lambda\in\dom\overline V$.
By convexity and lower semicontinuity of $\overline V$, $\lim_{t\uparrow1}J(tu+t\Lambda\cdot x;Q)=J(u+\Lambda\cdot x;Q)$. For all $t\in(0,1)$, since $\Lambda \in\dom\overline V$, we have $0\in\inter\dom\overline V(\cdot+t\Lambda)$ and Proposition~\ref{prop:approxw}(ii)(b) then gives a sequence $(u_{n,t})_n$ of continuous piecewise affine $Q$-periodic functions such that $\Lambda+\nabla u_{n,t}\in\Q^{1\times d}\cap \inter\dom \overline V$, $u_{n,t}\to tu$ in $W^{1,p}(Q)$, and $J(u_{n,t}+t\Lambda\cdot x;Q)\to J(tu+t\Lambda\cdot x;Q)$. We can then apply the result for the (rational) continuous piecewise affine approximations on $\Omega'''$, and the conclusion then follows from a diagonalization argument.

In the case when $J(\alpha u+\alpha \Lambda\cdot x;Q)<\infty$ for some $\alpha>1$, the limits $t\uparrow1$ can be dropped. Similarly as in the proof of Corollary~\ref{cor:dir} (see Sections~\ref{chap:liftdirbd} and~\ref{chap:2dirbd}), it is enough to replace $u,\Lambda$ by $u/t,\Lambda/t$ for $t\in[1/\alpha,1)$ and to adapt Steps~2.2 and~2.3 above accordingly. We omit the details.

\medskip
\step{3} $\Gamma$-$\limsup$ inequality under assumption~(2). We assume $p>d$, and $V(y,\Lambda,\omega)\le C(1+|\Lambda|^p)$ for all $\omega$ and all $y\notin E(\omega)$, for some random stationary set $E^\omega=\bigcup_{n=1}^\infty B_{R_n^\omega}(q_n^\omega)$ satisfying almost surely, for all $n$, for some constant $C>0$,
\begin{align}\label{eq:conditionradii}
2\delta_n^\omega:=\frac1{R_n^\omega} ~{\inf_{m,m\ne n}\dist(B_{R_m^\omega}(q_m^\omega),B_{R_n^\omega}(q_n^\omega))}\ge \frac1C,\qquad R_n^\omega\le C.
\end{align}
We prove the existence of a subset $\Omega'''\subset\Omega''$ of maximal probability such that, for all $u\in W^{1,p}_{\per}(Q)$, all $\Lambda\in\R^{m\times d}$, and all $\omega\in\Omega'''$, there exists a sequence $u_\e^\omega\cvf{}0$ in $W^{1,p}_{\per}(Q;\R^m)$ such that
\begin{align*}
\lim_{t\uparrow1}\liminf_{\e\downarrow0}J_\e^{\per}(tu_\e^\omega+tu,t\Lambda,\omega)=\lim_{t\uparrow1}\limsup_{\e\downarrow0}J_\e^{\per}(tu_\e^\omega+tu,t\Lambda,\omega)= J(u+\Lambda\cdot x;Q).
\end{align*}
Moreover, the limits $t\uparrow1$ can be dropped if we have $J(\alpha u+\alpha\Lambda\cdot x;Q)<\infty$ for some $\alpha>1$. Finally, we have for all $\Lambda\in\R^{m\times d}$ and almost all $\omega$
\begin{align}\label{eq:convdefJlimperincl}
\overline V(\Lambda)=\lim_{\e\downarrow0}\inf_{u\in W^{1,p}_{\per}(Q;\R^m)}\int_{Q}V^{1/\e}(y/\e,\Lambda+\nabla u(y),\omega)dy,
\end{align}
where the convergence also holds in expectation. We split the proof of this step into three parts.

\medskip
\step{3.1} Recovery sequence for affine functions. We prove that for all $\Lambda\in\R^{m\times d}$ for almost all $\omega$ there is a sequence $u_\e^\omega\cvf{}0$ in $W^{1,p}_{\per}(Q;\R^m)$ such that $J_\e^{\per}(\Lambda,u_\e^\omega,\omega)\to J(\Lambda\cdot x;Q)$.

Let $\Lambda\in\R^{m\times d}$ be fixed. Lemma~\ref{lem:sublincor} and the Birkhoff-Khinchin ergodic theorem give a random field $\varphi_{\Lambda}\in\Mes(\Omega;W^{1,p}_\loc(\R^d;\R^m))$ such that, for almost all $\omega$, $\e\varphi_{\Lambda}(\cdot/\e,\omega)\cvf{}0$ weakly in $W^{1,p}_{\loc}(\R^d;\R^m)$ and also $J_\e(\Lambda\cdot x+\e\varphi_\Lambda(\cdot/\e,\omega),\omega;O)\to J(\Lambda\cdot x;O)$, for all bounded domains $O\subset\R^d$. Let $\omega\in\Omega''$ be fixed such that both properties hold, and set $u_\e^\omega:=\Lambda\cdot x+\e\varphi_\Lambda(\cdot/\e,\omega)$.

For all $\e>0$, consider the $Q/\e$-periodic random set $E_\e^\omega=\bigcup_{n=1}^\infty B_{R_{\e,n}^\omega}(q_{\e,n}^\omega)$ associated with the periodization in law $V^{1/\e}$ at scale $1/\e$. By stationarity, these periodized inclusions also satisfy~\eqref{eq:conditionradii}, that is $2\delta_{\e,n}^\omega\ge1/C$ and $R_{\e,n}^\omega\le C$ almost surely. For all $\e,n$, define
$B_{\e,n}^\omega:=B_{R_{\e,n}^\omega}(q_{\e,n}^\omega)$ and $\tilde B_{\e,n}^\omega:=B_{R_{\e,n}^\omega(1+1\wedge\delta_{\e,n}^\omega)}(q_{\e,n}^\omega)$. By definition, the $\tilde B_{\e,n}^\omega$'s are all disjoint. For all $\e,n$, choose a smooth cut-off function $\chi_{\e,n}^\omega$ that equals $1$ on $\e B_{\e,n}^\omega$, vanishes outside $\e\tilde B_{\e,n}^\omega$ and satisfies $|\nabla \chi_{\e,n}^\omega|\le 2/(\e R_{\e,n}^\omega(1\wedge \delta_{\e,n}^\omega))\le 4C/(\e R_{\e,n}^\omega)$. We choose these cut-off functions in such a way that $\sum_n\chi_{\e,n}^\omega$ is $Q$-periodic.

Let $\theta\in(0,1)$. Denote by $N_{\e,\theta}^\omega$ (resp. $M_{\e}^\omega$) the set of all $n\ge1$ such that $(Q\setminus\theta Q)\cap\e\tilde B_{\e,n}^\omega\ne\varnothing$ (resp. $\partial Q\cap\e\tilde B_{\e,n}^\omega\ne\varnothing$). Choose a smooth cut-off function $\chi_{\theta}$ that equals $1$ on $\theta Q$, vanishes outside $Q$ and satisfies $|\nabla \chi_\theta|\le2/(1-\theta)$. Set
\[\chi_{\e,\theta}^\omega(y):=\chi_\theta(y)-\chi_\theta(y)\sum_{n\in M_\e^\omega}\chi_{\e,n}^\omega(y)+\sum_{n\in N_{\e,\theta}^\omega\setminus M_\e^\omega}\chi_{\e,n}^\omega(y)\fint_{\e\tilde B_{\e,n}^\omega}(\chi_\theta(z)-\chi_\theta(y))dz.\]
This defines a smooth cut-off function $\chi_{\e,\theta}^\omega$ that equals $1$ on $(\theta-2C\e) Q$, vanishes outside $Q$, is constant on each inclusion $B_{\e,n}$, and satisfies $|\nabla \chi_{\e,\theta}^\omega|\le C/(1-\theta)$ for some constant $C>0$. We then define
\begin{align*}
u_{\e,\theta}^\omega(y)&=(1-\chi_{\e,\theta}^\omega(y))\Lambda\cdot y+ (1-\chi_{\e,\theta}^\omega(y))\sum_{n\in N_{\e,\theta}^\omega}\chi_{\e,n}^\omega(y)\fint_{\e\tilde B_{\e,n}^\omega}\Lambda\cdot (z-y)dz\\
&\qquad+\chi_{\e,\theta}^\omega(y)\bigg(1-\sum_{n\in N_{\e,\theta}^\omega}\chi_{\e,n}^\omega(y)\bigg)u_\e^\omega(y)+\chi_{\e,\theta}^\omega(y)\sum_{n\in N_{\e,\theta}^\omega}\chi_{\e,n}^\omega(y)\fint_{\e\tilde B_{\e,n}^\omega}u_\e^\omega.
\end{align*}
By construction $u_{\e,\theta}^\omega\in \Lambda\cdot x+W^{1,p}_{\per}(Q;\R^m)$. Since $u_{\e,\theta}^\omega=u_\e^\omega$ at least on $(\theta-2\e C)Q$, and since $\nabla u_{\e,\theta}^\omega=0$ on all $\e B_{\e,n}^\omega$ with $n\in N_{\e,\theta}^\omega$, we have, for all $\e>0$ small enough,
\begin{align}
\int_QV^{1/\e}(\cdot/\e,\nabla u_{\e,\theta}^\omega,\omega)&\le \int_{(\theta-2\e C)Q}V(\cdot/\e,\nabla u_\e^\omega,\omega)\nonumber\\
&\qquad+M(0)|Q\setminus (\theta-2\e C)Q|+C\int_{Q\setminus (\theta-2\e C)Q}(1+|\nabla u_{\e,\theta}^\omega|^p)\nonumber\\
&\le J_\e(u_\e^\omega,\omega;Q)
+C(1-\theta+2\e C)+C\int_{Q\setminus (\theta-2\e C)Q}|\nabla u_{\e,\theta}^\omega|^p.\label{eq:stiffinclbound}
\end{align}
We need to estimate the last term. A straightforward calculation yields
\begin{align*}
\int_{Q\setminus (\theta-2\e C)Q}|\nabla u_{\e,\theta}^\omega|^p&\le C(1-\theta+2\e C)|\Lambda|^p+\frac C{(1-\theta)^p}\int_{Q}|u_\e^\omega(y)-\Lambda\cdot y|^pdy+C\int_{Q\setminus (\theta-2\e C)Q}|\nabla u_\e^\omega|^p\\
&\qquad+C\sum_{n\in M_\e^\omega}((1-\theta)^{-p}+(\e R_{\e,n}^\omega)^{-p})\int_{\e\tilde B_{\e,n}^\omega}\bigg|\fint_{\e\tilde B_{\e,n}^\omega}\Lambda\cdot (z-y)dz\bigg|^pdy\\
&\qquad+C\sum_{n\in N_{\e,\theta}^\omega}((1-\theta)^{-p}+(\e R_{\e,n}^\omega)^{-p})\int_{\e\tilde B_{\e,n}^\omega}\bigg|u_\e^\omega-\fint_{\e\tilde B_{\e,n}^\omega}u_\e^\omega\bigg|^p,
\end{align*}
and hence, by the Poincaré inequality and the estimate $\sum_{n\in M_\e^\omega}|\e\tilde B_{\e,n}^\omega|\le|Q\setminus (1-2\e C)Q|\le 2d\e C$,
\begin{align}\label{eq:stiffinclbounderr}
\int_{Q\setminus (\theta-2\e C)Q}|\nabla u_{\e,\theta}^\omega|^p&\le C(1-\theta+2\e C)|\Lambda|^p+\frac C{(1-\theta)^p}\int_{Q}|\e\varphi_\Lambda(\cdot/\e,\omega)|^p\\
&\quad+C\bigg(1+\frac{\e^p}{(1-\theta)^p}\bigg)\int_{Q\setminus(\theta-2\e C)Q}(|\Lambda|^p+|\nabla \varphi_\Lambda(\cdot/\e,\omega)|^p).\nonumber
\end{align}
For all $\e_0>0$, the Birkhoff-Khinchin ergodic theorem yields, for almost all $\omega$,
\begin{align*}
\limsup_{\e\downarrow0}\int_{Q\setminus(\theta-2\e C)Q}(|\Lambda|^p+|\nabla \varphi_\Lambda(\cdot/\e,\omega)|^p)&\le\lim_{\e\downarrow0}\int_{Q\setminus(\theta-2\e_0 C)Q}(|\Lambda|^p+|\nabla \varphi_\Lambda(\cdot/\e,\omega)|^p)\\
&\le|Q\setminus(\theta-2\e_0 C)Q|\,(|\Lambda|^p+\E[|\nabla \varphi_\Lambda|^p])\\
&\le C(1-\theta+2\e_0 C)(|\Lambda|^p+\E[|\nabla \varphi_\Lambda|^p]).
\end{align*}
Combined with~\eqref{eq:stiffinclbounderr} and the sublinearity of the corrector in the form of $\e\varphi_\Lambda(\cdot/\e,\omega)\to0$ in $\Ld^p(Q;\R^m)$ (see Lemma~\ref{lem:sublincor} and the Rellich-Kondrachov theorem), this yields, for almost all~$\omega$,
\begin{align*}
\limsup_{\e\downarrow0}\int_{Q\setminus (\theta-2\e C)Q}|\nabla u_{\e,\theta}^\omega|^p&\le C(1-\theta)|\Lambda|^p+C(1-\theta+2\e_0 C)\E[|\Lambda+\nabla \varphi_\Lambda|^p].
\end{align*}
We then insert this estimate in~\eqref{eq:stiffinclbound} and to pass to the limits $\theta\uparrow1$ and $\e_0\downarrow0$ to obtain for almost all $\omega$
\[\limsup_{\theta\uparrow1}\limsup_{\e\downarrow0}\int_QV^{1/\e}(\cdot/\e,\nabla u_{\e,\theta}^\omega,\omega)\le \lim_{\e\downarrow0}J_\e(u_\e^\omega,\omega;Q)= \overline V(\Lambda)=J(\Lambda\cdot x;Q).\]
The conclusion then follows from Step~1 and a diagonalization argument.

\medskip
\step{3.2} Recovery sequence for general target functions.

As we assume here $p>d$, we may use the gluing argument of Step~2 of the proof of Proposition~\ref{prop:gammasupN} to pass from affine to piecewise affine functions. The conclusion follows by approximation similarly as in Step~2.3 above. We omit the details.

\medskip
\step{3.3} Proof that the almost sure convergence~\eqref{eq:convdefJlimperincl} also holds in expectation.

Using the same notation and cut-off functions as in Step~3.1, we define $u_\e^\omega\in W^{1,p}_{\per}(Q;\R^m)$ as
\[u_\e^\omega(y):=\sum_{n}\chi_{\e,n}^\omega(y)\fint_{\e\tilde B_{\e,n}^\omega}\Lambda\cdot(z-y)dz.\]
Testing the infimum in~\eqref{eq:convdefJlimperincl} with $u=u_\e^\omega$, we obtain by similar calculations as in~\eqref{eq:stiffinclbound} and~\eqref{eq:stiffinclbounderr} that for all $\e>0$ and all $\omega$
\begin{align*}
&\inf_{u\in W^{1,p}_{\per}(Q;\R^m)}\int_{Q}V^{1/\e}(\cdot/\e,\Lambda+\nabla u,\omega)\\
\le~& C(1+|\Lambda|^p)+C\sum_n(\e R_{\e,n}^\omega)^{-p}\int_{Q\cap \e \tilde B_{\e,n}^\omega}\bigg|\fint_{\e\tilde B_{\e,n}^\omega}\Lambda\cdot (z-y)dz\bigg|^pdy\le C(1+|\Lambda|^p).
\end{align*}
The conclusion then follows by dominated convergence.
\end{proof}

\bigskip
\section*{Acknowledgements}
The authors acknowledge financial support from the European Research Council under
the European Community's Seventh Framework Programme (FP7/2014-2019 Grant Agreement
QUANTHOM 335410).

\renewcommand{\thesection}{A}
\section{Appendix}\label{annexe}

\subsection{Normal random integrands}\label{app:intnormal}
In this appendix, we briefly recall the precise definition of normal random integrands (as defined e.g.~in~\cite[Section~VIII.1.3]{Ekeland76}) and we prove their main properties, mentioned in the beginning of Section~\ref{sec:results} and used throughout this paper. Let $(\Omega,\F,P)$ be a complete probability space. We denote by $\B(\R^k)$ the (not completed) Borel $\sigma$-algebra on $\R^k$.

\begin{defin}\label{def:integrand}
A {\it normal random integrand} is a map $W:\R^d\times\R^{m\times d}\times\Omega\to[0,\infty]$ such that
\begin{enumerate}[(a)]
\item $W$ is jointly measurable (i.e. with respect to the completion of $\B(\R^d)\times\B(\R^{m\times d})\times\F$);
\item for almost all $\omega$, there exists a map $V_\omega:\R^d\times\R^{m\times d}\to[0,\infty]$ that is $\B(\R^d)\times\B(\R^{m\times d})$-measurable and such that $W(y,\cdot,\omega)=V_\omega(y,\cdot)$ for almost all $y$;
\item for almost all $y$, there exists a map $V_y:\R^{m\times d}\times\Omega\to[0,\infty]$ that is $\B(\R^{m\times d})\times\F$-measurable and such that $W(y,\cdot,\omega)=V_y(\cdot,\omega)$ for almost all $\omega$;
\item for almost all $y,\omega$, the map $W(y,\cdot,\omega)$ is lower semicontinuous on $\R^{m\times d}$.
\end{enumerate}
It is said to be {\it$\tau$-stationary} if it satisfies~\eqref{eq:statnormalrandomint} for all $\Lambda,y,z,\omega$.\qed
\end{defin}

As shown e.g.~in~\cite[Section~VIII.1.3]{Ekeland76},
a simple example of normal random integrands is given by the so-called {\it Carathéodory random integrands}, that are maps $W:\R^d\times\R^{m\times d}\times\Omega\to[0,\infty]$ such that $W(y,\cdot,\omega)$ is continuous on $\R^{m\times d}$ for almost all $y,\omega$, and such that $W(\cdot,\Lambda,\cdot)$ is jointly measurable on $\R^d\times\Omega$ for all $\Lambda$.

As already advertised in the beginning of Section~\ref{sec:results}, the reason for these technical assumptions is that they are particularly weak but still guarantee the following properties:

\begin{lem}\label{lem:intnormal}
Let $W:\R^d\times\R^{m\times d}\times\Omega\to[0,\infty]$ be a normal random integrand. Then,
\begin{enumerate}[(i)]
\item for almost all $\omega$, the map $y\mapsto W(y,u(y),\omega)$ is measurable for all $u\in\Mes(\R^d,\R^{m\times d})$;
\item for almost all $y$, the map $\omega\mapsto W(y,u(\omega),\omega)$ is measurable for all $u\in\Mes(\Omega,\R^{m\times d})$.\qed
\end{enumerate}
\end{lem}

\begin{proof}
For almost all $\omega$, part~(b) of Definition~\ref{def:integrand} gives a $\B(\R^d)\times\B(\R^{m\times d})$-measurable map $V_\omega$ on $\R^d\times\R^{m\times d}$ such that $W(y,\cdot,\omega)=V_\omega(y,\cdot)$ for almost all $y$. Hence, for $u\in\Mes(\R^d;\R^{m\times d})$, the map $y\mapsto W(y,u(y),\omega)$ is equal almost everywhere to the map $y\mapsto V_\omega(y,u(y))$, which is necessarily measurable since $\Id\times u:\R^d\times\R^d\to\R^d\times\R^{m\times d}$ is measurable and since $V_\omega$ is Borel-measurable. This proves~(i), and~(ii) is similar.
\end{proof}

If $W$ is $\tau$-stationary, we may write $W(y,\Lambda,\omega)=W(0,\Lambda,\tau_{-y}\omega)$, which thus receives a pointwise meaning in the first variable, and Lemma~\ref{lem:intnormal}(ii) may obviously be strengthened as follows.

\begin{lem}
Let $W:\R^d\times\R^{m\times d}\times\Omega\to[0,\infty]$ be a $\tau$-stationary normal random integrand. For \emph{all} $y$ and all $u\in\Mes(\Omega;\R^{m\times d})$, the map $\omega\mapsto W(y,u(\omega),\omega)$ is measurable.\qed
\end{lem}

\subsection{Stationary differential calculus in probability}\label{app:statcal}
In this appendix, we precisely define the measurable action $\tau$ that is used throughout this paper to induce the stationarity, and we discuss the properties of the stationary derivatives defined in Section~\ref{chap:preliminD}, and prove in particular the useful identity~\eqref{eq:linkDnabla}.

\subsubsection{Stationary random fields}\label{app:cadreproba}
As usual, the standard notion of stationarity of random fields (defined as the translation invariance of all the finite-dimensional distributions) is strictly equivalent to a formulation of stationarity as the invariance under some (measure-preserving) action of the group of translations $(\R^d,+)$ on the probability space (see e.g.~\cite[Section~16.1]{KorSinai}). This point of view is of great interest, since it puts us into the realm of ergodic theory.

Because we focus on jointly measurable random fields, which is standard in stochastic homogenization theory (see also Remark~\ref{rem:meascont} below), a further measurability requirement is added in our definition of an action, as e.g.~in~\cite[Section~7.1]{JKO94},

\begin{defin}\label{def:measaction}
A {\it measurable action} of the group $(\R^d,+)$ on $(\Omega,\F,\p)$ is a collection $\tau:=(\tau_x)_{x\in\R^d}$ of measurable transformations of $\Omega$ such that
\begin{enumerate}[(i)]
\item $\tau_x\circ\tau_y=\tau_{x+y}$ for all $x,y\in\R^d$;
\item $\p[\tau_xA]=\p[A]$ for all $x\in\R^d$ and all $A\in\F$;
\item the map $\R^d\times\Omega\to\Omega:(x,\omega)\mapsto\tau_x\omega$ is measurable.\qed
\end{enumerate}
\end{defin}

For any random variable $f\in\Mes(\Omega;\R)$, we may define its $\tau$-stationary extension $f:\R^d\times\Omega\to\R$ by $f(x,\omega):= f(\tau_{-x}\omega)$, which is a $\tau$-stationary random field on $\R^d$ (in the sense that $f(x+y,\omega)=f(x,\tau_{-y}\omega)$ for all $x,y,\omega$) and which is by definition jointly measurable on $\R^d\times\Omega$. For $f\in\Ld^p(\Omega)$, $1\le p<\infty$, the $\tau$-stationary extension $f$ belongs to $\Ld^p(\Omega;\Ld^p_{\loc}(\R^d))$. In this way, we get a bijection between the random variables (resp. in $\Ld^p(\Omega)$) and the $\tau$-stationary random fields (resp. in $\Ld^p(\Omega;\Ld^p_{\loc}(\R^d))$).

We may also naturally consider the associated action $T:=(T_x)_{x\in\R^d}$ of $(\R^d,+)$ on $\Mes(\Omega;\R)$, defined by $(T_x f)(\omega)=f(\tau_{-x}\omega)$ for all $\omega\in\Omega$ and $f\in\Mes(\Omega;\R)$. Let $1\le p<\infty$. The following gives elementary properties of this action (see e.g.~\cite[Section~7.1]{JKO94}):

\begin{lem}\label{lem:unit}
The action $T$ defined above is unitary and strongly continuous on $\Ld^p(\Omega)$.\qed
\end{lem}

In the context of stochastic homogenization theory, the measurability hypotheses made above (as in e.g.~\cite[Section~7.1]{JKO94}) are sometimes replaced by stochastic continuity hypotheses (see e.g.~\cite[Section~2]{PapaVara}). As the following remark shows, both are actually equivalent.

\begin{rem}[Measurability or continuity]\label{rem:meascont}
It should be noted that the additional measurability assumption~(iii) in Definition~\ref{def:measaction} above is not inoffensive at all. Indeed, a stochastic version of the Lusin theorem can easily be proven: a random field $h$ on $\R^d$ is jointly measurable if and only if, for almost all $x\in\R^d$, for all $\delta>0$,
\[\lim_{y\to0}\p[|h(x+y,\omega)-h(x,\omega)|>\delta]=0.\]
Hence, for a {\it stationary} random field $h$ on $\R^d$, joint measurability is actually equivalent to stochastic continuity (and even to continuity in the $p$-th mean, in the case when $h(0,\cdot)\in\Ld^p(\Omega)$). In the same vein, the measurability property~(iii) in Definition~\ref{def:measaction} is equivalent to the strong continuity of the action $T$ of $(\R^d,+)$ on $\Ld^p(\Omega)$, and also to the property that all $\tau$-stationary extensions are stochastically continuous.
\qed
\end{rem}

\subsubsection{Stationary Sobolev spaces}\label{app:statsobolev}
Let $1\le p<\infty$, and let the stationary gradient $D$ and the space $W^{1,p}(\Omega)$ be defined as in Section~\ref{chap:preliminD}.
Now we present another useful vision for derivatives of stationary random fields.

Given a random variable $f\in \Ld^p(\Omega)$, the $\tau$-stationary extension is an element $f\in\Ld^p_\loc(\R^d;\Ld^p(\Omega))=\Ld^p(\Omega;\Ld^p_\loc(\R^d))$ and can thus be seen as an $\Ld^p(\Omega)$-valued distribution on $\R^d$. We may then define its distributional gradient $\nabla f$ in the usual way. Note that by definition, for almost all $\omega$, $\nabla f(\cdot,\omega)$ is nothing but the usual distributional gradient of $f(\cdot,\omega)\in\Ld^p_{\loc}(\R^d)$. As usual, the Sobolev space $W^{1,p}_\loc(\R^d;\Ld^p(\Omega))$ is defined as the space of functions $f\in\Ld^p_\loc(\R^d;\Ld^p(\Omega))$ such that $\nabla f\in\Ld^p_\loc(\R^d;\Ld^p(\Omega;\R^d))$, and in that case $\nabla f$ is called the {\it weak gradient}. The following result shows the link with stationary gradients and with the space $W^{1,p}(\Omega)$, in particular proving identity~\eqref{eq:linkDnabla}.

\begin{lem}\label{lem:statsobolev}
Modulo the correspondence between random variables and $\tau$-stationary random fields, we have
\[W^{1,p}(\Omega)=\{f\in W^{1,p}_{\loc}(\R^d;\Ld^p(\Omega)): f(x+y,\omega)=f(x,\tau_{-y}\omega),\forall x,y,\omega\},\]
and moreover $\nabla f=Df$ for all $f\in W^{1,p}(\Omega)$.\qed
\end{lem}

\begin{proof}
Denote for simplicity $E_p:=\{f\in W^{1,p}_{\loc}(\R^d;\Ld^p(\Omega)): f(x+y,\omega)=f(x,\tau_{-y}\omega),\forall x,y,\omega\}$.
For all $i$, the stationary derivative $D_if$ is defined as the strong derivative of the map $\R\to\Ld^p(\Omega):h\mapsto T_{-he_i}f$, so that
\[W^{1,p}(\Omega)=\{f\in C^1(\R^d;\Ld^p(\Omega)):f(x+y,\omega)=f(x,\tau_{-y}\omega),\forall x,y,\omega\}.\]
Hence, for all $f\in W^{1,p}(\Omega)$, we have $f\in E_p$, and $\nabla f=Df$, since weak derivatives are generalizations of strong derivatives.

We now turn to the converse statement. Let $f\in W^{1,p}(\Omega)$.
As in~\cite[Section~7.2]{JKO94}, choose a nonnegative even function $\rho\in C_c^\infty(\R^d)$ with $\int\rho=1$ and $\supp\rho\subset B_1$, write $\rho_\delta(x)=\delta^{-d}\rho(x/\delta)$ for all $\delta>0$, and define a regularization $R_\delta [f]\in\Ld^p(\Omega)$ by
\begin{align}\label{eq:regconv}
R_\delta [f](\omega)=\int_{\R^d}\rho_\delta(y)f(y,\omega)dy,
\end{align}
or equivalently, as $\rho_\delta$ is even,
\[R_\delta [f](x,\omega)=\int_{\R^d}\rho_\delta(y)f(x+y,\omega)dy=\int_{\R^d}\rho_\delta(y-x)f(y,\omega)dy=(\rho_\delta\ast f(\cdot,\omega))(x).\]
Clearly, $R_\delta [f]\to f$ in $\Ld^p(\Omega)$, and hence by stationarity $R_\delta [f](x,\cdot)\to f(x,\cdot)$ in $\Ld^p(\Omega)$ uniformly in $x$. As by definition $R_\delta [f]\in C^\infty(\R^d;\Ld^p(\Omega))$, we have $R_\delta [f]\in W^{1,p}(\Omega)$ and the stationary gradient is simply $DR_\delta [f]=R_\delta [\nabla f]$. This proves $DR_\delta [f]\to\nabla f$ in $\Ld^p(\Omega;\R^d)$, and thus $DR_\delta [f](x,\cdot)\to\nabla f(x,\cdot)$ in $\Ld^p(\Omega;\R^d)$ uniformly in $x$. Hence $R_\delta[f]\to f$ in $C^1(\R^d;\Ld^p(\Omega))$, so $f\in C^1(\R^d;\Ld^p(\Omega))$, from which we conclude $f\in W^{1,p}(\Omega)$.
\end{proof}

\subsection{Ergodic Weyl decomposition}\label{app:weylerg}
In this appendix, we discuss the various properties of the ergodic Weyl spaces recalled in Section~\ref{chap:preliminweyl}. In particular, we prove the ergodic Weyl decomposition~\eqref{eq:decweyl} as well as the density result~\eqref{eq:densweyl}.

Let $\tau$ be a measurable action of $(\R^d,+)$ on the probability space $(\Omega,\F,\p)$, let $1<p<\infty$, and let $\Ld^p_\pot(\Omega)$, $\Ld^p_\sol(\Omega)$, $F^p_\pot(\Omega)$ and $F^p_\sol(\Omega)$ be defined by~\eqref{eq:defweylerg-0} (or equivalently~\eqref{eq:defweylerg-1}) and~\eqref{eq:defweylerg-F}.
First we prove the ergodic Weyl decomposition~\eqref{eq:decweyl}:

\begin{prop}[Ergodic Weyl decomposition]\label{prop:ergweyldec}
Let $\tau$ be ergodic and let $1<p<\infty$. Then the following Banach direct sum decompositions hold:
\[\Ld^p(\Omega;\R^d)=\Ld^p_\pot(\Omega)\oplus F^p_\sol(\Omega)=F^p_\pot(\Omega)\oplus \Ld^p_\sol(\Omega)=F^p_\pot(\Omega)\oplus F^p_\sol(\Omega)\oplus\R^d.\]
\qed
\end{prop}

\begin{proof}
Let $1<p<\infty$ and let $q:=p'=p/(p-1)$. The following weaker form of the result is proven e.g. in~\cite[Lemma~15.1]{JKO94}:
\begin{align}\label{eq:embryonweyl}
(F^p_\pot(\Omega))^\bot&=\Ld^{q}_\sol(\Omega)=F^{q}_\sol(\Omega)\oplus\R^d,\\
(F^p_\sol(\Omega))^\bot&=\Ld^{q}_\pot(\Omega)=F^{q}_\pot(\Omega)\oplus\R^d,\nonumber
\end{align}
in the sense of duality between $\Ld^p(\Omega;\R^d)$ and $\Ld^{q}(\Omega;\R^d)$. In the Hilbert case $p=2$, this is already enough to conclude by the orthogonal decomposition theorem. For the general case, we will need to use more subtle (nonlinear) decomposition results in Banach spaces.

As the Banach space $\Ld^p(\Omega;\R^d)$ is uniformly smooth and uniformly convex, the following (nonlinear) direct sum decomposition (see e.g.~\cite[Theorem~2.13]{Alber-05}) holds for any closed subspace $M\subset \Ld^p(\Omega;\R^d)$:\footnote{The direct sum means here that any element of $\Ld^p(\Omega;\R^d)$ can be written in a unique way as the sum of an element of $M$ and an element of $J(M^\bot)$, and that moreover $M\cap J(M^\bot)=\{0\}$. As shown in~\cite{Alber-05}, this is actually even a direct sum in Birkhoff-James' sense, that is $\|u\|_{\Ld^p(\Omega)}\le\|u+tv\|_{\Ld^p(\Omega)}$ for all $u\in M$, $v\in J(M^\bot)$ and $t\in\R$.}
\[\Ld^p(\Omega;\R^d)=M\oplus J(M^\bot),\]
where $M^\bot\subset\Ld^q(\Omega;\R^d)$ denotes the orthogonal in the sense of duality, and where the (nonlinear) map $J:\Ld^q(\Omega;\R^d)\to\Ld^p(\Omega;\R^d)$ is defined by $J(u):=\|u\|_{\Ld^q}^{2-q}|u|^{q-2}u$. Let us apply this result to the choice $M=F^p_{\pot}(\Omega)$. By~\eqref{eq:embryonweyl} we have $M^\bot=\Ld^q_\sol(\Omega)$, and we may compute $DJ(u)=(q-1)\|u\|_{\Ld^q}^{2-q}|u|^{q-2}Du$, so that $JM^\bot\subset\Ld^p_\sol(\Omega)$. On the other hand, given $u\in\Ld^p_\sol(\Omega)$, we may write $u=J(v)$ with $v:=\|u\|_{\Ld^p}^{2-p}|u|^{p-2}u\in\Ld^q_\sol(\Omega)=M^\bot$. This proves equality $J(M^\bot)=\Ld^p_\sol(\Omega)$, and the result follows.
\end{proof}

The density result~\eqref{eq:densweyl} is now easily obtained, using the same (nonlinear) direct sum decomposition in Banach spaces:

\begin{prop}\label{prop:graddenspot}
Let $\tau$ be ergodic, and let $1<p<\infty$. Then,
\begin{align*}
F^p_\pot(\Omega)&=\adh_{\Ld^p(\Omega;\R^d)}\{D\chi\,:\,\chi\in W^{1,p}(\Omega)\},\\
F^p_\sol(\Omega)&=\adh_{\Ld^p(\Omega;\R^d)}\{D\times\chi\,:\,\chi\in W^{1,p}(\Omega)\}.
\end{align*}\qed
\end{prop}

\begin{proof}
Using the same notation as in the proof of Proposition~\ref{prop:ergweyldec}, the (nonlinear) direct sum decomposition of~\cite[Theorem~2.13]{Alber-05} gives
\begin{align*}
\Ld^p(\Omega;\R^d)=\adh DW^{1,p}(\Omega)\oplus J((\adh DW^{1,p}(\Omega))^\bot).
\end{align*}
By definition $(\adh DW^{1,p}(\Omega))^\bot=\Ld^q_{\sol}(\Omega)$, while as in the proof of Proposition~\ref{prop:ergweyldec} we find $J(\Ld^q_{\sol}(\Omega))=\Ld^p_{\sol}(\Omega)$. This yields
\begin{align*}
\Ld^p(\Omega;\R^d)=\adh DH^\infty(\Omega)\oplus \Ld^p_\sol(\Omega).
\end{align*}
The obvious inclusion $\adh D W^{1,p}(\Omega)\subset F^p_\pot(\Omega)$ and the direct sum decomposition of Proposition~\ref{prop:ergweyldec} then imply the equality $\adh D W^{1,p}(\Omega)= F^p_\pot(\Omega)$. The result for $F^p_{\sol}$ is obtained similarly.
\end{proof}

\subsection{Measurability results}\label{chap:selmes}
This appendix is concerned with various measurability properties.

\subsubsection{Measurable potentials for random fields}
The following result complements the equivalent definition~\eqref{eq:defweylerg-1} of $\Ld^p_{\pot}(\Omega)$, and shows that potentials associated with potential random fields may be chosen in a measurable way with respect to the alea.

\begin{prop}\label{prop:measpot}
Let $\tau$ be an ergodic measurable action on a complete probability space $(\Omega,\F,\p)$. Let $1< p<\infty$. For all $f\in \Ld^p_\pot(\Omega)$ there exists a random field $\phi\in\Mes(\Omega;W^{1,p}_\loc(\R^d))$ such that $f(\cdot,\omega)=\nabla \phi(\cdot,\omega)$ for almost all $\omega$.\qed
\end{prop}

\begin{proof}
Let $f\in \Ld^p_\pot(\Omega)$ be fixed.
For all $n,k\ge1$, define the space
\[X_{n,k}:=\left\{\phi\in W^{1,p}(B_n):\|\nabla \phi\|_{\Ld^p(B_{n})}\le k,\,\int_{B_n}\phi=0\right\},\]
endowed with the weak topology. By Poincaré's inequality and by the Banach-Alaoglu theorem, this space is metrizable and compact, hence Polish. Consider the multifunction $\Gamma_{n,k}:\Omega\rightrightarrows X_{n,k}$ defined by
\[\Gamma_{n,k}(\omega):=\{\phi\in X_{n,k}\,:\,\nabla\phi|_{B_{n}}=f(\cdot,\omega)|_{B_{n}}\}.\]
Clearly $\Gamma_{n,k}(\omega)$ is closed for all $\omega$. We first prove further properties of this multifunction, and the conclusion will then follow by applying the Rokhlin--Kuratowski--Ryll Nardzewski theorem.

\medskip
\step{1} For all $n\ge1$, we claim the existence of an increasing sequence of events $\Omega_{n,k}\subset\Omega$ such that $\p[\Omega_{n,k}]\uparrow1$ as $k\uparrow\infty$ for fixed $n$, and such that $\Gamma_{n,k}(\omega)\ne\varnothing$ for all $\omega\in\Omega_{n,k}$.

By Definition~\ref{eq:defweylerg-1}, there is a subset $\Omega'\subset \Omega$ of maximal probability, such that for all $\omega\in\Omega'$ the function $f(\cdot,\omega)$ is a potential field in $\Ld^p_{\loc}(\R^d;\R^d)$, and hence there exists $\phi^\omega\in W^{1,p}_{\loc}(\R^d)$ such that $f(\cdot,\omega)=\nabla \phi^\omega$. For all $n\ge1$,
define $\phi_n^\omega:= \phi^\omega-\fint_{B_{n}}\phi^\omega(z)dz\in W^{1,p}(B_n)$.
By definition, $\int_{B_n}\phi_n^\omega=0$ and $\nabla \phi_n^\omega=f(\cdot,\omega)$ on $B_n$.
Moreover, $\|\nabla \phi_n^\omega\|_{\Ld^p(B_{n})}\le M$ holds for all $\omega\in\Omega_{n,k}$, where we define the event
\[\Omega_{n,k}:=\{\omega\in\Omega':\|f(\cdot,\omega)\|_{\Ld^p(B_{n})}\le k\}.\]
Integrability and stationarity of $f$ easily imply that $\p[\Omega_{n,k}]\uparrow1$ as $k\uparrow\infty$.

\medskip
\step {2} Proof that $\Gamma_{n,k}$ is measurable, in the sense that $\Gamma_{n,k}^{-1}(O)\in \F$ for all open subset $O\subset X_{n,k}$, where we have set
\[\Gamma_{n,k}^{-1}(O):=\{\omega\in\Omega\,:\,\Gamma_{n,k}(\omega)\cap O\ne\varnothing\}.\]

As $X_{n,k}$ is metrizable, it suffices to check $\Gamma_{n,k}^{-1}(F)\in \F$ for all closed subset $F\subset X_{n,k}$ (see e.g. \cite[Lemma~18.2]{AliprantisBorder}). Given a closed subset $F\subset X_{n,k}$, we may write, using Poincaré's inequality and the weak lower semicontinuity of the norm,
\begin{align*}
\Gamma_{n,k}^{-1}(F)&=\{\omega\in\Omega:\exists \phi\in F,\nabla \phi=f(\cdot,\omega)|_{B_n}\}\\
&=\bigcap_{j=1}^\infty\{\omega\in\Omega:\exists \phi\in F,\|\nabla \phi-f(\cdot,\omega)\|_{\Ld^p(B_n)}\le 1/j\}.
\end{align*}
Separability of the Polish space $X_{n,k}$ implies that $F$ is itself separable, and there exists a countable dense subset $F_0\subset F$. Hence
\[\Gamma_{n,k}^{-1}(F)=\bigcap_{j=1}^\infty\bigcup_{\phi\in F_0}\{\omega\in\Omega:\|\nabla \phi-f(\cdot,\omega)\|_{\Ld^p(B_n)}\le 1/j\},\]
and measurability of $\Gamma_{n,k}^{-1}(F)$ then follows from measurability of $f$.

\medskip
\step{3} Conclusion.

By steps~1 and~2, for all $n,k\ge1$, the restricted multifunction $\Gamma_{n,k}|_{\Omega_{n,k}}:\Omega_{n,k}\rightrightarrows X_{n,k}$ is measurable and has nonempty closed values. As $X_{n,k}$ is a Polish space, we may apply the Rokhlin--Kuratowski--Ryll Nardzewski theorem (see e.g.~\cite[Theorem~18.13]{AliprantisBorder}), which gives a measurable function $\phi_{n,k}:\Omega_{n,k}\to X_{n,k}$ such that $\phi_{n,k}(\omega)\in\Gamma_{n,k}(\omega)$, that is $\nabla\phi_{n,k}(\cdot,\omega)=f(\cdot,\omega)|_{B_n}$ for all $\omega\in\Omega_{n,k}$. For all $n>0$, define a measurable function $\phi_n:\Omega\to W^{1,p}(B_n)$ by
\[\phi_n(\omega)=\mathds1_{\Omega_1}(\omega)\phi_{1,n}(\omega)+\sum_{k=2}^\infty \mathds1_{\Omega_{k}\setminus\Omega_{k-1}}(\omega)\phi_{n,k}(\omega).\]
By definition, we have $\nabla\phi_n(\cdot,\omega)=f(\cdot,\omega)|_{B_n}$ for all $\omega\in\Omega_n$, where $\Omega_{n}:=\bigcup_{k=1}^\infty \Omega_{n,k}$ is a subset of maximal probability. Denote $\Omega'':=\bigcap_{n=1}^\infty \Omega_n$.

Let $n\ge1$. By definition, $\nabla\phi_n-\nabla\phi_1$ vanishes on $B_1$, hence the difference $\delta_n(\omega):=\phi_n(\cdot,\omega)-\phi_1(\cdot,\omega)$ is constant on $B_1$ for all $\omega\in \Omega''$ and defines a measurable function $\delta_n:\Omega''\to\R$. Then consider the measurable function $\psi_n:\Omega\to W^{1,p}(B_n)$ defined by $\psi_n(x,\omega):=\phi_n(x,\omega)-\delta_n(\omega)$. By construction, for all $m>n\ge1$, we have $\psi_n=\psi_{m}$ on $B_n$, so the $\psi_n$'s can be glued together and yield a measurable function $\psi:\Omega''\to W^{1,p}_{\loc}(\R^d)$ such that $\nabla\psi(\cdot,\omega)=f(\cdot,\omega)$ for all $\omega\in\Omega''$.
\end{proof}

\subsubsection{Sufficient conditions for Hypothesis~\ref{hypo:asmeasinf}}\label{app:infmes}
As will be shown below, the measurability Hypothesis~\ref{hypo:asmeasinf} is automatically satisfied if the integrand is quasiconvex and has the following nice approximation property, introduced by~\cite{AnzaHafsa-Mandallena-10}:

\begin{defin}\label{def:approxquasiconv}
A normal random integrand $W:\R^d\times\R^{m\times d}\times\Omega\to[0,\infty]$ is said to be \emph{quasiconvex} if $W(y,\cdot,\omega)$ is quasiconvex for almost all $y,\omega$. Given $p\ge1$, it is further said to be \emph{$p$-sup-quasiconvex} if there exists a sequence $(W_k)_k$ of quasiconvex normal random integrands such that $W_k(y,\Lambda,\omega)\uparrow W(y,\Lambda,\omega)$ pointwise as $k\uparrow\infty$, and such that, for all $k$, for almost all $y,\omega$ and for all $\Lambda,\Lambda'$,
\begin{align}\label{eq:pgrowthquasiconvapprox}
\frac1C|\Lambda|^p-C\le W_k(y,\Lambda,\omega)\le C_k(1+|\Lambda|^p),
\end{align}
for some constants $C,C_k>0$.\qed
\end{defin}

Note that Tartar~\cite{Tartar-93} has proven the existence of quasiconvex functions that are not $p$-sup-quasiconvex properties for any $p\ge1$. Before stating our measurability result, let us examine some important particular cases:

\begin{lem}\label{lem:psupquasiconv}
Let $W:\R^d\times\R^{m\times d}\times\Omega\to[0,\infty]$ be a normal random integrand. Assume that there exist $C>0$ and $p>1$ such that, for almost all $\omega,y$ and for all $\Lambda$,
\begin{align}\label{eq:lowerboundcoercselmes}
\frac1C|\Lambda|^p-C\le W(y,\Lambda,\omega).
\end{align}
Also assume that one of the following holds:
\begin{enumerate}[(1)]
\item $W=V$ is a convex normal random integrand, and the convex function $M:=\supess_{y,\omega}V(y,\cdot,\omega)$ has $0$ in the interior of its domain;
\item $W$ satisfies Hypothesis~\ref{hypo:model-nc}, with a convex part $V$ such that $M:=\supess_{y,\omega}V(y,\cdot,\omega)$ has $0$ in the interior of its domain.
\end{enumerate}
Then, $W$ is a $p$-sup-quasiconvex normal random integrand.\qed
\end{lem}

\begin{proof}
Case~(2) directly follows from the approximation result of case~(1) applied to the convex part $V$. We may thus focus on case~(1). Let $W=V$ be a convex normal random integrand. For all $k\ge0$, consider the following Yosida transform:
\begin{align}\label{eq:defyosida}
V_k(y,\Lambda,\omega)=\inf_{\Lambda'\in\R^{m\times d}}\left(V(y,\Lambda',\omega)+k|\Lambda-\Lambda'|^p\right).
\end{align}
For almost all $y,\omega$, convexity of $V_k(y,\cdot,\omega)$ easily follows from convexity of $V(y,\cdot,\omega)$. For almost all $y,\omega$, the lower semicontinuity of $V(y,\cdot,\omega)$ ensures that $V_k(y,\cdot,\omega)\uparrow V(y,\cdot,\omega)$ pointwise as $k\uparrow\infty$. Moreover, by definition, $V_k(y,\Lambda,\omega)\le M(0)+k|\Lambda|^p$, while the lower bound~\eqref{eq:lowerboundcoercselmes} implies $V_k(y,\Lambda,\omega)\ge \frac1C|\Lambda|^p-C$.

It remains to check that the $V_k$'s are normal random integrands. For almost all $y,\omega$, the function $V(y,\cdot,\omega)$ is convex and lower semicontinuous, hence it is continuous on its domain $D_{y,\omega}$ (not only on the interior). As by assumption $0\in\inter D_{y,\omega}$, the set $D_{y,\omega}$ is a convex subset of maximal dimension, and hence points with rational coordinates are dense in $D_{y,\omega}$. The infimum~\eqref{eq:defyosida} defining $V_k$ may thus be restricted to $\Q^{m\times d}$. As a countable infimum, the required measurability properties follow.
\end{proof}

We now turn to the validity of the measurability Hypothesis~\ref{hypo:asmeasinf} for $p$-sup-quasiconvex integrands.

\begin{prop}\label{prop:infmes}
Let $O\subset\R^d$ be a bounded domain, let $(\Omega,\F,\p)$ be a complete probability space, and let $W:\R^d\times\R^{m\times d}\times\Omega\to[0,\infty]$ be a $p$-sup-quasiconvex normal random integrand for some $p>1$ (in the sense of Definition~\ref{def:approxquasiconv}).
Given some fixed function $f\in \Ld^p(\Omega;\Ld^p(O;\R^{m\times d}))$, consider the random integral functional $I:W^{1,p}(O;\R^m)\times\Omega\to[0,\infty]$ defined by
\begin{align}\label{eq:defintInnconvap}
I(u,\omega)=\int_OW(y,f(y,\omega)+\nabla u(y),\omega)dy.
\end{align}
Then, $I$ is weakly lower semicontinuous on $W^{1,p}(O;\R^m)$. Moreover, for all weakly closed subsets $F\subset W^{1,p}_0(O;\R^m)$ or $F\subset \{u\in W^{1,p}(O;\R^m):\int_Ou=0\}$, the function $\omega\mapsto\inf_{v\in F}I(v,\omega)$ is $\F$-measurable. In particular, Hypothesis~\ref{hypo:asmeasinf} is satisfied.\qed
\end{prop}

\begin{proof}
For all $k$, define the approximated random functional $I_k:W^{1,p}(O;\R^m)\times\Omega\to[0,\infty]$ by
\[I_k(u,\omega)=\int_OW_k(y,f(y,\omega)+\nabla u(y),\omega)dy.\]
As the $W_k$'s are nonnegative, monotone convergence yields that $I_k\uparrow I$ pointwise. Moreover, for all $k$, and almost all $\omega$, the quasiconvexity and the upper bound~\eqref{eq:pgrowthquasiconvapprox} satisfied by $W_k(\cdot,\cdot,\omega)$ imply the weak lower semicontinuity of $I_k(\cdot,\omega)$ on $W^{1,p}(O;\R^m)$ (see~\cite{Acerbi-Fusco-84}).
As a pointwise supremum of weakly lower semicontinuous functions, we deduce that $I(\cdot,\omega)$ is itself weakly lower semicontinuous on $W^{1,p}(O;\R^m)$.

Combining the weak lower semicontinuity of $I_k$ with the uniform coercivity assumption (cf. lower bound in~\eqref{eq:pgrowthquasiconvapprox}) and with Poincaré's inequality, we easily conclude, for any weakly closed subset $F\subset W^{1,p}_0(O;\R^m)$ or $F\subset \{u\in W^{1,p}(O;\R^m):\int_Ou=0\}$,
\begin{align}\label{eq:limjkmes-1}
\lim_{k\uparrow\infty}\inf_{v\in F}I_k(v,\omega)=\inf_{v\in F} I(v,\omega).
\end{align}
For all $k$, as $W_k$ is quasiconvex hence rank-1 convex in its second variable (see~\cite{Ball-76}), the $p$-growth condition~\eqref{eq:pgrowthquasiconvapprox} implies the following local Lipschitz condition: for almost all $y,\omega$, for all $\Lambda,\Lambda'$,
\[|W_k(y,\Lambda,\omega)-W_k(y,\Lambda',\omega)|\le C_k|\Lambda-\Lambda'|(1+|\Lambda|^{p-1}+|\Lambda'|^{p-1}),\]
for some constant $C_k>0$. The Hölder inequality then gives, for all $u,v$, for almost all $\omega$,
\begin{align*}
|I_k(u,\omega)-I_k(v,\omega)|
\le C_kC\|\nabla(u-v)\|_{\Ld^p(O)}(1+\|\nabla u\|_{\Ld^p(O)}^{p-1}+\|\nabla v\|_{\Ld^p(O)}^{p-1}+\|f(\cdot,\omega)\|_{\Ld^p(O)}^{p-1}).
\end{align*}
This proves that the map $I_k(\cdot,\omega)$ is strongly continuous on $W^{1,p}(O;\R^m)$, for almost all $\omega$. Given a weakly (hence strongly) closed subset $F\subset W^{1,p}_0(O;\R^m)$ or $F\subset \{u\in W^{1,p}(O;\R^m):\int_Ou=0\}$, strong separability of $W^{1,p}(O;\R^m)$ implies strong separability of $F$, so there exists a countable strongly dense subset $F_0\subset F$. Therefore, the map
\[\omega\mapsto\inf_{v\in F}I_k(v,\omega)=\inf_{v\in F_0}I_k(v,\omega)\]
is $\F$-measurable, and the conclusion follows from~\eqref{eq:limjkmes-1}.
\end{proof}

\subsubsection{Measurable minimizers}
We show that in the convex case (or more generally in the $p$-sup-quasiconvex case) the random functional~\eqref{eq:defintInnconvap} admits a measurable minimizer. For that purpose, we begin with the following useful reformulation of the Rokhlin--Kuratowski--Ryll Nardzewski theorem, which essentially asserts that the measurability of the infimum implies the measurability of minimizers.

\begin{lem}\label{lem:selmes}
Let $X$ be a Polish space, let $(\Omega,\F,\mu)$ be a complete measure space, and let $I:X\times \Omega\to[0,\infty]$. Assume that
\begin{enumerate}[(i)]
\item for all $\omega$, $I(\cdot,\omega)$ is lower semicontinuous on $X$;
\item for all $\omega$, $I(\cdot,\omega)$ is coercive on $X$ (i.e. the sublevel sets $\{u\in X:I(u,\omega)\le c\}$ are compact for all $c>0$);
\item for all closed subset $F\subset X$, the map $\phi_F:\Omega\to[0,\infty]$ defined by $\phi_F(\omega):=\min_{v\in F}I(v,\omega)$ is $\F$-measurable.
\end{enumerate}
Then, there exists an $\F$-measurable map $u:\Omega\to X$ such that, for all $\omega\in\Omega$,
\[I(u(\omega),\omega)=\min_{v\in X}I(v,\omega)=\phi_X(\omega).\]\qed
\end{lem}

\begin{proof}
By coercivity and lower semicontinuity, the minima of $I(\cdot,\omega)$ are always attained on all closed subsets $F$, so that the function $\phi_F$ is always well-defined.

Consider the multifunction $\Gamma:\Omega\rightrightarrows X$ defined by $\Gamma(\omega)=\{u\in X\,:\,I(u,\omega)=\phi_X(\omega)\}$. By lower semicontinuity, $\Gamma(\omega)\subset X$ is nonempty and closed for all $\omega$. Moreover, we claim that $\Gamma$ is measurable, in the sense that
\[\Gamma^{-1}(O):=\{\omega\in\Omega\,:\,\Gamma(\omega)\cap O\ne\varnothing\}\]
belongs to $\F$ for all open subsets $O\subset X$. As $X$ is metrizable, it actually suffices to check $\Gamma^{-1}(F)\in \F$ for all closed subsets $F\subset X$ (see e.g. \cite[Lemma~18.2]{AliprantisBorder}). By the coercivity and the lower semicontinuity of $I$, we may write
\begin{align*}
\Gamma^{-1}(F)&=\{\omega\in\Omega\,:\,\exists u\in F,\,I(u,\omega)=\phi_X(\omega)\}\\
&=\bigcap_{n=1}^\infty\left\{\omega\in\Omega\,:\,\exists u\in F,\,I(u,\omega)\le\phi_X(\omega)+\frac1n\right\}\\
&=\bigcap_{n=1}^\infty\left\{\omega\in\Omega\,:\,\phi_F(\omega)\le\phi_X(\omega)+\frac1n\right\},
\end{align*}
where the right-hand side belongs to $\F$, by measurability of $\phi_X$ and $\phi_F$. Hence, $\Gamma$ is measurable, and we may thus apply the Rokhlin--Kuratowski--Ryll Nardzewski measurable selection theorem (see e.g.~\cite[Theorem~18.13]{AliprantisBorder}), which states the existence of a $\F$-measurable map $u:\Omega\to X$ such that $u(\omega)\in\Gamma(\omega)$ for all $\omega$.
\end{proof}

Combining this measurable selection lemma with the measurability result of Proposition~\ref{prop:infmes}, we obtain the following:

\begin{prop}\label{prop:selmes}
Let $O$ be some bounded domain, let $(\Omega,\F,\p)$ be a complete probability space, and let $W:\R^d\times\R^{m\times d}\times\Omega\to[0,\infty]$ be a normal random integrand. Assume that there exists $C>0$ and $p>1$ such that, for almost all $y,\omega$ and for all $\Lambda$,
\begin{align}\label{eq:lowerboundselmes}
\frac1C|\Lambda|^p-C\le W(y,\Lambda,\omega).
\end{align}
Also assume that $W$ satisfies Hypothesis~\ref{hypo:asmeasinf} and that $I(\cdot,\omega)$ is weakly lower semicontinuous on $W^{1,p}(O;\R^m)$ for almost all $\omega$ (in particular this is the case if $W$ is convex or $p$-sup-quasiconvex in the sense of Definition~\ref{def:approxquasiconv}). Given some fixed function $f\in\Ld^p(\Omega;\Ld^p(O;\R^{m\times d}))$, consider the random integral functional $I:W^{1,p}(O;\R^m)\times\Omega\to[0,\infty]$ defined by
\[I(u,\omega)=\int_OW(y,f(y,\omega)+\nabla u(y),\omega)dy.\]
Then, for all nonempty weakly closed subsets $F\subset W^{1,p}_0(O;\R^m)$ or $F\subset \{u\in W^{1,p}(O;\R^m):\int_Ou=0\}$, there exists a $\F$-measurable map $u:\Omega\to F$ such that, for almost all $\omega$,
\[I(u(\omega),\omega)=\inf_{v\in F}I(v,\omega).\]\qed
\end{prop}

\begin{proof}
Let $X$ denote the Banach space $W^{1,p}_0(O;\R^m)$ or $\{u\in W^{1,p}(O;\R^m):\int_Ou=0\}$, endowed with the weak topology, and, for all $k\ge1$, consider the subset $X_k:=\{u\in X:\|\nabla u\|_{\Ld^p(O)}\le k\}$, endowed with the induced weak topology. By Poincaré's inequality and by the Banach-Alaoglu theorem, $X_k$ is easily seen to be metrizable and compact, hence Polish. Let $F\subset X$ be a nonempty (weakly) closed subset.

Let $\Omega'\subset\Omega$ denote a subset of maximal probability such that $I(\cdot,\omega)$ is weakly lower semicontinuous on $W^{1,p}(O;\R^m)$ for all $\omega\in\Omega'$. Since $X_k\subset X$ is (weakly) closed, the intersection $X_k\cap F$ is also (weakly) closed, and hence Hypothesis~\ref{hypo:asmeasinf} asserts that the map $\omega\to\inf_{v\in F\cap X_k}I(v,\omega)$ is $\F$-measurable. Applying Lemma~\ref{lem:selmes} on the compact Polish space $X_k$ and on $\Omega'$ then yields a $\F$-measurable map $u_k:\Omega'\to X_k$ such that, for all $\omega\in\Omega'$,
\[I(u_k(\omega),\omega)=\min_{v\in F\cap X_k}I(v,\omega).\]
The lower bound~\eqref{eq:lowerboundselmes} and the triangle inequality give
\[\int_O|\nabla u|^p\le C^22^{p-1}+C2^{p-1}I(u,\omega)+2^{p-1}\int_O|f(\cdot,\omega)|^p.\]
Define for all $k\ge1$,
\[\Omega_k:=\left\{ \omega\in\Omega': C^22^{p-1}+C2^{p-1}\inf_{v\in F} I(v,\omega)+2^{p-1}\int_O|f(\cdot,\omega)|^p\le k\right\},\]
and note that $\Omega_k\in\F$ by Hypothesis~\ref{hypo:asmeasinf}. By definition, for all $\omega\in\Omega_k$ we have
\[I(u_k(\omega),\omega)=\min_{v\in F\cap X_k}I(v,\omega)=\min_{v\in F}I(v,\omega).\]
The sequence $(\Omega_k)_k$ is increasing, $\Omega_k\uparrow\Omega'':=\bigcup_{k=1}^\infty\Omega_k$. By integrability of $f$, $\Omega''\subset\Omega'''$ and $\p[\Omega'''\setminus\Omega'']=0$, where we have defined the event $\Omega''':=\{\omega\in\Omega:\inf_{v\in F}I(v,\omega)<\infty\}$. Given some fixed $w\in F$, the measurable map $u:\Omega\to X$ defined by
\[u(\omega):=\mathds1_{\Omega\setminus \Omega''}w+\mathds1_{\Omega_1}u_1(\omega)+\sum_{k=2}^\infty\mathds1_{\Omega_k\setminus\Omega_{k-1}}u_k(\omega)\]
satisfies by definition $I(u(\omega),\omega)=\inf_{v\in F}I(v,\omega)$ for all $\omega\in\Omega''\cup(\Omega\setminus\Omega''')$.
\end{proof}

\subsection{Approximation results}\label{append:approx}
In the present appendix, we prove two general approximation results that are crucially needed in this paper. The first one is an extension of~\cite[Lemma~3.6]{Muller-87} and~\cite[Proposition~2.6 of Chapter~X]{Ekeland76}.

\begin{prop}\label{prop:approxw}
Let $O\subset \R^d$ be a bounded Lipschitz domain, which is also strongly star-shaped, in the sense that there exists $x_0\in O$ such that
\[\overline{-x_0+O}\subset \alpha(-x_0+O),\quad\text{for all $\alpha>1$}.\]
Let $\Theta:\R^{m\times d}\to[0,\infty]$ be a convex lower semicontinuous function with $0\in\inter\dom\Theta$, and let $u\in W^{1,1}(O;\R^m)$ such that $\int_O\Theta(\nabla u)<\infty$. Then,
\begin{enumerate}[(i)]
\item there is a sequence $(v_n)_n\subset C^\infty(\adh O;\R^m)$ such that $\nabla v_n\in\inter\dom\Theta$ pointwise,
\[v_n\to u\quad\text{in $W^{1,1}(O;\R^m)$},\qquad\text{and}\qquad\int_O\Theta(\nabla v_n(y))dy\to\int_O\Theta(\nabla u(y))dy;\]
\item there is a sequence $(w_n)_n$ of (continuous) piecewise affine functions such that $\nabla w_n\in\Q^{m\times d}\cap\inter\dom\Theta$ pointwise,
\[w_n\to u\quad\text{in $W^{1,1}(O;\R^m)$},\qquad\text{and}\qquad\int_U\Theta(\nabla w_n(y))dy\to\int_O\Theta(\nabla u(y))dy.\]
\end{enumerate}
If in addition $u$ belongs to $W^{1,p}(O;\R^m)$ for some $1\le p<\infty$, then the sequences $(v_n)_n$ and $(w_n)_n$ can be chosen such that $v_n\to u$ and $w_n\to u$ in $W^{1,p}(O;\R^m)$. Moreover,
\begin{enumerate}[(a)]
\item if $u$ belongs to $W^{1,1}_0(O;\R^m)$, then we can choose $v_n\in C_c^\infty(O;\R^m)$ and $w_n|_{\partial O}=0$ (and in that case the assumption that $O$ be strongly star-shaped can be relaxed);
\item if $O=Q$ and $u\in W^{1,1}_{\per}(Q;\R^m)$, then $v_n$ and $w_n$ can be both chosen to be $Q$-periodic;
\item if $\Xi:\R^{m\times d}\to[0,\infty]$ is a (nonconvex) ru-usc lower semicontinuous function which is continuous on $\inter\dom\Theta$ and satisfies $0\le\Xi\le\Theta$ pointwise, then the sequences $(v_n)_n$ and $(w_n)_n$ can be chosen in such a way that $\int_O\Xi(\nabla v_n)\to \int_O\Xi(\nabla u)$ and $\int_O\Xi(\nabla w_n)\to \int_O\Xi(\nabla u)$.\qed
\end{enumerate}
\end{prop}

\begin{proof}
We divide the proof into four steps.

\medskip
\step1 Proof of~(i).

First, we show that we can assume $\nabla u\in \inter\dom\Theta$ almost everywhere. Indeed, assume that~(i) is proven for such $u$'s, and let us deduce the general case. Given $u\in W^{1,p}(O;\R^m)$ with $\int_O\Theta(\nabla u)<\infty$, we have $\nabla u\in\dom\Theta$ almost everywhere, and hence by convexity $t\nabla u\in\dom\Theta$ almost everywhere for all $t\in[0,1)$. As convexity also implies $\int_O\Theta(t\nabla u)\le t\int_O\Theta(\nabla u)+(1-t)\Theta(0)<\infty$, we can apply the result~(i) to $tu$, for any $t\in[0,1)$: this gives a sequence $(v_{n,t})_n\subset C^\infty(\adh O;\R^m)$ such that $v_{n,t}\to tu$ in $W^{1,1}(O;\R^m)$ and $\int_O\Theta(\nabla v_{n,t}(y))dy\to\int_O\Theta(t\nabla u(y))dy$. Weak lower semicontinuity of the integral functional $u\mapsto \int_O\Theta(\nabla u)$ on $W^{1,1}(O;\R^m)$ (which follows from convexity and lower semicontinuity of $\Theta$) implies $\liminf_{t\uparrow1}\int_O\Theta(t\nabla u)\ge\int_O\Theta(\nabla u)$. As the converse inequality follows from convexity, we obtain
\begin{align*}
&\limsup_{t\uparrow1}\limsup_{n\uparrow\infty}\left(\|v_{n,t}-u\|_{W^{1,1}(O)}+\left|\int_O\Theta(\nabla v_{n,t})- \int_O\Theta(\nabla u)\right|\right)\\
=~&\limsup_{t\uparrow1}\left(\|tu-u\|_{W^{1,1}(O)}+\left|\int_O\Theta(t\nabla u)- \int_O\Theta(\nabla u)\right|\right)=0,
\end{align*}
and hence a standard diagonalization argument gives a sequence $(v_n)_n\subset C^\infty(\adh O;\R^m)$ such that $\nabla v_n\in \inter\dom\Theta$ almost everywhere, $v_n\to u$ in $W^{1,1}(O;\R^m)$, and $\int_O\Theta(\nabla v_n)\to\int_O\Theta(\nabla u)$, and proves the general version of~(i).

Hence, from now on, we assume $\nabla u\in \inter\dom\Theta$ almost everywhere. Moreover, without loss of generality, we may also assume that $O$ is strongly star-shaped with respect to $x_0=0$. Choose $(\alpha_k)_k\subset(1,\infty)$ a decreasing sequence of positive numbers converging to $1$, sufficiently slowly so that $\frac1k<\frac12\dist(\adh O,\partial(\alpha_k O))$ for all $k$, and define
\[u_k:\alpha_kO\to\R^m:x\mapsto u_k(x)=u(x/\alpha_k).\]
Take $\rho\in C_c^\infty(\R^d)$ such that $\int\rho=1$, $\rho\ge0$ and $\supp\rho\subset B(0,1)$, and write $\rho_k(x)=k^{-d}\rho(kx)$, for all $k\ge1$. Consider the sequence $(v_k)_k$ defined by $v_k=\rho_k\ast (\alpha_ku_k)$ (which is well-defined by virtue of the condition on the $\alpha_k$'s). Note that $v_k\in C^\infty(\adh O;\R^m)$ and $\nabla v_k=\rho_k\ast(\nabla u)_k$, where we use the notation $(\nabla u)_k(x)=\nabla u(x/\alpha_k)$. Moreover, we then observe $v_k\to u$ in $W^{1,1}(O;\R^m)$ since $u_k\to u$ et $(\nabla u)_k\to\nabla u$ in $\Ld^1(O;\R^m)$. Now, Jensen's inequality yields
\[0\le\Theta(\nabla v_k)=\Theta(\rho_k\ast(\nabla u)_k)\le\rho_k\ast(\Theta((\nabla u)_k))=\rho_k\ast(\Theta(\nabla u))_k.\]
As the sequence $(\rho_k\ast(\Theta(\nabla u))_k)_k$ converges to $\Theta\circ\nabla u$ in $\Ld^1(O;\R^m)$, it is uniformly integrable, and the same thus holds for the sequence $(\Theta(\nabla v_k))_k$. As $\nabla v_k\to\nabla u$ in $\Ld^1(O;\R^m)$, the convergence holds almost everywhere up to an extraction. Since by convexity $\Theta$ is continuous on $\inter\dom\Theta$, since $\nabla u\in \inter\dom\Theta$ almost everywhere, and since $\nabla v_k\to\nabla u$ almost everywhere up to an extraction, we deduce $\Theta(\nabla v_{k})\to \Theta(\nabla u)$ almost everywhere up to an extraction. By uniform integrability the latter convergence also holds in $\Ld^1(O;\R^m)$, which allows us to get rid of the extraction. In particular,
\[\int_O\Theta(\nabla v_k)\to\int_O\Theta(\nabla u).\]
Moreover, $(\nabla u)_k\in \inter\dom\Theta$ almost everywhere, and thus $\nabla v_k=\rho_k\ast(\nabla u)_k\in \inter\dom\Theta$ everywhere for all $k$, since $\inter\dom\Theta$ is a convex set containing $0$. This proves part (i).

Let us now assume that $u\in W^{1,p}(O;\R^m)$ for some $1\le p<\infty$, and consider the sequences $(u_k)_k$, $((\nabla u)_k)_k$ and $(v_k)_k$ defined above. First, Lebesgue's dominated convergence theorem implies $u_k\to u$ and $(\nabla u)_k\to\nabla u$ in $\Ld^p(O;\R^m)$. Further, since the Jensen inequality gives
\begin{align*}
\int_O|v_k-u|^p&=\int_O|\alpha_k\rho_k\ast u_k-u|^p\le\int_O\int_{B_{1/k}}\rho_k(t)|\alpha_ku_k(x-t)-u(x)|^pdt\,dx,
\end{align*}
and likewise for gradients, we conclude that the sequence $(v_k)_k$ converges to $u$ in $W^{1,p}(O;\R^m)$. This proves part (i) in the case when $u\in W^{1,p}(O;\R^m)$.

\medskip
\step2 Proof of~(ii).

For all $k$, since $v_k\in C^\infty(\adh O;\R^m)$, there exists a sequence $(w_{k,j})_j$ of piecewise affine functions such that $w_{k,j}\to v_k$ in $W^{1,\infty}(O;\R^m)$ as $j\uparrow\infty$ and $\|\nabla w_{k,j}\|_{\Ld^\infty}\le\|\nabla v_{k}\|_{\Ld^\infty}$ for all $j,k$ (see e.g.~\cite[Proposition~2.1 of Chapter~X]{Ekeland76}). Further, these functions $w_{k,j}$ can (simply remember their construction by triangulation) be chosen taking their values in $\inter\dom\Theta$, since we have constructed $\nabla v_k\in \inter\dom\Theta$ everywhere for all $k$. Another approximation argument further allows us to choose $w_{k,j}$ such that $\nabla w_{k,j}$ only takes rational values. The desired result then follows from Step~1 and a diagonalization argument. Finally, the particular case when $u$ belongs to $W^{1,p}(O;\R^m)$ is obtained similarly as in Step~1.

\medskip
\step3 Proof of the additional statements.

It remains to address the particular cases~(a) and~(b). First assume that $u$ belongs to $W^{1,1}_0(O;\R^m)$. For the corresponding result, we refer to~\cite[Proposition~2.6 of Chapter~X]{Ekeland76}. The only difference is that the argument in~\cite{Ekeland76} requires continuity of $\Theta$. Instead, we replace $u$ by $tu$ for $t<1$ as in Step~1, so that by convexity $t\nabla u\in\inter\dom\Theta$ almost everywhere, hence all the constructed quantities have almost all their values in $\inter\dom\Theta$, on which $\Theta$ is continuous by convexity. No further continuity assumption is then needed.

Finally, if we assume $O=Q$ with $u\in W^{1,1}_\per(Q;\R^m)$, then we can consider the periodic extension of $u$ on $\R^d$ and repeat the arguments in such a way that periodicity is conserved.

\medskip
\step4 Proof in the nonconvex case.

Let $u\in W^{1,p}(O;\R^m)$ be such that $\int_U\Theta(\nabla u)<\infty$, which implies $\nabla u\in\dom\Theta$ almost everywhere. Let $t\in (0,1)$. The approximation result given by point~(i) gives a sequence  $(u_{n,t})_n$ of smooth functions such that $u_{n,t}\to tu$ (strongly) in $W^{1,p}(O;\R^m)$ and $\int_O\Theta(\nabla u_{n,t})\to \int_O\Theta(t\nabla u)$ as $n\uparrow\infty$, and such that $\nabla u_{n,t}\in\inter\dom\Theta$ pointwise. Up to an extraction, we have $\nabla u_{n,t}\to t\nabla u$ almost everywhere, and thus $\Xi(\nabla u_{n,t})\to\Xi(t\nabla u)$ almost everywhere, which follows from continuity of $\Xi$ on the interior of its domain, with indeed $t\nabla u\in\inter\dom\Xi$ almost everywhere. Then noting that
\[0\le\Xi(\nabla u_{n,t})\le C(1+\Theta(\nabla u_{n,t})),\]
and invoking both Lebesgue's dominated convergence theorem (for $\Xi(\nabla u_{n,t})$) and its converse (for $\Theta(\nabla u_{n,t})$), we deduce convergence $\int_O\Xi(\nabla u_{n,t})\to\int_O\Xi(t\nabla u)$ as $n\uparrow\infty$. As $\Xi$ is lower semi-continuous and also ru-usc, we compute, by Fatou's lemma,
\begin{align*}
\int_O\Xi(\nabla u(y))dy\le \int_O\liminf_{t\uparrow1}\Xi(t\nabla u(y))dy&\le \liminf_{t\uparrow1}\int_O\Xi(t\nabla u(y))dy\\
&\le \limsup_{t\uparrow1}\int_O\Xi(t\nabla u(y))dy\le \int_O\Xi(\nabla u(y))dy.
\end{align*}
Hence,
\[\lim_{t\uparrow1}\lim_{n\uparrow\infty}\int_O\Xi(\nabla u_{n,t})= \lim_{t\uparrow1}\int_O\Xi(t \nabla u)=\int_O\Xi(\nabla u),\]
and similarly
\[\lim_{t\uparrow1}\lim_{n\uparrow\infty}\int_O\Theta(\nabla u_{n,t})= \int_O\Theta(\nabla u),\]
so that the conclusion follows from a standard diagonalization argument. The other properties are deduced in a similar way.
\end{proof}

For technical reasons, we need in the proof of Proposition~\ref{prop:gammasupN} to further approximate piecewise affine functions by refined ones with smoother variations. The precise approximation result we need is the following.

\begin{prop}\label{prop:approxaffine}
Let $u$ be an $\R^m$-valued continuous piecewise affine function on a bounded Lipschitz domain $O\subset\R^d$. Consider the open partition $O=\biguplus_{l=1}^kO^l$ associated with $u$ (i.e. $u$ is affine on each piece $O^l$). Define $M:=(\bigcup_{l=1}^k\partial O^l)\setminus\partial O$ the interior boundary of this partition of $O$, and, for fixed $r>0$, also define $M_r:=(M+B_r)\cap O$ the $r$-neighborhood of this interior boundary. Then, for all $\kappa>0$, there exists a continuous piecewise affine function $u_{\kappa,r}$ on $O$ with the following properties:
\begin{enumerate}[(i)]
\item $\nabla u_{\kappa,r}=\nabla u$ pointwise on $O\setminus M_r$, and $\limsup_{r\downarrow0}\sup_{0<\kappa\le1}\|u_{\kappa,r}-u\|_{\Ld^\infty(O)}=0$;
\item $\nabla u_{\kappa,r}\in \conv(\{\nabla u(x):x\in O\})$ pointwise (where $\conv(\cdot)$ denotes the convex hull);
\item denoting by $O:=\biguplus_{l=1}^{n_{\kappa,r}}O_{\kappa,r}^l$ the open partition associated with $u_{\kappa,r}$, and $\Lambda^l_{\kappa,r}:= \nabla u_{\kappa,r}|_{O^l_{\kappa,r}}$ for all~$l$, we have $|\Lambda^i_{\kappa,r}-\Lambda^j_{\kappa,r}|\le \kappa$ for all $i,j$ with $\partial O^i_{\kappa,r}\cap\partial O^j_{\kappa,r}\ne\varnothing$.\qed
\end{enumerate}
\end{prop}

\begin{proof}
Let $u$, $O$ and $r$ be fixed. Without loss of generality, we can assume $0\in O$. Denote $r_0:=\dist(0,\partial O)$ and $R_0:=\max_{x\in\partial O}\dist(0,x)$, and define $\alpha_r>1$ by $(\alpha_r-1)R_0=r/2$. Choose a nonnegative smooth function $\rho_r$ supported in $B_{(\alpha_r-1)r_0}$ with $\int\rho_r=1$, and consider the smooth function $u_r$ on $O$ defined by $u_r=\rho_r\ast[\alpha_r u(\cdot/\alpha_r)]$. Since by definition inequality $|\alpha_rx-x|\le(\alpha_r-1)R_0= r/2$ holds for any $x\in O$, we easily check $\nabla u_r=\nabla u$ on the set $O\setminus M_r$. As $u_r$ is smooth, we can consider $L_r:=\max_{x\in \adh O}|\nabla\nabla u_r(x)|<\infty$. Choose a triangulation $(O_{\kappa,r}^l)_{l=1}^{n_{\kappa,r}}$ of $O$ which is a refinement of the partition $\{O^l\setminus M_r:l=1,\ldots,k\}\cup\{M_r\}$, such that the diameter of each of the $O_{\kappa,r}^l$'s is at most $\kappa/(2L_r)$. Now construct as usual the piecewise affine approximation $u_{\kappa,r}$ of $u_r$ with respect to the triangulation $(O_{\kappa,r}^l)_l$ (see e.g. \cite[Proposition~2.1 of Chapter~X]{Ekeland76}). By construction, $\nabla u_{\kappa,r}=\nabla u_r=\nabla u$ on $O\setminus M_r$, since $u_r$ is affine on each connected component of $O\setminus M_r$. Also note that the construction ensures that $\nabla u_{\kappa,r}$ belongs pointwise to the set $\{\nabla u_r(x):x\in O\}$ (as a consequence of the mean value theorem), which is by definition included in the convex hull $\conv(\{\nabla u(x):x\in O\})$. Moreover, if $\partial O^i_{\kappa,r}\cap \partial O^j_{\kappa,r}\ne\varnothing$, it implies that $O^i_{\kappa,r}\cup O^j_{\kappa,r}$ has diameter bounded by $\kappa/L_r$. The construction of $u_{\kappa,r}$ then gives
\[\left|\nabla u_{\kappa,r}|_{O^i_{\kappa,r}}-\nabla u_{\kappa,r}|_{O^j_{\kappa,r}}\right|\le\sup_{x\in O^i_{\kappa,r}}\sup_{y\in O^j_{\kappa,r}}|\nabla u_r(x)-\nabla u_r(y)|\le \frac{\kappa}{L_r}\sup_{x\in O}|\nabla\nabla u_r(x)|\le \kappa,\]
which proves property~(iii). Finally, the last property of~(i) directly follows from the construction.
\end{proof}

To treat the case of periodic boundary conditions, we further need the following periodic version of the previous result. The proof is omitted, since it is a direct adaptation on the torus of the proof above.

\begin{prop}\label{prop:approxaffineper}
Let $u$ be an $\R^m$-valued continuous piecewise affine $Q$-periodic function on $\R^d$. Consider the periodic partition $\R^d=\biguplus_{l=1}^\infty O^l$ associated with $u$ (i.e. $u$ is affine on each piece $O^l$). Define $M:=\bigcup_{l=1}^\infty\partial O^l$ the boundary of this partition, and, for fixed $r>0$, also define $M_r:=(M+B_r)\cap Q$ the $r$-neighborhood of this boundary. Then, for all $\kappa>0$, there exists a continuous piecewise affine $Q$-periodic function $u_{\kappa,r}$ on $\R^d$ with the following properties:
\begin{enumerate}[(i)]
\item $\nabla u_{\kappa,r}=\nabla u$ pointwise on $\R^d\setminus M_r$, and $\limsup_{r\downarrow0}\sup_{0<\kappa\le1}\|u_{\kappa,r}-u\|_{\Ld^\infty(Q)}=0$;
\item $\nabla u_{\kappa,r}\in \conv(\{\nabla u(x):x\in Q\})$ pointwise (where $\conv(\cdot)$ denotes the convex hull);
\item denoting by $\R^d:=\biguplus_{l=1}^{\infty}T_{\kappa,r}^l$ the $Q$-periodic partition associated with $u_{\kappa,r}$, and $\Lambda^l_{\kappa,r}:= \nabla u_{\kappa,r}|_{T^l_{\kappa,r}}$ for all~$l$, we have $|\Lambda^i_{\kappa,r}-\Lambda^j_{\kappa,r}|\le \kappa$ for all $i,j$ with $\partial T^i_{\kappa,r}\cap\partial T^j_{\kappa,r}\ne\varnothing$.\qed
\end{enumerate}
\end{prop}

\bibliographystyle{plain}
\bibliography{biblio}

\end{document}